\documentclass[final]{siamltex}
\usepackage{amsfonts,amsmath,amssymb}
\usepackage{mathrsfs,mathtools}
\usepackage{graphicx}
\usepackage{enumerate}
\usepackage{xspace}
\usepackage{hyperref}

\usepackage{color}

\newcommand{\sr}{{\tfrac12}}
\newcommand{\stwo}{{\tfrac{s}2}}
\newcommand{\srn}{{\tfrac12}}

\newcommand{\ve}{\mathfrak{u}}
\newcommand{\R}{\mathbb{R}}
\newcommand{\F}{\mathcal{F}}
\newcommand{\Y}{\mathpzc{Y}}

\newcommand{\Fm}{\mathpzc{F}}

\newcommand{\N}{\mathpzc{N}}

\newcommand{\vero}{\texttt{v}}

\newcommand{\C}{\mathcal{C}}

\newcommand{\D}{\mathcal{D}}
\newcommand{\V}{\mathbb{V}}

\newcommand{\T}{\mathscr{T}}

\newcommand{\Q}{\mathbb{Q}}

\newcommand{\Ss}{\mathscr{S}}
\newcommand{\Tm}{\mathcal{T}}

\newcommand{\HL}{ \mbox{ \raisebox{7.4pt} {\tiny$\circ$} \kern-10.3pt} {H_L^1} }
\newcommand{\Wp}{ \mbox{ \raisebox{7.7pt} {\scriptsize$\circ$} \kern-10.1pt} {W^{1,p}} }
\newcommand{\Wpp}{ \mbox{ \raisebox{7.7pt} {\scriptsize$\circ$} \kern-10.1pt} {W^{1,p'}} }
\newcommand{\Sz}{ \mbox{ \raisebox{7.5pt} {\scriptsize$\circ$} \kern-10.1pt} {\Ss} }
\newcommand{\HLnew}{ \mbox{ \raisebox{7pt} {\scriptsize$\circ$} \kern-10.1pt}{H}^1_L }
\newcommand{\HLn}{{\mbox{\,\raisebox{5.1pt} {\tiny$\circ$} \kern-9.1pt}{H}^1_L  }}
\newcommand{\HLs}{{\mbox{\raisebox{8.7pt} {\scriptsize$\circ$} \kern-10.1pt}{H}^1_L  }}
\newtheorem{remark}[theorem]{Remark}
\newcommand{\verot}{\emph{\texttt{v}}}
\newcommand{\Tr}{\mathbb{T}}
\newcommand{\U}{\mathbb{U}}
\newcommand{\A}{\mathcal{A}}
\newcommand{\B}{\mathcal{B}}
\newcommand{\xa}{\xi_{\alpha}}

\DeclareMathOperator*{\tr}{tr_\Omega}
\DeclareMathOperator{\dom}{Dom}
\DeclareMathOperator*{\diam}{diam}

\newcommand{\Laps}{(-\Delta)^s}

\newcommand{\DIV}{\textrm{div}}
\newcommand{\diff}{\, \mbox{\rm d}}
\newcommand{\ie}{i.e.,\@\xspace}
\newcommand{\cf}{cf.\@\xspace}
\newcommand{\Hs}{\mathbb{H}^s(\Omega)}
\newcommand{\Ws}{\mathbb{H}^{1-s}(\Omega)}
\newcommand{\GL}{{\textup{\textsf{GL}}}}
\DeclareMathAlphabet{\mathpzc}{OT1}{pzc}{m}{it}

\title{A PDE approach to fractional diffusion in general domains:
a priori error analysis\thanks{This work is supported by NSF grants: DMS-1109325 and DMS-0807811.
AJS is also supported by NSF grant DMS-1008058 and an AMS-Simons Grant.
EO is supported by the Conicyt-Fulbright Fellowship Beca Igualdad de Oportunidades.}}

\author{Ricardo H.~Nochetto\thanks{Department of Mathematics and Institute for Physical
Science and Technology, University of Maryland, College Park, MD 20742, USA. \texttt{rhn@math.umd.edu}},
\and
Enrique Ot\'arola\thanks{Department of Mathematics, University of Maryland,
College Park, MD 20742, USA. \texttt{kike@math.umd.edu}},
\and
Abner J.~Salgado\thanks{Department of Mathematics, University of Maryland,
College Park, MD 20742, USA. \texttt{abnersg@math.umd.edu}}}

\begin{document}

\maketitle

\begin{abstract}
The purpose of this work is the study of solution techniques for problems involving fractional 
powers of symmetric coercive elliptic operators in a bounded domain 
with Dirichlet boundary conditions.
These operators can be realized as the Dirichlet to Neumann map for
a degenerate/singular elliptic problem posed on a semi-infinite cylinder,
which we analyze in the framework of weighted Sobolev spaces.
Motivated by the rapid decay of the solution of this problem, we propose a truncation that
is suitable for numerical approximation. We discretize this truncation using
first degree tensor product finite elements.
We derive a priori error estimates in weighted Sobolev spaces.
The estimates exhibit optimal regularity
but suboptimal order for quasi-uniform meshes. 
For anisotropic meshes, instead, they are
quasi-optimal in both order and regularity.
We present numerical experiments to illustrate the method's performance.
\end{abstract}



% \begin{center}
% \emph{
% EO and AJS wish to dedicate this work to Lemmy Kilmister for his excellent music and
% dictating their way of life.}
% \end{center}

\section{Introduction}
\label{sec:Intro}
Singular integrals and nonlocal operators have been an active area of research in
different branches of mathematics such as operator theory and harmonic analysis (see
\cite{Stein}). In addition, they have received significant attention because  of their strong
connection with real-world problems, since they constitute a fundamental part of the
modeling and simulation of complex phenomena that span vastly different length scales.

Nonlocal operators arise in a number of applications such as: boundary control problems
\cite{Duvaut}, finance \cite{Carr.Geman.ea2002}, electromagnetic fluids \cite{McCN:81},
image processing \cite{GO:08}, materials science \cite{Bates}, optimization \cite{Duvaut},
porous media flow \cite{CG:93}, turbulence \cite{Bakunin}, peridynamics
\cite{Silling}, nonlocal continuum field theories \cite{Eringen} and others.
Therefore the domain of definition $\Omega$ could be rather general.

To make matters precise, in this work we shall be interested in fractional powers
of the Dirichlet Laplace operator $\Laps$, with $s \in (0,1)$, which for convenience we
will simply call the fractional Laplacian. In other words, we shall be concerned with
the following problem. Let $\Omega$ be an open and bounded subset of  $\R^n$ ($n\ge1$),
with boundary $\partial\Omega$. Given $s\in (0,1)$ and a smooth enough function $f$,
find $u$ such that
\begin{equation}
\label{fl=f_bdddom}
  \begin{dcases}
    \Laps u = f, & \text{in } \Omega, \\
    u = 0, & \text{on } \partial\Omega.
  \end{dcases}
\end{equation}
Our approach, however, is by no means particular to the fractional Laplacian.
In section~\ref{sec:general} we will discuss how, with little modification, our developments
can be applied to a general second order, symmetric and uniformly elliptic operator. 

The study of boundary value problems involving the fractional Laplacian is important
in physical applications where long range or anomalous diffusion is considered.
For instance, in the flow in porous media,
it is used when modeling the transport of particles that experience
very large transitions arising from high heterogeneity and very long spatial
autocorrelation (see \cite{Water:00}).
In the theory of stochastic processes, the fractional Laplacian
is the infinitesimal generator of a stable L\'evy process (see \cite{Bertoin}).

One of the main difficulties in the study of problem \eqref{fl=f_bdddom} is that the
fractional Laplacian is a nonlocal operator (see \cite{Landkof,CS:07,CS:11}).
To localize it, Caffarelli and Silvestre showed in \cite{CS:07} that any power
of the fractional Laplacian in $\R^n$ can be realized as an operator that maps a
Dirichlet boundary condition to a Neumann-type condition via an extension problem on
the upper half-space $\R^{n+1}_+$.
For a bounded domain $\Omega$, the result by Caffarelli and Silvestre has been adapted
in \cite{CDDS:11, BCdPS:12,ST:10}, thus obtaining an extension problem which is now posed
on the semi-infinite cylinder $\C = \Omega \times (0,\infty)$.
This extension is the following mixed boundary value problem:
\begin{equation}
\label{alpha_harm_intro}
\begin{dcases}
  \DIV\left( y^\alpha \nabla \ve \right) = 0, & \text{in } \C, \\
  \ve = 0, & \text{on } \partial_L \C, \\
  \frac{ \partial \ve }{\partial \nu^\alpha} = d_s f, & \text{on } \Omega \times \{0\}, \\
\end{dcases}
\end{equation}
where $\partial_L \C= \partial \Omega \times [0,\infty)$ denotes
the lateral boundary of $\C$, and
\begin{equation}
\label{def:lf}
\frac{\partial \ve}{\partial \nu^\alpha} = -\lim_{y \rightarrow 0^+} y^\alpha \ve_y,
\end{equation}
is the the so-called conormal exterior derivative of $\ve$ with $\nu$ being the unit outer normal to
$\C$ at $\Omega \times \{ 0 \}$. The parameter $\alpha$ is defined as
\begin{equation}
\label{eq:defofalpha}
  \alpha = 1-2s \in (-1,1).
\end{equation}
Finally, $d_s$ is a positive normalization constant
which depends only on $s$; see \cite{CS:07} for details. We will call $y$ the
\emph{extended variable} and the dimension $n+1$ in $\R_+^{n+1}$ the
\emph{extended dimension} of problem \eqref{alpha_harm_intro}.

The limit in \eqref{def:lf} must be understood in the distributional sense;
see \cite{BCdPS:12,CS:11,CS:07} or section~\ref{sec:Prelim} for more details.
As noted in \cite{CS:07,CDDS:11, ST:10}, the fractional Laplacian and the Dirichlet to
Neumann operator of problem \eqref{alpha_harm_intro} are related by
\[
  d_s \Laps u = \frac{\partial \ve}{\partial \nu^\alpha } \quad \text{in } \Omega.
\]

Using the aforementioned ideas, we propose the following strategy to find the solution of
\eqref{fl=f_bdddom}: given a sufficiently smooth function $f$ we solve
\eqref{alpha_harm_intro}, thus obtaining a function $\ve: (x',y) \in \C \mapsto \ve(x',y) \in \R$.
Setting $u: x' \in \Omega \mapsto u(x') = \ve(x',0) \in \R$, we obtain 
the solution of \eqref{fl=f_bdddom}.
The purpose of this work is then to make these ideas rigorous and to analyze
a discretization scheme, which consists of approximating the solution of
\eqref{alpha_harm_intro} via first degree tensor product finite elements.
We will show sub-optimal error estimates for quasi-uniform discretizations of \eqref{alpha_harm_intro}
in suitable weighted Sobolev spaces and quasi-optimal error estimates using anisotropic elements.

The main advantage of the proposed algorithm is that we solve
the local problem \eqref{alpha_harm_intro} instead
of dealing with the nonlocal operator $\Laps$ of problem \eqref{fl=f_bdddom}.
However, this
comes at the expense of incorporating one more dimension to the problem,
and raises questions about computational efficiency.
% Nevertheless, this does not prevent us from carrying out a complete analysis of the
% discretization scheme. 
The development of efficient computational techniques for
the solution of problem \eqref{alpha_harm_intro} and issues such as multilevel methods,
a posteriori error analysis and adaptivity will be deferred to future reports.
In this paper we carry out a complete a priori error analysis of the discretization scheme.

Before proceeding with the analysis of our method, it is instructive to compare
it with those advocated in the literature. First of all,
for a general Lipschitz domain $\Omega \subset \R^n$ ($n > 1$),
we may think of solving problem \eqref{fl=f_bdddom} via a spectral decomposition of the
operator $-\Delta$. However,
to have a sufficiently good approximation, this requires the solution of
a large number of eigenvalue problems which, in general, is very time consuming.
In \cite{ILTA:05,ILTA:06} the authors studied computationally
problem \eqref{fl=f_bdddom} in
the one-dimensional case and with boundary conditions of Dirichlet, Neumann and
Robin type, and introduced the so-called matrix transference technique (MTT). 
Basically, MTT computes a spatial discretization of the fractional
Laplacian by first finding a matrix approximation, $A$,
of the Laplace operator (via finite differences or finite elements)
and then computing the $s$-th power of this matrix. This requires diagonalization of $A$
which, again, amounts to the solution of a large number of eigenvalue problems.
For the case $\Omega = (0,1)^2$ and $s\in(1/2,1)$, \cite{YTLI:11} applies the MTT
technique and avoids diagonalization of $A$ by writing a numerical scheme in terms of
the product of a function of the matrix and a vector, $f(A)b$, where $b$ is a suitable vector.
This product is then approximated by a preconditioned Lanczos method.
Under the same setting, the work \cite{BHK}, makes a computational comparison of
three techniques for the computation of $f(A)b$:
the contour integral method, extended Krylov subspace methods and the pre-assigned
poles and interpolation nodes method.

The outline of this paper is as follows. In \S~\ref{sec:Prelim} we introduce
the functional framework that is suitable for the study of
problems \eqref{fl=f_bdddom} and \eqref{alpha_harm_intro}.
We recall the definition of the fractional Laplacian
on a bounded domain via spectral theory and, in addition, in \S~\ref{subsub:regularity}
we study regularity of the solution to \eqref{alpha_harm_intro}.
The numerical analysis of \eqref{fl=f_bdddom}
begins in \S~\ref{sec:Truncation}. Here we introduce a truncation of problem
\eqref{alpha_harm_intro} and study some properties of its solution.
Having understood the truncation we proceed, in \S~\ref{sec:errors}, to study its finite
element approximation. We prove interpolation estimates in weighted Sobolev spaces, under mild
shape regularity assumptions that allow us to consider anisotropic elements in the
extended variable $y$.
Based on the regularity results of \S~\ref{subsub:regularity} we derive, in 
\S~\ref{sec:ErrorEstimate}, a priori error estimates for quasi-uniform
meshes which exhibit optimal regularity but suboptimal order. To
restore optimal decay, we resort to
the so-called principle of error equidistribution and construct
graded meshes in the extended variable $y$. They in turn capture the singular behavior of the solution 
to \eqref{alpha_harm_intro}
and allow us to prove a quasi-optimal rate of convergence
with respect to both regularity and degrees of freedom.
In \S~\ref{sec:numerics}, to illustrate the method's performance and
theory, we provide several numerical experiments. 
Finally, in \S~\ref{sec:general} we show that our developments
apply to general second order, symmetric and uniformly elliptic operators.

\section{Notation and preliminaries}
\label{sec:Prelim}
Throughout this work $\Omega$ is an open, bounded and connected subset
of $\R^n$, $n\geq1$, with Lipschitz boundary $\partial\Omega$,
unless specified otherwise. We define the semi-infinite cylinder
\begin{equation}
  \label{cilinder}
  \C = \Omega \times (0,\infty),
\end{equation}
and its lateral boundary
\begin{equation}
  \label{cilinderboundary}
  \partial_L \C  = \partial \Omega \times [0,\infty).
\end{equation}
Given $\Y>0$, we define the truncated cylinder
\begin{equation}
  \label{trunccylinder}
  \C_\Y = \Omega \times (0,\Y).
\end{equation}
The lateral boundary $\partial_L\C_\Y$ is defined accordingly.

Throughout our discussion we will be dealing with objects defined in $\R^{n+1}$ 
and it will be convenient to distinguish the extended dimension, as it plays a special r\^ole.
A vector $x\in \R^{n+1}$, will be denoted by
\[
  x =  (x^1,\ldots,x^n, x^{n+1}) = (x', x^{n+1}) = (x',y),
\]
with $x^i \in \R$ for $i=1,\ldots,{n+1}$, $x' \in \R^n$ and $y\in\R$.
The upper half-space in $\R^{n+1}$ will be denoted by
\[
  \R^{n+1}_+ = \left\{x=(x',y): x' \in \R^n\  y \in \R, \ y > 0 \right\}.
\]
Let $ \gamma = (\gamma^1,\gamma^2) \in \R^2$ and $z \in \R^{n+1}$, the
binary operation $ \odot : \R^2 \times \R^{n+1} \rightarrow \R^{n+1}$
is defined by 
\begin{equation}
\label{eq:defcolon}
  \gamma \odot z = (\gamma^1 z', \gamma^2 z^{n+1}) \in \R^{n+1}.
\end{equation}

The relation $a \lesssim b$ indicates that $a \leq Cb$, with a constant $C$ that does not
depend on neither $a$ nor $b$ but it might depend on $s$ and $\Omega$. The value of $C$ might change
at each occurrence. Given two objects $X$ and $Y$ in the same category, we write $X \hookrightarrow Y$
to indicate the existence of a monomorphism between them.
Generally, these will be objects in some subcategory of the topological vector spaces
(metric, normed, Banach, Hilbert spaces) and, in this case,
the monomorphism is nothing more than continuous embedding.
% If $X$ and $Y$ are normed vector spaces, we write $X \hookrightarrow Y$
% to denote that $X$ is continuously embedded in $Y$.
If $X$ is a vector space, we denote by $X'$ its dual.

\subsection{Fractional Sobolev spaces and the fractional Laplacian}
\label{sub:sub:fractional_spaces}
Let us recall some function spaces; for details the reader is referred to
\cite{Lions,McLean,NPV:12,Tartar}.
For $0<s<1$, we introduce the so-called Gagliardo-Slobodecki\u\i\ seminorm
\[
  |w|_{H^s(\Omega)}^2 =
      \int_{\Omega} \int_{\Omega} \frac{ |w(x_1')-w(x_2')|^2 }{ |x_1'-x_2'|^{n+2s} } \diff x_1' \diff x_2'.
\]
The Sobolev space $H^s(\Omega)$ of order $s$ is defined by
\begin{equation}
\label{def:H^sOmega}
  H^s(\Omega) = \left\{ w \in L^2(\Omega): | w |_{H^s(\Omega)} <  \infty \right \},
\end{equation}
which equipped with the norm
\[
  \| u\|_{H^s(\Omega)} = \left(\|u\|^2_{L^2(\Omega)} + | u |^2_{H^s(\Omega)}\right)^{\srn},
\]
is a Hilbert space. An equivalent construction of $H^s(\Omega)$ is obtained by restricting functions
in $H^s(\R^n)$ to $\Omega$ (\cf \cite[Chapter 34]{Tartar}).
The space $H_0^s(\Omega)$ is defined as the closure of
$C_0^{\infty}(\Omega)$ with respect to the norm $\| \cdot \|_{H^s(\Omega)}$, \ie
\begin{equation}
  \label{def:H_0^s}
  H_0^s(\Omega) = \overline{C_0^{\infty}(\Omega)}^{H^s(\Omega)}.
\end{equation}

If the boundary of $\Omega$ is smooth,
an equivalent approach to define fractional Sobolev spaces is given by interpolation in
\cite[Chapter 1]{Lions}. Set $H^0(\Omega) = L^2(\Omega)$, then Sobolev spaces with real
index $0 \leq s \leq 1$ can be defined as interpolation spaces of index $\theta = 1 - s$
for the pair $[H^1(\Omega), L^2(\Omega)]$, that is
\begin{equation}
  \label{H^sinterpolation}
  H^s(\Omega) = \left[H^1(\Omega),L^2(\Omega) \right]_{\theta}.
\end{equation}
Analogously, for $s \in [0,1]\setminus\{\srn\}$, the spaces $H_0^s(\Omega)$ are defined
as interpolation spaces of index $\theta = 1 - s$ for the pair $[H_0^1(\Omega), L^2(\Omega)]$,
in other words
\begin{equation}
  \label{H_0^sinterpolation}
  H_0^s(\Omega) = \left[H_0^1(\Omega),L^2(\Omega) \right]_{\theta},  \quad \theta \neq \srn.
\end{equation}
The space $[H_0^1(\Omega), L^2(\Omega)]_{\srn}$ is the so-called \emph{Lions-Magenes} space,
\[
  H_{00}^{\srn}(\Omega) = \left[H_0^1(\Omega),L^2(\Omega) \right]_{\srn},
\]
which can be characterized as
\begin{equation}
  \label{characLM}
  H_{00}^{\srn}(\Omega) = \left\{ w \in H^{\srn}(\Omega):
        \int_{\Omega} \frac{w^2(x')}{\textrm{dist}(x',\partial \Omega)} \diff x'  < \infty  \right\},
\end{equation}
see \cite[Theorem~11.7]{Lions}. Moreover, we have
the strict inclusion
$ H_{00}^{1/2}(\Omega) \subsetneqq H_0^{1/2}(\Omega)$
because $1 \in H_0^{1/2}(\Omega)$ but $1 \notin H_{00}^{1/2}(\Omega)$.
If the boundary of $\Omega$ is Lipschitz, the characterization
\eqref{characLM} is equivalent to the definition via interpolation, and
definitions \eqref{H^sinterpolation} and \eqref{H_0^sinterpolation}
are also equivalent to definitions \eqref{def:H^sOmega} and \eqref{def:H_0^s}, respectively.
To see this, it suffices to notice that when $\Omega = \R^n$ these definitions
yield identical spaces and equivalent norms; see \cite[Chapter 7]{Adams}.
Consequently, using the well-known extension result of  Stein
\cite{Stein} for Lipschitz domains, we obtain the asserted equivalence
(see \cite[Chapter 7]{Adams} for details).

When the boundary of $\Omega$ is Lipschitz, the space $C_0^{\infty}(\Omega)$
is dense in $H^s(\Omega)$ if and only if $s \leq \srn$,
\ie $H_0^s(\Omega) = H^s(\Omega)$. If $s>\srn$, we have
that $H_0^s(\Omega)$ is strictly contained in $H^s(\Omega)$; see \cite[Theorem~11.1]{Lions}.
In particular, we have the inclusions
$ H_{00}^{1/2}(\Omega) \subsetneqq H_0^{1/2}(\Omega) = H^{1/2}(\Omega) $.

\subsubsection{The fractional Laplace operator}
\label{sub:fractional_laplacian}
It is important to mention that there is not a unique way of defining a nonlocal
operator related to the fractional Laplacian in a bounded domain. A first possibility
is to suitably extend the functions to the whole space $\R^n$ and use Fourier transform
\[
  \F(\Laps w)(\xi') = |\xi'|^{2s}\F(w)(\xi').
\]
After extension, the following point-wise formula also serves as a definition of the fractional
Laplacian
\begin{equation}
\label{int_rep}
    \Laps w(x') = C_{n,s} \textrm{v.p.}\! \int_{\R^n} \frac{w(x')-w(z')}{|x'-z'|^{n+2s}} \diff z',
\end{equation}
where v.p. stands for the Cauchy principal value and $C_{n,s}$ is a
positive normalization constant that depends only on $n$ and $s$ which is introduced
to guarantee that the symbol of the resulting operator is $|\xi'|^{2s}$. For details
we refer the reader to \cite{CS:11,Landkof,NPV:12} and, in particular, to
\cite[Section 1.1]{Landkof} or \cite[Proposition 3.3]{NPV:12} for a proof of the equivalence of 
these two definitions.

Even if we restrict ourselves to definitions that do not require extension, there is more
than one possibility. For instance, the so-called regional fractional Laplacian (\cite{rfl1,rfl2})
is defined by restricting the Riesz integral to $\Omega$, leading to an operator related to a
Neumann problem. A different operator is obtained by using the spectral decomposition 
of the Dirichlet Laplace operator $-\Delta$, see \cite{BCdPS:12,CT:10, CDDS:11}.
This approach is also different to the integral formula  \eqref{int_rep}.
Indeed, the spectral definition depends on the domain $\Omega$ considered, while the integral 
one at any point  is independent of the domain in which the equation is set. For more details see 
the discussion in \cite{SV:12}.

The definition that we shall adopt is as in \cite{BCdPS:12,CT:10, CDDS:11} and
is based on the spectral theory of the Dirichlet Laplacian (\cite{Evans, GT})
as we summarize below.

We define $-\Delta : L^2(\Omega) \rightarrow L^2(\Omega)$ with domain
$\dom(-\Delta) = \{v\in H^1_0(\Omega): \Delta v \in  L^2(\Omega)\}$.
This operator is unbounded, closed and, since $\Omega$ is bounded and with Lipschitz
boundary, regularity theory implies that its inverse is compact.
This implies that the spectrum of the operator $-\Delta$ is discrete, positive and accumulates
at infinity. Moreover, there exist
$\{ \lambda_k,\varphi_k \}_{k\in \mathbb N} \subset \R_+\times H^1_0(\Omega)$
such that $\{\varphi_k \}_{k \in \mathbb N}$ is an orthonormal basis of $L^2(\Omega)$ and,
for $k\in\mathbb N$,
\begin{equation}
  \label{eigenvalue_problem}
  \begin{dcases}
    -\Delta \varphi_k = \lambda_k \varphi_k, & \text{in } \Omega, \\
    \varphi_k = 0, & \text{on } \partial\Omega.
  \end{dcases}
\end{equation}
Moreover, $\{\varphi_k \}_{k \in \mathbb N}$ is an orthogonal basis of $H_0^1(\Omega)$ and
$\|\nabla_{x'} \varphi_k\|_{L^2(\Omega)} = \sqrt{\lambda_k}$.

With this spectral decomposition at hand, fractional powers of the Dirichlet Laplacian
$\Laps$ can be defined for $u \in C_0^{\infty}(\Omega)$ by
\begin{equation}
  \label{def:s_in_O}
  \Laps u = \sum_{k=1}^\infty u_k \lambda_k^{s}\varphi_k,
\end{equation} 
where the coefficients $u_k$ are defined by
$ u_k = \int_{\Omega} u \varphi_k $. 
Therefore, if $f=\sum_{k=1}^{\infty}f_k \varphi_k$, 
and $\Laps u = f$, then $u_k = \lambda_k^{-s}f_k$ for all $k \geq 1$.

By density the operator $\Laps$ can be extended to the Hilbert space
\[
  \Hs = \left\{ w = \sum_{k=1}^\infty w_k \varphi_k\in L^2(\Omega): \| w \|_{\Hs}^2 = \sum_{k=1}^{\infty}
    \lambda_k^s |w_k|^2 < \infty \right\}.
\]
The theory of Hilbert scales presented in \cite[Chapter 1]{Lions} shows that
\[
  \left[H_0^1(\Omega),L^2(\Omega) \right]_{\theta} = \dom(-\Delta)^{\stwo},
\]
where $\theta = 1-s$. This implies the following characterization of the
space $\Hs$,
\begin{equation}
  \label{H}
  \Hs =
  \begin{dcases}
    H^s(\Omega),  & s \in (0,\sr), \\
    H_{00}^{1/2}(\Omega), & s = \sr, \\
    H_0^s(\Omega), & s \in (\sr,1).
  \end{dcases}
\end{equation} 

\subsection{Weighted Sobolev spaces}
\label{sub:sub:weighted_spaces}
To exploit the Caffarelli-Silvestre extension
\cite{CS:07}, or its variants \cite{BCdPS:12,CT:10, CDDS:11},
we need to deal with a degenerate/singular elliptic equation on
$\R_+^{n+1}$. To this end, we consider
weighted Sobolev spaces (see, for instance, \cite{FKS:82,HKM,Kufner}),
with the specific weight $|y|^{\alpha}$ with $\alpha \in (-1,1)$.

Let $\D \subset \R^{n+1}$ be an open set and $\alpha \in (-1,1)$. We define
$L^2(\D, |y|^{\alpha})$ as the space of all measurable functions defined on $\D$ such that
\[
\| w \|_{L^2(\D,|y|^{\alpha})}^2 = \int_{\D}|y|^{\alpha} w^2 < \infty.
\]
Similarly we define the weighted Sobolev space
\[
H^1(\D,|y|^{\alpha}) =
  \left\{ w \in L^2(\D,|y|^{\alpha}): | \nabla w | \in L^2(\D,|y|^{\alpha}) \right\},
\]
where $\nabla w$ is the distributional gradient of $w$. We equip $H^1(\D,|y|^{\alpha})$ with
the norm
\begin{equation}
\label{wH1norm}
\| w \|_{H^1(\D,|y|^{\alpha})} =
\left(  \| w \|^2_{L^2(\D,|y|^{\alpha})} + \| \nabla w \|^2_{L^2(\D,|y|^{\alpha})} \right)^{\srn}.
\end{equation}
Notice that taking $\alpha = 0$ in the definition above, we obtain the classical
$H^1(\D)$.

Properties of this weighted Sobolev space can be found in classical references
like \cite{HKM, Kufner}. It is remarkable that most of the properties of classical Sobolev spaces
have a weighted counterpart and it is more so that this is not because of the specific form of the
weight but rather due to the fact that the weight $|y|^\alpha$ belongs to the so-called
Muckenhoupt class $A_2(\R^{n+1})$; see \cite{FKS:82,GU,Muckenhoupt}.
We recall the definition of Muckenhoupt classes.

\begin{definition}[Muckenhoupt class $A_p$]
 \label{def:Muckenhoupt}
Let $\omega$ be a positive and measurable function such that
$\omega \in L^1_{loc}(\R^{N})$ with $N \geq 1$. We say $\omega \in A_p(\R^N)$,
$1 < p < \infty$, if there exists a positive constant $C_{p,\omega}$ such that
\begin{equation}
  \label{A_pclass}
  \sup_{B} \left( \frac{1}{|B|}\int_{B} \omega \right)
            \left(\frac{1}{|B|}\int_{B} \omega^{1/(1-p)} \right)^{p-1}  = C_{p,\omega} < \infty,
\end{equation}
where the supremum is taken over all balls $B$ in $\R^N$ and $|B|$ denotes the Lebesgue measure of
$B$.
\end{definition}

Since $\alpha \in (-1,1)$ it is immediate that $|y|^{\alpha} \in A_2(\R^{n+1})$, which
implies the following important result (see \cite[Theorem~1]{GU}).

\begin{proposition}[Properties of weighted Sobolev spaces]
\label{PR:}
Let $\D \subset \R^{n+1}$ be an open set and $\alpha \in (-1,1)$. Then $H^1(\D,|y|^{\alpha})$,
equipped with the norm \eqref{wH1norm}, is a Hilbert space. Moreover, the set
$C^{\infty}(\D) \cap H^1(\D,|y|^{\alpha})$ is dense in $H^1(\D,|y|^{\alpha})$.
\end{proposition}

\begin{remark}[Weighted $L^2$ vs $L^1$]
\label{RK:L1} \rm
If $\D$ is a bounded domain and $\alpha \in (-1,1)$ then, $L^2(\D,|y|^{\alpha}) \subset
L^1(\D)$. Indeed, since $|y|^{-\alpha} \in L^1_{loc}(\R^{n+1})$,
\[
  \int_{\D} |w| = \int_{\D} |w| |y|^{\alpha/2} |y|^{-\alpha/2}
  \leq \left( \int_{\D} |w|^2 |y|^{\alpha} \right)^{\srn}\left( \int_{\D} |y|^{-\alpha} \right)^{\srn}
  \lesssim \| w\|_{L^2(\D,|y|^{\alpha})}.
\]
\end{remark}

The following result is given in \cite[Theorem~6.3]{Kufner}. For completeness we
present here a version of the proof on the truncated cylinder $\C_\Y$, which will be
important for the numerical approximation of problem \eqref{alpha_harm_intro}.

\begin{proposition}[Embeddings in weighted Sobolev spaces]
\label{PR:comp_classical_weighted}
Let $\Omega$ be a bounded domain in $\R^n$ and $\Y > 0$. Then
\begin{equation}
\label{clas_into_weighted}
  H^1(\C_\Y) \hookrightarrow H^1(\C_\Y,y^{\alpha}), \quad \textrm{ for } \alpha \in (0,1),
\end{equation}
and
\begin{equation}
\label{weighted_into_clas}
  H^1(\C_\Y,y^{\alpha}) \hookrightarrow H^1(\C_\Y), \quad \textrm{ for } \alpha \in (-1,0).
\end{equation}
\end{proposition}
\begin{proof}
Let us prove \eqref{clas_into_weighted}, the proof of \eqref{weighted_into_clas} being
similar. Since $\alpha>0$ we have $y^\alpha\le\Y^\alpha$, whence
$y^\alpha w^2 \leq \Y^\alpha w^2$ and
$y^{\alpha} |\nabla w|^2 \leq \Y^\alpha |\nabla w|^2$ a.e.~on $\C_\Y$
for all $w\in H^1(\C_\Y)$. This implies
$
  \| w \|_{H^1(\C_\Y,y^{\alpha})} \leq {\sqrt{2}} \Y^{\alpha/2} \| w \|_{H^1(\C_\Y )},
$
which is \eqref{clas_into_weighted}.
\end{proof}

Define
\begin{equation}
  \label{HL10}
  \HL(\C,y^{\alpha}) = \left\{ w \in H^1(y^\alpha;\C): w = 0 \textrm{ on } \partial_L \C\right\}.
\end{equation}
This space can be equivalently defined as the set of
measurable functions $w: \C \rightarrow \R$ such that $w \in H^1(\Omega \times (s,t))$ for all
$0 < s < t < \infty$, $w = 0$ on $\partial_L \C$ and for which the following seminorm is finite
\begin{equation}
  \label{seminormHL10}
  \| w \|^2_{\HLn(\C,y^{\alpha})} = \int_{\C}y^{\alpha} |\nabla w |^2;
\end{equation}
see \cite{CDDS:11}. As a consequence of the usual Poincar\'e inequality, for any $k \in \mathbb{Z}$
and any function $w \in H^1(\Omega \times (2^k,2^{k+1}))$ with $w = 0$ on
$\partial \Omega \times (2^k,2^{k+1})$, we have
\begin{equation}
\label{Poincare_interval}
  \int_{ \Omega \times (2^k,2^{k+1})} y^{\alpha}  w^2
    \leq C_{\Omega} \int_{ \Omega \times (2^k,2^{k+1})} y^{\alpha} |\nabla w |^2,
\end{equation}
where $C_\Omega$ denotes a positive constant that depends only on $\Omega$.
Summing up over $k \in \mathbb{Z}$, we obtain the following \emph{weighted Poincar\'e inequality}:
\begin{equation}
\label{Poincare_ineq}
  \int_{ \C }y^{\alpha}  w^2 \lesssim \int_{ \C}y^{\alpha} |\nabla w |^2.
\end{equation}
Hence, the seminorm \eqref{seminormHL10} is a norm on $\HL(\C,y^{\alpha})$,
equivalent to \eqref{wH1norm}.

For a function $w \in H^1(\C, y^{\alpha})$, we shall denote by $\tr w$ its trace onto
$\Omega \times \{ 0 \}$. It is well known that $\tr H^1(\C) = {H^{1/2}(\Omega)}$;
see \cite{Adams,Tartar}. In the subsequent analysis we need a characterization of the trace
of functions in $H^1(\C, y^{\alpha})$. For a smooth domain this was given in
\cite[Proposition~1.8]{CT:10} for $s=1/2$ and in \cite[Proposition~2.1]{CDDS:11} 
for any $s \in (0,1)\setminus\{\srn\}$.
However, since the eigenvalue decomposition \eqref{def:s_in_O} of the Dirichlet Laplace operator
holds true on a Lipschitz domain, we are able to extend this trace characterization
to such domains. In summary, we have the following result.
\begin{proposition}[Characterization of  $\tr \HL(\C, y^\alpha)$]
\label{H=LM=TR}
Let $\Omega \subset \R^n$ be a bounded Lipschitz domain.
The trace operator $\tr$ satisfies $\tr \HL(\C, y^\alpha) = \Hs$ and
\[
  \|\tr v\|_{\Hs} \lesssim \|v\|_{\HLn(\C, y^\alpha)}
\qquad\forall \, v\in \HL(\C, y^\alpha),
\]
where the space $\Hs$ is defined in \eqref{H}.
\end{proposition}

\subsection{The Caffarelli-Silvestre extension problem}
\label{sub:CaffarelliSilvestre}
It has been shown in \cite{CS:07} that any power of the fractional
Laplacian in $\R^n$ can be determined as an operator that maps a Dirichlet boundary condition
to a Neumann-type condition via an extension problem posed on $\R_+^{n+1}$. For
a bounded domain, an analogous result has been obtained in \cite{CT:10} for $s=\sr$, and in
\cite{BCdPS:12,CDDS:11,ST:10} for any $s \in (0,1)$.

Let us briefly describe these results.
Consider a function $u$ defined on $\Omega$. We define the $\alpha$-harmonic extension
of $u$ to the cylinder $\C$, as the function $\ve$ that solves the boundary value problem
\begin{equation}
\label{alpha_harmonic_extension}
  \begin{dcases}
    \DIV (y^\alpha \nabla \ve ) = 0, & \text{in } \C, \\
    \ve = 0, & \text{on } \partial_L\C, \\
    \ve = u, & \text{on } \Omega\times\{0\}.
  \end{dcases}
\end{equation}
From Proposition~\ref{H=LM=TR} and the Lax Milgram lemma we can conclude that this
problem has a unique solution $\ve \in \HL(\C,y^\alpha)$ whenever $u\in \Hs$.
We define the  \emph{Dirichlet-to-Neumann} operator $\Gamma_{\alpha,\Omega} : \Hs \to \Hs'$
\[
   u \in \Hs \longmapsto
    \Gamma_{\alpha,\Omega}(u) = \frac{\partial \ve}{\partial \nu^{\alpha}} \in \Hs',
\]
where $\ve$ solves \eqref{alpha_harmonic_extension} 
and $\tfrac{\partial \ve}{\partial \nu^{\alpha}}$ is given in \eqref{def:lf}.
The space $\Hs'$ can be characterized as the space of distributions $h = \sum_k h_k \varphi_k$
such that $\sum_k |h_k|^2 \lambda_k^{-s} < \infty$. The fundamental result of \cite{CS:07},
see also \cite[Lemma~2.2]{CDDS:11}, is then that
\[
  d_s \Laps u = \Gamma_{\alpha,\Omega}(u),
\]
where $d_s$ is given by
\begin{equation}
\label{d_s}
  d_s = 2^{1-2s} \frac{\Gamma(1-s)}{\Gamma(s)}.
\end{equation}
It seems remarkable that this constant does not depend on the dimension. This
was proved originally in \cite{CS:07} and its precise value appears in
several references, for instance \cite{BCdPS:12,CS:11}.

The relation between the fractional Laplacian and the extension problem is now clear.
Given $f \in \Hs'$, a function $u \in \Hs$ solves \eqref{fl=f_bdddom} if and only if its
$\alpha$-harmonic extension $\ve \in \HL(\C,y^{\alpha})$ solves \eqref{alpha_harm_intro}.

If $u=\sum_k u_k \varphi_k$, then, as shown in the proofs of
\cite[Proposition~2.1]{CDDS:11} and \cite[Lemma~2.2]{BCdPS:12}, $\ve$ can be expressed as
\begin{equation}
\label{exactforms}
  \ve(x) = \sum_{k=1}^\infty u_k \varphi_k(x') \psi_k(y),
\end{equation}
where the functions $\psi_k$ solve
\begin{equation}
\label{psik}
  \begin{dcases}
    \psi_k'' + \frac{\alpha}{y}\psi_k' - \lambda_k \psi_k = 0, & \text{in } (0,\infty), \\
    \psi_k(0) = 1, & \lim_{y\rightarrow \infty} \psi_k(y) = 0.
  \end{dcases}
\end{equation}
If $s=\sr$, then clearly $\psi_k(y) = e^{-\sqrt{\lambda_k}y}$ (see \cite[Lemma~2.10]{CT:10}).
For $s \in (0,1)\setminus\{\srn\}$ instead (\cf \cite[Proposition~2.1]{CDDS:11})
\[
  \psi_k(y) = c_s \left(\sqrt{\lambda_k}y\right)^s K_s(\sqrt{\lambda_k} y),
\]
where $K_s$ denotes the modified Bessel function of the second kind (see \cite[Chapter~9.6]{Abra}).
Using the condition $\psi_k(0) = 1$, and formulas for small arguments of the function $K_s$
(see for instance \S~\ref{subsub:asymptotics}) we obtain
\[
  c_s = \frac{2^{1-s}}{\Gamma(s)}.
\]

The function $\ve \in \HL(\C,y^{\alpha})$ is the unique solution of
\begin{equation}
\label{alpha_harmonic_extension_weak}
  \int_{\C} y^{\alpha} \nabla \ve \cdot \nabla \phi = d_s \langle f, \tr \phi \rangle_{\Hs \times \Hs'},
  \quad \forall \phi \in \HL(\C,y^{\alpha}),
\end{equation}
where  $\langle \cdot, \cdot \rangle_{\Hs \times \Hs'}$ denotes the duality pairing between $\Hs$ and
$\Hs'$ which, in light of Proposition~\ref{H=LM=TR} is well defined for all
$f \in \Hs'$ and $\phi \in \HL(\C,y^{\alpha})$. This implies the
following equalities
(see \cite[Proposition 2.1]{CDDS:11} for $s \in (0,1)\setminus\{\srn\}$
and  \cite[Proposition 2.1]{CT:10} for $s = \srn $):
\begin{equation}
\label{estimate_s}
  \| \ve \|_{\HLn(\C,y^{\alpha})} = d_s \| u \|_{\Hs} = d_s \| f\|_{\Hs'}.
\end{equation}

Notice that for $s = \srn $, or equivalently $\alpha=0$, 
problem \eqref{alpha_harmonic_extension_weak} reduces to the 
weak formulation of the Laplace operator with mixed boundary conditions,
which is posed on the classical Sobolev space $\HL(\C)$. 
Therefore, the value $s = \srn$ becomes a special case for
problem \eqref{alpha_harmonic_extension_weak}. In addition,
$d_{1/2} = 1$, and $\| \ve \|_{\HLn(\C)} = \| u \|_{H^{1/2}_{00}(\Omega)}$.

At this point it is important to give a precise meaning to the Dirichlet boundary condition
in \eqref{fl=f_bdddom}.
For $s=\srn$, the boundary condition is interpreted in the 
sense of the Lions--Magenes space.
If $\srn < s \leq 1$, there is a trace operator from $\Hs$ into $L^2(\partial \Omega)$
and the boundary condition can be interpreted in this sense. 
For $0 < s < 1/2$
this interpretation is no longer possible
and thus, for an arbitrary $f \in \Hs'$ the boundary condition does not have a clear
meaning. For instance, for every $s \in (0,\srn)$, $f = \Laps 1\in \Hs'$
and the solution to \eqref{fl=f_bdddom} for this right hand side is $u=1$.
If $f \in H^\zeta(\Omega)$ with $\zeta > \srn - 2s > -s$,
using that $\Laps$ is a pseudo-differential operator of order $2s$ a shift-type result is valid, \ie  
$u \in H^{\varrho}(\Omega)$ with $\varrho = \zeta + 2s > 1/2$.
In this case, the trace of $u$ on $\partial \Omega$ is well defined and the boundary
condition is meaningful. Finally, we comment that it has been proved in \cite[Lemma~2.10]{CDDS:11},
that if $f \in L^{\infty}(\Omega)$ then the solution of \eqref{fl=f_bdddom} belongs to
$C^{0,\varkappa}(\overline{\Omega})$ with $\varkappa \in (0,\min\{2s,1\})$.

\subsection{Asymptotic estimates}
\label{subsub:asymptotics}
It is important to understand the behavior of the solution $\ve$ of problem \eqref{alpha_harm_intro}, 
given by \eqref{exactforms}. Consequently, it becomes necessary
to recall some of the main properties of the modified Bessel function of the second kind 
$K_{\nu}(z)$, $\nu \in \R$; see \cite[Chapter 9.6]{Abra}
for \eqref{itemi}-\eqref{itemiv} and \cite[Theorem 5]{MS:01} for \eqref{itemv}:

\begin{enumerate}[(i)]
 \item \label{itemi}For $\nu >-1$ , $K_{\nu}(z)$ is real and positive.
 \item \label{itemii}For $\nu \in \R$, $K_{\nu}(z) = K_{-\nu}(z)$.
 \item \label{itemiii}For $\nu > 0$,
   \begin{equation}
     \label{asympat0}
     \lim_{z \downarrow 0} \frac{K_{\nu}(z)}{\sr \Gamma(\nu) \left( \sr z \right)^{-\nu} } = 1.
   \end{equation}
 \item \label{itemiv}For $k \in \mathbb{N}$,
  \[
   \left( \frac{1}{z} \frac{\diff }{\diff z}\right)^k \left( z^{\nu} K_{\nu}(z) \right)
    = (-1)^k z^{\nu-k}K_{\nu-k}(z).
  \]
In particular, for $k=1$ and $k=2$, respectively, we have
\begin{equation}
 \label{1derivative}
 \frac{\diff }{\diff z} \left( z^{\nu} K_{\nu}(z) \right) = -z^{\nu} K_{\nu-1}(z) = -z^{\nu} K_{1-\nu}(z),
\end{equation}
and
\begin{equation}
 \label{2derivative}
 \frac{\diff^2 }{\diff z^2} \left( z^{\nu} K_{\nu}(z) \right) = z^{\nu} K_{2- \nu}(z) -z^{\nu-1} K_{1-\nu}(z).
\end{equation}
\item \label{itemv}For $z>0$, $z^{\min\{\nu,1/2\}}e^{z}K_{\nu}(z)$ is a decreasing function.
\end{enumerate}

As an application we obtain the following important
properties of the function $\psi_k$, defined in \eqref{psik}. First, for $s \in (0,1)$,
properties \eqref{itemii}, 
\eqref{itemiii} and \eqref{itemiv}
imply
\begin{equation}
 \label{asym-psi_k-p}
 \lim_{y \downarrow 0^{+} } \frac{y^{\alpha}\psi_{k}'(y)}{d_s \lambda_k^s } = - 1,
\end{equation}
Property \eqref{itemv} provides the following asymptotic estimate for $s \in (0,1)$
and $y \geq 1$:
\begin{equation}
 \label{y-psi_k-psi_kp}
 |y^{\alpha} \psi_k(y) \psi_k'(y)| 
 \leq C(s) \lambda_{k}^s \left( \sqrt{\lambda_k} y \right)^{\left|s-\sr \right|}
 e^{-2 \sqrt{\lambda_k} y}.
\end{equation}
Multiplying the differential equation of problem \eqref{psik} 
by $y^{\alpha}\psi_k(y)$ and integrating by parts yields
\begin{equation}
 \label{besselenergy}
 \int_{a}^{b} y^{\alpha} \left( \lambda_k \psi_k(y)^2 + \psi_k'(y)^2 \right) \diff y
 = \left. y^{\alpha} \psi_k(y)\psi_k'(y) \right|_{a}^{b},
\end{equation}
where $a$ and $b$ are real and positive constants.

Let us conclude this section with some remarks on the asymptotic
behavior of the function $\ve$ that solves \eqref{alpha_harmonic_extension_weak}.
Using \eqref{exactforms} we obtain
\[
  \ve(x)|_{y=0} = \sum_{k=1}^{\infty} u_k \varphi_k(x') \psi_k(0) =
  \sum_{k=1}^{\infty} u_k \varphi_k(x') = u(x').
\]
For $s \in (0,1)$, using formula \eqref{asym-psi_k-p} together with \eqref{def:s_in_O}, we arrive at
\begin{equation}
\label{asymptoticve}
  \frac{\partial\ve}{\partial \nu^\alpha}(x',0) = 
  -\lim_{y \downarrow 0 } y^{\alpha} \ve_y(x',y) = d_s f(x'), \quad \textrm{on } \Omega \times \{0\}.
\end{equation}
Notice that, if $s = \srn$, then $\alpha=0$, $d_{1/2}=1$ and thus
\eqref{asymptoticve} reduces to
\begin{equation}
\nonumber
  \left.\frac{\partial \ve}{\partial \nu}\right|_{\Omega \times \{0\}} =  f(x').
\end{equation}
For $s \in (0,1) \setminus \{ \srn \}$ the asymptotic behavior of 
the second derivative $\ve_{yy}$ as $y \approx 0^+$ is a consequence of
\eqref{2derivative} applied to the function $\psi_k(y)$. 
For $s=\srn$ the behavior follows from $\psi_k(y) = e^{-\sqrt{\lambda_k}y}$.
In conclusion, for $y \approx 0^{+}$, we have
\begin{equation}
\label{asymptoticveyy}
  \ve_{yy} \approx  y^{-\alpha-1} \quad\textrm{for } s \in (0,1) \setminus \{ \srn \}, 
\qquad \ve_{yy} \approx 1 \quad\textrm{for } s=\srn.
\end{equation}

\subsection{Regularity of the solution}
\label{subsub:regularity}
Since we are interested in the approximation of the
solution of problem \eqref{alpha_harmonic_extension_weak}, and this
is closely related to its regularity, let us now study the
behavior of its derivatives. According to \eqref{asymptoticve},
$\ve_y \approx y^{-\alpha}$ for $y \approx 0^+$. This clearly shows the
necessity of introducing the weight, as this behavior, together
with the exponential decay given by \eqref{itemv} of \S~\ref{subsub:asymptotics}, imply that
$\ve_y \in L^2(\C,y^\alpha)\setminus L^2(\C)$ for $s \in (0,1/4]$.

However, the situation with second derivatives  is much more delicate. To see this,
let us first argue heuristically and compute how these derivatives scale with $y$.
From the asymptotic formula \eqref{asymptoticveyy},
we see that, for $0< \delta \ll 1 $ and $s \in (0,1)\setminus \{\srn\}$,
\begin{equation}
\label{int=inf}
  \int_{\Omega \times (0,\delta)} y^{\alpha}\left| \ve_{yy} \right|^2 \diff x' \diff y
  \approx \int_{0}^{\delta} y^{\alpha} y^{-2-2\alpha} \diff y =
  \int_{0}^{\delta} y^{-2-\alpha} \diff y,
\end{equation}
which, since $\alpha \in (-1,1)\setminus \{0\}$, does not converge. However,
\[
  \int_{\Omega \times (0,\delta)} y^{\beta}\left| \ve_{yy} \right|^2 \diff x \diff y \approx
  \int_{0}^{\delta} y^{\beta-2-2\alpha} \diff y,
\]
converges for $\beta > 2\alpha + 1$, hinting at the fact that
$\ve \in H^2(\C, y^\beta) \setminus H^2(\C, y^\alpha)$.
The following result makes these considerations
rigorous. 

\begin{theorem}[Global regularity of the $\alpha$-harmonic extension]
\label{TH:regularity}
Let $f \in \Ws$, where $\Ws$ is defined in \eqref{H} for $s\in(0,1)$. Let
$\ve \in \HL(\C,y^{\alpha})$ solve \eqref{alpha_harmonic_extension_weak} with $f$ as data.
Then, for $s \in (0,1) \setminus \{\sr\}$, we have
\begin{equation}
\label{reginx}
  \| \Delta_{x'} \ve \|^2_{L^2(\C,y^{\alpha})} + \| \partial_y \nabla_{x'} \ve \|^2_{L^2(\C,y^{\alpha})}
  = d_s \| f \|_{\Ws}^2,
\end{equation}
and
\[
  \| \ve_{yy} \|_{L^2(\C,y^{\beta})} \lesssim \| f \|_{L^2(\Omega)},
\]
with $\beta>2\alpha+1$.
For the special case $s = \sr$, we obtain
\[
 \| \ve \|_{H^2(\C)} \lesssim \| f \|_{\mathbb{H}^{1/2}(\Omega)}.
\]
\end{theorem}

\begin{remark}[Compatibility of $f$]
\label{R:compatibitity} \rm
It is possible to interpret the result of Theorem~\ref{TH:regularity} as follows.
Consider $s \in (\srn,1)$, or equivalently $\alpha \in (-1,0)$. Then the
conormal exterior derivative condition for $\ve$ gives us that
$\ve_y \approx -d_s y^{-\alpha}f$ as $y \approx 0^+$ on $\Omega \times \{0\}$, which in turn implies
that $\ve_y \to 0$ as $y \to 0^+$ on $\Omega \times \{0\}$. This is compatible with
$\ve = 0$ on $\partial_L \C$ since this implies $\ve_y = 0$ on $\partial_L \C$.
Consequently, we do not need any compatibility condition on the data $f \in H^{1-s}(\Omega)$
to avoid a jump on the derivative $\ve_y$.
On the other hand, when $\alpha \in (0,1)$, we have that, for a general $f$,
$\ve_y \nrightarrow 0$ as $y \to 0^+$ on $\Omega \times \{0\}$.
To compensate this behavior
we need the data $f$ to vanish at the boundary $\partial \Omega$ at a certain rate.
This condition is expressed by the requirement $f \in H_0^{1-s}(\Omega)$. 
\end{remark}

\textit{Proof of Theorem~\ref{TH:regularity}.}
Let us first consider $s=\srn$. In this case \eqref{alpha_harmonic_extension_weak}
reduces to the Poisson problem with mixed boundary conditions. In general, the solution
of a mixed boundary value problem is not smooth, even for $C^\infty$ data. The singular
behavior occurs near the points of intersection between the Dirichlet and Neumann boundary.
For instance, the solution $w = \sqrt{r}\sin(\theta/2)$ of $\Delta w=0$ in $\R^2_+$,
with $w_{x_2} = 0$ for $\{ x_1 < 0, \ x_2 = 0\}$
and $w = 0$ for $\{ x_1 \geq 0, \ x_2 = 0\}$ does not belong to $H^2(\R^2_+)$.
To obtain more regular solutions, a compatibility condition between the data, the operator
and the boundary must be imposed (see, for instance, \cite{Savare:97}).
Since in our case we have the representation \eqref{exactforms}, we can
explicitly compute the second derivatives and,
using that $\{ \varphi_k\}_{k\in \mathbb N}$ is an orthonormal basis of $L^2(\Omega)$ and
$\{ \varphi_k/\sqrt{\lambda_k}\}_{k\in \mathbb N}$ of $H_0^1(\Omega)$, it is not 
difficult to show that $f \in H_{00}^{1/2}(\Omega)$ implies $\ve \in H^2(\C)$, 
and $\| \ve \|_{H^2(\C)} \lesssim \|f \|_{ H_{00}^{1/2}(\Omega)}$.

In the general case $s \in (0,1) \setminus \{\sr\}$, \ie $\alpha \in (-1,1)\setminus \{0\}$, 
using \eqref{besselenergy} as well as the asymptotic properties
\eqref{asym-psi_k-p} and \eqref{y-psi_k-psi_kp}, we obtain
\begin{align*}
  \| \Delta_{x'} \ve \|^2_{L^2(\C,y^{\alpha})} + \| \partial_y \nabla_{x'} \ve \|^2_{L^2(\C,y^{\alpha})}
    &= \sum_{k=1}^{\infty} u_k^2 \lambda_k \int_0^{\infty}y^{\alpha} \left( \lambda_k \psi_k(y)^2 + 
    \psi_k'(y)^2 \right)\diff y 
    \\
    & = d_s\sum_{k=1}^{\infty} u_k^2 \lambda_k^{1+s} = d_s\sum_{k=1}^{\infty} f_k^2 \lambda_k^{1-s} 
    = d_s\| f\|^2_{\Ws},
\end{align*}
which is exactly the regularity estimate given in \eqref{reginx}.
To obtain the regularity estimate on $\ve_{yy}$ we, again, use the exact representation
\eqref{exactforms} and properties of Bessel functions to conclude that any derivative with
respect to the extended variable $y$ is smooth away from the Neumann boundary 
$\Omega\times\{0\}$. By virtue of \eqref{psik} we deduce that the following
partial differential equation holds in the strong sense
% this means that the equation is satisfied in the 
% strong sense and that 
\begin{equation}
\label{pde}
  \DIV(y^{\alpha} \nabla \ve)=0 \Longleftrightarrow \ve_{yy}=-\Delta_{x'} \ve -\frac{\alpha}{y} \ve_y.
\end{equation}
Consider  sequences $\{a_k = 1/\sqrt{\lambda_k}\}_{k\geq1},\ \{b_k\}_{k\geq1}$ and
$\{\delta_k\}_{k\geq1}$ with $0 < \delta_k \leq a_k \leq b_k$.
Using \eqref{exactforms} we have, for $k\geq1$,
\begin{equation}
\label{int_vyy_reg}
    \| \ve_{yy} \|_{L^2(\C,y^\beta)}^2 = \sum_{k=1}^\infty  u_k^2
    \left(
      \lim_{\delta_k \downarrow 0} \int_{\delta_k}^{a_k} y^\beta | \psi_k''(y)|^2 \diff y
      + \lim_{b_k \uparrow \infty} \int_{a_k}^{b_k} y^\beta | \psi_k''(y)|^2 \diff y
    \right)
\end{equation}
Let us now estimate the first integral on the right hand side of \eqref{int_vyy_reg}. 
Formulas \eqref{2derivative} and \eqref{asympat0} yield 
\begin{equation}
  \begin{aligned}
    \lim_{\delta_k \downarrow 0}\int_{{\delta_k}}^{a_k} y^\beta |\psi_k''(y)|^2 \diff y  & = 
    c_s^2 \lambda_k^{2 - \beta/2 -1/2} 
    \lim_{\delta_k \downarrow 0} \int_{\sqrt{\lambda_k}\delta_k}^{1} 
    z^\beta  \left| \frac{\diff^2 }{\diff z^2} \left( z^s K_s(z) \right) \right|^2 \diff z \\
    & \lesssim c_s^2 \lambda_k^{2 - \beta/2 -1/2} 
      \lim_{\delta_k \downarrow 0} \int_{\sqrt{\lambda_k}\delta_k}^{1}
      z^{\beta-2-2\alpha}\diff z \approx \lambda_k^{2 - \beta/2 -1/2}
  \end{aligned}
\label{reg_y_1}
\end{equation}
where the integral converges because $\beta > 2\alpha + 1$.
Let us now look at the second integral.
Using property \eqref{itemv} of the modified Bessel functions, we have
\begin{equation}
 \begin{aligned}
    \lim_{b_k \uparrow \infty}\int_{a_k}^{b_k} y^\beta |\psi_k''(y)|^2 \diff y & = 
    c_s^2 \lambda_k^{2 - \beta/2 -1/2}
    \lim_{b_k \uparrow \infty}
    \int_{1}^{\sqrt{\lambda_k}b_k}
    z^\beta  \left|\frac{\diff^2 }{\diff z^2} \left( z^s K_s(z) \right) \right|^2 \diff z
    \\
    & \lesssim c_s^2  \lambda_k^{2 - \beta/2 -1/2}.
 \end{aligned}
\label{reg_y_2}
\end{equation}
Replacing \eqref{reg_y_1} and \eqref{reg_y_2} into \eqref{int_vyy_reg}, and using that
$u_k = \lambda_k^{-s}f_k$, we deduce
\[
  \| \ve_{yy} \|_{L^2(\C,y^\beta)}^2 \lesssim 
  \sum_{k=1}^\infty \lambda_k^{2 - \beta/2 -1/2-2s} f_k^2 \leq 
  \| f \|_{L^2(\Omega)}^2,
\]
because $2-2s-\frac\beta2-\frac12 = \frac12(1+2\alpha-\beta)< 0$.
This concludes the proof.
\endproof

For the design of graded meshes later in \S~\ref{sub:graded} we also need
the following local regularity result in the extended variable.

\begin{theorem}[Local regularity of the $\alpha$-harmonic extension]\label{T:local}
Let $\C(a,b) := \Omega \times (a,b)$ for $0 \le a < b\le 1$.
The solution $\ve\in\HL(\C,y^{\alpha})$ of
\eqref{alpha_harmonic_extension_weak} satisfies for all $a, b$
\begin{equation}\label{local_alpha}
  \| \Delta_{x'} \ve \|^2_{L^2(\C(a,b),y^{\alpha})} 
  + \| \partial_y \nabla_{x'} \ve \|^2_{L^2(\C(a,b),y^{\alpha})}
  \lesssim  \left( b - a \right) \| f \|_{\Ws}^2,
\end{equation}
and, with $\delta := \beta-2\alpha-1 > 0$,
\begin{equation}\label{local_beta}
  \| \ve_{yy} \|_{L^2(\C(a,b),y^{\beta})}^2 \lesssim 
  \left( b^\delta - a^\delta \right) \| f \|_{L^2(\Omega)}^2.
\end{equation}
\end{theorem}
\begin{proof}
To derive \eqref{local_alpha} we proceed as in Theorem~\ref{TH:regularity}.
Since $0 \leq a < b \leq 1$, property \eqref{itemiii} of \S~\ref{subsub:asymptotics},
together with \eqref{asym-psi_k-p} imply that
\[
  \left| y^\alpha \psi_k(y)\psi'_k(y) \right| \lesssim \lambda_{k}^s.
\]
This, together with \eqref{besselenergy} and the property
$u_k=\lambda_k^{-s}f_k$, allows us to conclude
\begin{align*}
  \| \Delta_{x'} \ve \|^2_{L^2(\C(a,b),y^{\alpha})} + \| \partial_y \nabla_{x'} \ve 
  \|^2_{L^2(\C(a,b),y^{\alpha})}
    &= \sum_{k=1}^{\infty} u_k^2 \lambda_k \int_a^{b}y^{\alpha} \left( \lambda_k \psi_k(y)^2 + 
    \psi_k'(y)^2 \right)\diff y 
    \\
    & \lesssim (b-a)\sum_{k=1}^{\infty} u_k^2 \lambda_k^{1+s} = (b-a)\| f\|^2_{\Ws}.
\end{align*}
To prove \eqref{local_beta} we observe that the same argument used in
\eqref{reg_y_1} gives
\[
  \int_a^b y^\beta \left| \psi_k''(y) \right|^2 \diff y \lesssim
  \lambda_k^{2 - \beta/2 -1/2} \left( b^\delta - a^\delta \right),
\]
whence
\[
  \|\ve_{yy}\|^2_{L^2(\C(a,b),y^\alpha)} \lesssim
  \left(b^\delta - a^\delta \right) \sum_{k=1}^\infty f_k^2 \lambda_k^{2 - \beta/2 -1/2-2s}
  \lesssim
  \left( b^\delta - a^\delta \right) \|f\|_{L^2(\Omega)}^2,
\]
because $2-2s-\frac\beta2-\frac12 < 0$.
\end{proof}

\section{Truncation}
\label{sec:Truncation}
The solution $\ve$ of problem \eqref{alpha_harmonic_extension_weak} is defined on the infinite
domain $\C$ and, consequently, it cannot be directly approximated with finite element-like
techniques. In this section we will show that $\ve$ decays sufficiently fast -- in fact
exponentially -- in the extended direction. This suggests 
truncating the cylinder $\C$ to $\C_{\Y}$, for a
suitably defined $\Y$. The exponential decay is the content of the next result.

\begin{proposition}[Exponential decay]
\label{PR:energyTinf}
For every $\Y > 1$, the solution $\ve$ of \eqref{alpha_harmonic_extension_weak}
satisfies
\begin{equation}
\label{energyTinf}
  \|\nabla \ve\|_{L^2(\Omega \times (\Y,\infty),y^{\alpha})} \lesssim e^{-\sqrt{\lambda_1} \Y/2}
  \| f\|_{\Hs'}.
\end{equation}
\end{proposition}
\begin{proof}
Recall that if $u \in \Hs$ has the decomposition $u = \sum_k u_k \varphi_k(x')$, the
solution $\ve \in \HL(\C,y^\alpha)$ to \eqref{alpha_harmonic_extension_weak} has the representation
$\ve = \sum_k u_k \varphi(x')\psi_k(y)$, where the functions $\psi_k$ solve \eqref{psik}.

Consider $s=\srn$. In this case $\psi_k(y) = e^{-\sqrt{\lambda_k}y}$. Using the fact that
$\{\varphi_k\}_{k=1}^\infty$ are  eigenfunctions of Dirichlet Laplacian 
on $\Omega$, orthonormal in $L^2(\Omega)$ and orthogonal in $H_0^1(\Omega)$
with $\|\nabla_{x'} \varphi_k\|_{L^2(\Omega)}= \sqrt{\lambda_k}$, we get
\begin{align*}
  \int_\Y^\infty \hspace{-0.2cm}\int_\Omega |\nabla \ve |^2 =
    \int_\Y^\infty \hspace{-0.2cm} \int_\Omega \left( |\nabla_{x'} \ve|^2 + |\partial_y \ve|^2 \right)=
    \sum_{k=1}^\infty \lambda_k^\sr |u_k|^2 e^{-2\sqrt{\lambda_k}\Y}
    \leq  e^{-2\sqrt{\lambda_1} \Y} \| u \|_{\mathbb{H}^{1/2}(\Omega)}^2.
\end{align*}
Since $\| u \|_{\mathbb{H}^{1/2}(\Omega)} =\| f\|_{{\mathbb{H}^{1/2}(\Omega)}'}$,
this implies \eqref{energyTinf}.

Consider now $s\in (0,1)\setminus\{\sr\}$ and
$\psi_k(y) = c_s \left(\sqrt{\lambda_k}y\right)^s K_s(\sqrt{\lambda_k} y)$.
To be able to argue as before, we need the 
estimates on $K_s$ and its derivative for sufficiently large arguments
discussed in \S~\ref{subsub:asymptotics}. In fact, using 
\eqref{y-psi_k-psi_kp} and \eqref{besselenergy}, we obtain
% $y^{\min\{s,1/2\}}e^{y}K_s(y)$ is decreasing for all $y >0$ (see \cite{MS:01}).
% In addition, we use property \eqref{1derivative}, which reads
% $ (z^s K_s(z) )'= -z^s K_{1-s}(z)$.
% As a consequence, 
% we have that for all $y>0$ and $s \in (0,1)$
% \begin{equation}
% \label{eq:prodestimate}
%  | y^\alpha \psi_k(y) \psi_k'(y) | \lesssim \lambda_k^s e^{-\sqrt{\lambda_k}y},
% \end{equation}
% uniformly on $k$.
% With this estimate at hand,
\begin{align*}
  \int_\Y^\infty \int_\Omega y^\alpha |\nabla \ve|^2  &=
  \int_\Y^\infty y^\alpha \int_\Omega \left( |\nabla_{x'} \ve|^2 + |\partial_y \ve|^2 \right) \\
  &= \sum_{k=1}^\infty |u_k|^2 \int_\Y^\infty y^\alpha
    \left( \lambda_k \psi_k(y)^2 + \psi_k'(y)^2 \right) \diff y \\
  &= \left. \sum_{k=1}^\infty |u_k|^2 y^\alpha \psi_k(y)\psi_k'(y) \right|_\Y^\infty
  \lesssim e^{-\sqrt{\lambda_1}\Y} \| u \|^2_{\Hs}.
\end{align*}
% where to obtain the third equality we used that $\psi_k$ solves \eqref{psik} and
% the last inequality holds by \eqref{eq:prodestimate}. 
Again, since
$\| u \|_{\Hs} = \| f\|_{\Hs'}$ we get \eqref{energyTinf}.
\end{proof}

Expression \eqref{energyTinf} motivates the approximation of
$\ve$ by a function $v$ that solves
\begin{equation}
  \label{alpha_harmonic_extension_T}
  \begin{dcases}
    \DIV(y^{\alpha} \nabla v) = 0, & \text{in } \C_\Y, \\
    v = 0, & \text{on } \partial_L \C_\Y \cup \Omega \times \{\Y\}, \\
    \frac{\partial v}{\partial \nu^{\alpha}} = d_s f, & \text{on } \Omega \times \{0\},
  \end{dcases}
\end{equation}
with $\Y$ sufficiently large. Problem \eqref{alpha_harmonic_extension_T} is understood
in the weak sense, \ie we define the space
\[
  \HL(\C_\Y,y^{\alpha}) = \left\{ v \in H^1(\C,y^\alpha): v = 0 \text{ on }
    \partial_L \C_\Y \cup \Omega \times \{ \Y\}\right \},
\]
and seek for
$v \in \HL(\C_\Y,y^{\alpha})$ such that
\begin{equation}
\label{alpha_harmonic_extension_weak_T}
  \int_{\C_\Y} y^\alpha \nabla v \cdot \nabla \phi = d_s \langle f, \tr \phi \rangle,
    \quad \forall \phi \in \HL(\C_\Y,y^{\alpha}).
\end{equation}
Existence and uniqueness of $v$ follows from the Lax-Milgram lemma.

\begin{remark}[Zero extension]
\label{re:extension} \rm
For every $\Y>0$ we have the embedding
\begin{equation}
\label{inclusion_remark}
  \HL(\C_\Y,y^{\alpha}) \hookrightarrow \HL(\C,y^{\alpha}).
\end{equation}
To see this, it suffices to consider the extension by zero for $y> \Y$.
\end{remark}

The next result shows the approximation properties of $v$, solution of
\eqref{alpha_harmonic_extension_weak_T} in $\C_\Y$.
\begin{lemma}[Exponential convergence in $\Y$]
\label{lemma:v-v^T}
For any positive $\Y > 1$, we have
\begin{equation}
\label{le:v-v^T}
  \| \nabla(\ve - v) \|_{L^2(\C_\Y, y^{\alpha})} \lesssim e^{-\sqrt{\lambda_1} \Y/4} \| f\|_{\Hs'}.
\end{equation}
\end{lemma}
\begin{proof}
Given $\phi \in \HL(\C_\Y,y^{\alpha})$ denote by $\phi_e$ its extension
by zero to $\C$. By Remark~\ref{re:extension}, $\phi_e \in \HL(\C,y^{\alpha})$.
Take $\phi_e$ and $\phi$ as test functions in \eqref{alpha_harmonic_extension_weak} and
\eqref{alpha_harmonic_extension_weak_T}, respectively. Subtract the resulting expressions
to obtain
\[
  \int_{\C_\Y} y^\alpha( \nabla \ve - \nabla v) \cdot \nabla  \phi = 0
      \quad \forall \phi \in \HL(\C_\Y,y^{\alpha}),
\]
which implies that $v$ is the best approximation of $\ve$ in $\HL(\C_\Y, y^\alpha)$, \ie
\begin{equation}
\label{best_approx}
  \| \nabla(\ve - v)\|_{L^2(\C_\Y, y^{\alpha})} =
  \inf_{\phi \in \HLn(\C_\Y,y^{\alpha})} \| \nabla(\ve - \phi)\|_{L^2(\C_\Y, y^{\alpha})}.
\end{equation}

Let us construct explicitly a function $\phi_0\in \HL(\C_\Y,y^{\alpha})$
to use in \eqref{best_approx}. Define
\begin{equation}
\label{rho}
  \rho(y) =
    \begin{dcases}
      1, & 0 \leq y \leq \Y/2, \\
      \frac{2}{\Y}(\Y-y),&  \Y/2 < y < \Y, \\
      0, & y \geq \Y.
    \end{dcases}
\end{equation}
Notice that $\rho \in W^1_\infty(0,\infty)$, $|\rho(y)| \leq 1$ and $|\rho'(y)| \leq 2/\Y$
for all $y > 0$. Set $\phi_0(x',y)= \ve(x',y)\rho(y)$ for $x' \in \Omega$ and $y>0$.
A straightforward computation shows
\[
  |\nabla \left( (1-\rho) \ve \right)|^2 \leq
  2 \left( |\rho'|^2 |\ve|^2 + (1-\rho)^2 |\nabla \ve|^2\right) \leq
  2 \left( \frac{4}{\Y^2} \ve^2 + |\nabla \ve|^2\right),
\]
so that
\begin{equation}
\label{v-psi0}
  \| \nabla(\ve-\phi_0)\|^2_{L^2(\C_\Y, y^{\alpha})} \leq 2
  \left( \frac{4}{\Y^2}\int_{\Y/2}^\Y \int_{\Omega} y^{\alpha}|\ve|^2 +
  \int_{\Y/2}^\Y \int_{\Omega}y^{\alpha} |\nabla \ve|^2\right).
\end{equation}

To estimate the first term on the right hand side of \eqref{v-psi0} we 
use the Poincar\'e inequality \eqref{Poincare_interval} over a
dyadic partition that covers the interval $[\Y/2,\Y]$ (see 
the derivation of \eqref{Poincare_ineq} in \S~\ref{sub:sub:weighted_spaces}), to obtain
\begin{equation}
 \int_{\Y/2}^\Y y^{\alpha}\int_{\Omega} |\ve|^2  \lesssim 
\int_{\Y/2}^\Y y^{\alpha} \int_{\Omega} |\nabla \ve|^2.
\end{equation}
To bound the second integral in \eqref{v-psi0} we 
use \eqref{besselenergy} as in the proof of
Proposition~\ref{PR:energyTinf}:
\[
  \int_{\Y/2}^\Y y^{\alpha} \int_{\Omega} |\nabla \ve|^2 =
  \left. \sum_{k=1}^{\infty} |u_k|^2 y^\alpha\psi_k(y)\psi_k'(y) \right|_{\Y/2}^{\Y}
  \lesssim e^{- \sqrt{\lambda_1}\Y/2}  \| f\|_{\Hs'}.
\]

Inserting these estimates into \eqref{best_approx} implies \eqref{le:v-v^T}.
\end{proof}

The following result is a direct consequence of Lemma~\ref{lemma:v-v^T}.
\begin{remark}[Stability]
\label{cor:grad_v^T} \rm
Let $\Y \geq 1$, then
\begin{equation}
\label{grad_v^T}
  \| \nabla v\|_{L^2(\C_{\Y}, y^{\alpha})} \lesssim \| f\|_{\Hs'}.
\end{equation}
Indeed, by the triangle inequality
\[
  \| \nabla v\|_{L^2(\C_{\Y}, y^{\alpha})} \leq \| \nabla (v-\ve)\|_{L^2(\C_{\Y}, y^{\alpha})}
  + \| \nabla \ve\|_{L^2(\C_{\Y}, y^{\alpha})}
  \lesssim \left(e^{-\sqrt{\lambda_1}\Y/4} + 1 \right) \| f\|_{\Hs'}.
\] 
\end{remark}
The previous two results allow us to show a full approximation estimate.
\begin{theorem}[Global exponential estimate]
\label{Thm:v-v^T}
Let $\Y > 1$, then
\begin{equation}
\label{le:v-v^TC}
  \| \nabla(\ve - v) \|_{L^2(\C, y^{\alpha})} \lesssim e^{-\sqrt{\lambda_1} \Y/4} \| f\|_{\Hs'}.
\end{equation}
In particular, for every $\epsilon>0$, let
\[
  \Y_0 = \frac{2}{\sqrt{\lambda_1}}\left( \log C + 2\log \frac{1}{\epsilon} \right),
\]
where $C$ depends only on  $s$ and $\Omega$. Then, for $\Y   \geq \max\{ \Y_0,1\}$, we have
\begin{equation}
\label{v-v^Mf}
  \| \nabla(\ve-v)\|_{L^2(\C, y^{\alpha})} \leq \epsilon \| f\|_{\Hs'}.
\end{equation}
\end{theorem}
\begin{proof}
Extending $v$ by zero outside of $\C_\Y$ we obtain
\[
  \| \nabla(\ve-v)\|^2_{L^2(\C, y^{\alpha})} = \| \nabla(\ve-v)\|^2_{L^2(\C_\Y, y^{\alpha})}
  + \| \nabla \ve\|^2_{L^2(\Omega \times (\Y,\infty),y^{\alpha})}.
\]
Hence Lemma~\ref{lemma:v-v^T} and Proposition~\ref{PR:energyTinf} imply
\begin{equation}
\label{v-v^Tglobal}
  \| \nabla(\ve-v)\|^2_{L^2(\C, y^{\alpha})} \leq  C e^{- \sqrt{\lambda_1} \Y/2 }\| f\|^2_{\Hs'}
 \leq  \epsilon^2 \| f\|^2_{\Hs'},
\end{equation}
for all $\Y \geq \max\{\Y_0,1\}$.
\end{proof}

\section{Finite element discretization and interpolation estimates}
\label{sec:errors}
In this section we prove error estimates for a piecewise $\Q_1$ interpolation operator
on anisotropic elements in the extended variable $y$. We consider
elements of the form $T = K \times I$, where $K \subset \R^n$ is an element
isoparametrically equivalent to the unit cube $[0,1]^n$, via a $\Q_1$
mapping and, $I \subset \R$ is an interval. The anisotropic character of the mesh 
$\T_\Y=\{T\}$ will be given by the family of intervals $I$.

The error estimates are derived in the weighted Sobolev spaces $L^2(\C_\Y,y^{\alpha})$ and
$H^1(\C_\Y,y^{\alpha})$, and they are valid under the condition that neighboring elements
have comparable size in the extended $n+1$--dimension (see \cite{DL:05}).
This is a mild assumption that includes general meshes which do not
satisfy the so-called shape-regularity assumption, i.e., mesh refinements for which the
quotient between outer and inner diameter of the elements does not remain bounded
(see \cite[Chapter 4]{BrennerScott}).

Anisotropic or narrow elements are elements with disparate sizes in each direction.
They arise naturally when approximating solutions of problems with a strong directional-dependent
behavior since, using anisotropy, the local mesh size can be adapted to capture
such features.
Examples of this include boundary layers, shocks and edge singularities (see
\cite{DL:05,DLP:12}).
In our problem, anisotropic elements are essential in order to capture the singular/degenerate
behavior of
the solution $\ve$ to problem \eqref{alpha_harmonic_extension_weak} at $y \approx 0^+$
given in \eqref{asymptoticve}. These elements will provide optimal error estimates, which cannot be
obtained using shape-regular elements.

Error estimates for weighted Sobolev spaces have been obtained in several works; see, for instance,
\cite{Apel:99,BBD:06,DL:05}. The type of weight considered in \cite{Apel:99,BBD:06} is related to
the distance to a point or an edge, and the type of quasi-interpolators are modifications of
the well known
Cl\'ement \cite{Clement} and Scott-Zhang \cite{SZ:90} operators. These works are developed
in 3D and 2D respectively, and the analysis developed in \cite{Apel:99} allows
for anisotropy.
Our approach follows the work of Dur\'an and Lombardi \cite{DL:05}, and is based
on a piecewise $\Q_1$ averaged interpolator on anisotropic elements. It allows us to
obtain anisotropic interpolation estimates in the extended variable $y$ and
in weighted Sobolev spaces,
using only that $|y|^{\alpha} \in A_2(\R^{n+1})$, the Muckenhoupt class $A_2$
of Definition \ref{def:Muckenhoupt}.

\subsection{Finite element discretization}
\label{sub:finite_elements}
Let us now describe the discretization of problem \eqref{alpha_harmonic_extension_T}.
To avoid technical difficulties we assume that the boundary
of $\Omega$ is polygonal. The difficulties inherent to curved boundaries
could be handled, for instance, with the methods of \cite{Bernardi}
(see also \cite{MR0282543,MR0657993}).
Let $\T_\Omega = \{K\}$ be a mesh of $\Omega$ made of
isoparametric quadrilaterals $K$ in the sense of Ciarlet \cite{CiarletBook} and Ciarlet and Raviart \cite{CR:72}.
In other words, given $\hat K = [0,1]^n$ and a family of mappings
$\{ \F_K \in \Q_1(\hat K)^n \}$ we have
\begin{equation}
\label{F_K}
  K = \F_{K}(\hat{K})
\end{equation}
and
\[
  \bar\Omega = \bigcup_{K \in \T_{\Omega}} K, \qquad
  |\Omega| = \sum_{K \in \T_{\Omega}} |K|.
\]
The collection of triangulations is denoted by $\Tr_\Omega$.

The mesh $\T_{\Omega}$ is assumed to be
conforming or compatible, \ie the intersection of
any two isoparametric elements $K$ and $K'$ in $\T_{\Omega}$ is either empty or a common lower
dimensional isoparametric element.

In addition, we assume that $\T_\Omega$ is shape regular (\cf \cite[Chapter 4.3]{CiarletBook}).
This means that $\F_K$ can be decomposed as $\F_K = \A_K + \B_K$, where $\A_K$ is affine and $\B_K$ 
is a perturbation map and, if we define $\tilde K = \A_K(\hat K )$, $h_K = \textrm{diam}(\tilde{K})$,
$\rho_K$ as the diameter of the largest sphere inscribed in $\tilde K$ and the shape coefficient of $K$ as
the ratio $\sigma_{K}=h_K/\rho_K$, then the following two conditions are satisfied:
\begin{enumerate}[(a)]
  \item There exists a constant $\sigma_{\Omega} > 1$ such that
  for all $\T_{\Omega} \in \Tr_{\Omega},$
  \[
    \max \left\{ \sigma_K : K \in \T_\Omega \right\} \leq \sigma_\Omega.
  \]
  \item For all $K \in \T_\Omega$ the mapping $\B_K$ is Fr\'echet differentiable and
  \[
    \|D \B_K\|_{L^\infty(\hat K )} = \mathcal  O(h_K^2),
  \]
  for all $K \in \T_\Omega$ and all $\T_\Omega \in \Tr_\Omega$.
\end{enumerate}
As a consequence of these conditions, if $h_K$ is small enough, the mapping $\F_K$ is one-to-one,
its Jacobian $J_{\F_K}$ does not vanish, and
\begin{equation}
  \label{J_F_K}
   J_{\F_K} \lesssim h_K^n, \quad \|D\F_K\|_{L^{\infty}(\hat{K})} \lesssim h_K.
\end{equation}

The set $\Tr_{\Omega}$ is called quasi-uniform if for all $\T_{\Omega} \in \Tr_{\Omega}$,
\[
  \max \left\{ \rho_K: K \in \T_\Omega \right\} \lesssim
  \min \left\{ h_K: K \in \T_\Omega \right\}.
\]
In this case, we define $h_{\T_{\Omega}}= \max_{K \in \T}h_K$.

We define $\T_\Y$ as a triangulation of $\C_\Y$
into cells of the form $T = K \times I$, where $K \in \T_\Omega$,
and $I$ denotes an interval in the extended dimension. Notice that each discretization
of the truncated cylinder $\C_{\Y}$ depends on the truncation parameter $\Y$.
The set of all such triangulations is denoted by $\Tr$.
In order to obtain a global regularity assumption for $\Tr$ we assume the aforementioned
conditions on $\Tr_\Omega$,
besides the following weak regularity condition:
\begin{enumerate}[(c)]
  \item There is a constant $\sigma$ such that,
  for all $\T_\Y \in \Tr$, if $T_1=K_1\times I_1,T_2=K_2\times I_2 \in \T_\Y$
  have nonempty intersection, then
  \[
    \frac{h_{I_1}}{h_{I_2}} \leq \sigma,
  \]
  where $h_I = |I|$.
\end{enumerate}

Notice that the assumptions imposed on $\Tr$ are weaker than the standard shape-regularity
assumptions, since they allow for anisotropy in the extended variable (\cf \cite{DL:05}).
It is also important to notice that, given the Cartesian product structure of the cells
$T\in \T_\Y$, they are isoparametrically equivalent to $\hat T = [0,1]^{n+1}$. We will denote
the corresponding mappings by $\F_T$.
Then,
\[
  \F_T: \hat x = (\hat x',\hat y) \in \hat T \longmapsto
   x = (x',y) = (\F_K(\hat x'),\F_I(\hat y ) ) \in T = K \times I,
\]
where $\F_K$ is the bilinear mapping defined in \eqref{F_K} for $K$
and, if $I = (c,d)$, $\F_I(y) = (y - c)/(d-c)$.
From \eqref{J_F_K}, we immediately conclude that
\begin{equation}
\label{J_F_T}
   J_{\F_T} \lesssim h_K^n h_{I}, \quad \|D \F_T\|_{L^{\infty}(\hat{T})} \lesssim h_T,
\end{equation}
for all elements $T \in \T_{\Y}$ where $h_T = \max\{h_K,h_I\}$.

Given $\T_{\Y} \in \Tr$, we define
the finite element  space $\V(\T_{\Y})$ by
\[
  \V(\T_\Y) = \left\{
            W \in \C^0( \overline{\C_\Y}): W|_T \in \Q_1(T) \ \forall T \in \T_\Y, \
            W|_{\Gamma_D} = 0
          \right\}.
\]
where $\Gamma_D = \partial_L \C_{\Y} \cup \Omega \times \{ \Y\}$ is called the
Dirichlet boundary.
The Galerkin approximation of \eqref{alpha_harmonic_extension_weak_T} is given by
the unique function $V_{\T_{\Y}} \in \V(\T_{\Y})$ such that
\begin{equation}
\label{harmonic_extension_weak}
  \int_{\C_\Y} y^{\alpha}\nabla V_{\T_{\Y}}\cdot \nabla W = d_s \langle f, \textrm{tr}_{\Omega} W \rangle,
  \quad \forall W \in \V(\T_{\Y}).
\end{equation}
Existence and uniqueness of $V_{\T_{\Y}}$ follows from $\V(\T_\Y) \subset \HL(\C_\Y,y^{\alpha})$
and the Lax-Milgram lemma.

We define the space $\U(\T_{\Omega})=\tr \V(\T_{\Y})$, which is nothing more
than a $\Q_1$ finite element space over the mesh $\T_\Omega$.
The finite element approximation of $u \in \Hs$, solution of \eqref{fl=f_bdddom},
is then given by
\begin{equation}
\label{eq:defofU}
  U_{\T_\Omega} = \tr V_{\T_\Y} \in \U(\T_\Omega ),
\end{equation}
and we have the following result.

\begin{theorem}[Energy error estimate]
\label{TH:upper_bound_1}
If $V_{\T_{\Y}} \in \V(\T_\Y)$ solves \eqref{harmonic_extension_weak} and
$U_{\T_{\Omega}} \in \U(\T_{\Omega})$ is defined in \eqref{eq:defofU}, then
\begin{equation}
\label{upper_bound_U}
  \| u - U_{\T_{\Omega}} \|_{\Hs} \lesssim \| \ve - V_{\T_{\Y}} \|_{\HLn(\C,y^\alpha)},
\end{equation}
and
\begin{equation}
\label{upper_bound_V}
  \| \ve - V_{\T_{\Y}} \|_{\HLn(\C,y^\alpha)} \lesssim \epsilon \|f \|_{\Hs'}
        + \| v - V_{\T_{\Y}} \|_{\HLn(\C_{\Y},y^{\alpha})}.
\end{equation}
\end{theorem}
\begin{proof}
Estimate \eqref{upper_bound_U} is just an application of the trace estimate
of Proposition~\ref{H=LM=TR}. Inequality \eqref{upper_bound_V} is  obtained by the triangle
inequality and \eqref{v-v^Mf}.
\end{proof}

By Galerkin orthogonality
\[
  \| v - V_{\T_\Y} \|_{\HLn(\C_\Y,y^{\alpha})} =
      \inf_{W \in \V(\T_{\Y})} \| v - W \|_{\HLn(\C_\Y, y^{\alpha})}.
\]
% then, as it is customary, we will bound the approximation error using $W = \Pi_{\T_{\Y}} v$, where
% $\Pi_{\T_{\Y}}$ is a  piecewise $\Q_1$ averaged interpolation operator
% defined over the mesh $\T_{\Y}$,
% which we describe below. The main originality in our construction is that it allows for
% the study of
% interpolation theory in weighted Sobolev spaces under weak  shape regularity assumptions,
% thus permitting the use of anisotropic elements in the extended dimension. 
% 
Theorem \ref{TH:upper_bound_1} and Galerkin orthogonality imply that the approximation 
estimate \eqref{upper_bound_V} depends on the regularity of $\ve$. To see this we introduce
\begin{equation}
\label{rhon}
  \rho(y) =
    \begin{dcases}
      1, & 0 \leq y < \Y/2, \\
      p,&  \Y/2 \leq y \leq \Y,
    \end{dcases}
\end{equation} 
where $p$ is the unique cubic polynomial on $[\Y/2,\Y]$ defined by the conditions
$p(\Y/2) = 1$, $p(\Y) = 0$, $p'(\Y/2) = 0$ and $p'(\Y) = 0$. Notice that
$\rho \in W_{\infty}^2(0,\Y)$, $|\rho(y)| \leq 1$, $|\rho'(y)| \lesssim 1$  
and $|\rho''(y)| \lesssim 1$. Set $\ve_0(x',y) = \rho(y) \ve(x',y)$
for $x' \in \Omega$ and $y \in [0,\Y]$, and notice that $\ve_0 \in \HL(\C_{\Y},y^{\alpha})$.
With this construction at hand, repeating the arguments used in the proof of
Lemma~\ref{lemma:v-v^T}, we have that
\begin{align*}
  \| \Delta_{x'} \ve_0 \|_{L^2(\C_\Y,y^{\alpha})} &\lesssim 
  \| \Delta_{x'} \ve \|_{L^2(\C_\Y,y^{\alpha})}, \\
% \label{regve_01} \\
  \| \partial_y \nabla_{x'} \ve_0 \|_{L^2(\C_\Y,y^{\alpha})} & \lesssim 
  \| \partial_y \nabla_{x'} \ve \|_{L^2(\C_\Y,y^{\alpha})} + \| f \|_{\Hs'}, \\
% \label{regve_02} \\
  \| \partial_{yy} \ve_0 \|_{L^2(\C_\Y,y^{\beta})} &\lesssim 
  \| \partial_{yy} \ve \|_{L^2(\C_\Y,y^{\beta})} + \| f \|_{\Hs'}.
%  \label{regve_03}
\end{align*}
In addition, if we assume that there is an operator
\[
 \Pi_{\T_\Y}: \HL(\C_\Y,y^\alpha) \rightarrow \V(\T_\Y),
\]
that is stable, i.e.,~$
\| \Pi_{\T_\Y} w \|_{\HLn(\C_\Y,y^\alpha)} \lesssim \| w \|_{\HLn(\C_\Y,y^\alpha)}$, for all
$w \in \HL(\C_\Y,y^\alpha)$, then
the following estimate holds 
\begin{equation}
\label{estimate_ve0}
\| \ve - V_{\T_{\Y}} \|_{\HLn(\C,y^\alpha)} 
\lesssim \epsilon \|f \|_{\Hs'} + 
\| \ve_0 -  \Pi_{\T_\Y}\ve_0  \|_{\HLn(\C_{\Y},y^{\alpha})}.
\end{equation}
To see this, we use \eqref{upper_bound_V}, together with Galerkin orthogonality 
and the stability of the operator $\Pi_{\T_\Y}$, to obtain
\begin{align*}
  \| \ve - V_{\T_{\Y}} \|_{\HLn(\C,y^\alpha)} & \lesssim 
  \epsilon \|f \|_{\Hs'}
  + \| v - \Pi_{\T_\Y}v  \|_{\HLn(\C_{\Y},y^{\alpha})}
\\
   & \lesssim \epsilon \|f \|_{\Hs'} + \| v - \ve_0  \|_{\HLn(\C_{\Y},y^{\alpha})} 
   + \| \ve_0 -  \Pi_{\T_\Y}\ve_0  \|_{\HLn(\C_{\Y},y^{\alpha})}.
\end{align*}
The second term on the right hand side of the previous inequality
is estimated as in Lemma~\ref{lemma:v-v^T}. We leave the details to the reader.

Estimates for $\ve_0 -  \Pi_{\T_\Y}\ve_0 $ on weighted Sobolev spaces are derived in 
\S\ref{sub:intweighted}. Clearly, these depend on the regularity of $\ve_0$ which,
in light of \eqref{estimate_ve0}, depends on the regularity of $\ve$.
For this reason, and to lighten the notation, we shall in the sequel
write $\ve$ and obtain
interpolation error estimates for it, even though $\ve$ does
not vanish at $y=\Y$.

% \AJS{Notice that 
% since $\Y > 1$, the weight $y^\alpha$ 
% is no longer singular nor degenerate, then
% the trace of $\ve$ at $y=\Y$ is well defined as a consequence of the exponential decayment
% of the function $\ve$ and the standard theory
% for Sobolev spaces; see \cite{GT,Evans}.}

\subsection{Interpolation estimates in weighted Sobolev spaces}
\label{sub:intweighted}
Let us begin by introducing some
notation and terminology. Given $\T_\Y$, we call $\N$ the set of its nodes and
$\N_{~\textrm{in}}$ the set of its interior and Neumann nodes. For each vertex
$\vero \in \N$, we write $\vero= (\vero',\vero'')$,
where $\vero'$ corresponds to a node of $\T_\Omega$, and
$\vero''$ corresponds to a node of the discretization of the $n+1$--dimension.
We define $h_{\vero'}= \min\{h_K: \vero' \textrm{ is a vertex of } K\}$, and
$h_{\vero''} = \min\{h_I: \vero'' \textrm{ is a vertex of } I\}$.

Given $\vero \in \N$, the \emph{star} or patch around $\vero$ is
defined as
\[
  \omega_{\vero} = \bigcup_{\vero \in T} T,
\]
and for $T \in \T_{\Y}$ we define its \emph{patch} as
\[
  \omega_T = \bigcup_{\vero \in T} \omega_\vero.
\]

Let $\psi \in C^{\infty}(\R^{n+1})$ be such that $\int \psi = 1$ and
$D:=\supp\psi \subset B_r \times (-r_{\Y},r_{\Y})$, where $B_r$ denotes the
ball in $\R^n$ of radius $r$ and centered at zero, and
$r \leq 1/\sigma_{\Omega}$ and $r_{\Y} \leq 1/\sigma$. 
% whence $\supp \psi \subset B_1 \times (-1,1)$.
For $\vero \in \N_{~\textrm{in}}$, we rescale $\psi$ as
\[
  \psi_{\vero}(x) =  \frac1{ h_{\vero'}^n h_{\vero''} } \psi
      \left(\frac{\vero'-x'}{h_{\vero'}}, \frac{\vero''-y}{h_{\vero''}} \right),
\]
and note that $\supp \psi_{\vero} \subset \omega_{\vero}$ for
all $\vero\in\N_{~\textrm{in}}$.

Given a function $w \in L^2(\C_\Y,y^{\alpha})$ and a node $\vero$ in $\N_{~\textrm{in}}$
we define, following Dur\'an and Lombardi \cite{DL:05}, the regularized Taylor polynomial of first degree 
of $w$ about $\vero$ as
\begin{equation}
\label{p1average}
  w_{\vero}(z) = \int P(x,z) \psi_{\vero}(x) \diff x 
  = \int_{\omega_{\vero}} P(x,z) \psi_{\vero}(x) \diff x,
\end{equation}
where $P$ denotes the Taylor polynomial of degree 1 in the variable $z$
of the function $w$ about the point $x$, \ie
\begin{equation}
\label{Ptaylor}
  P(x,z) = w(x) + \nabla w(x) \cdot (z - x).
\end{equation}
As a consequence of Remark~\ref{RK:L1} and the fact that the averaged 
Taylor polynomial is defined for functions in $L^1(\C_\Y)$ (\cf \cite[Proposition 4.1.12]{BrennerScott}),
we conclude that
$P$ is well defined for
any function in $L^2(\C_\Y,y^{\alpha})$.

We define the average $\Q_1$ interpolant $\Pi_{\T_{\Y}} w$, as the unique piecewise $\Q_1$ function such
that $\Pi_{\T_{\Y}} w(\vero) = 0$ if $\vero$ lies on the Dirichlet boundary $\Gamma_D$ and
$\Pi_{\T_{\Y}} w(\vero) = w_{\vero}(\vero)$ if $\vero \in \N_{~\textrm{in}}$.
If $\lambda_{\vero}$ denotes the Lagrange basis function
associated with node $\vero$, then
\[
  \Pi_{\T_{\Y}} w = \sum_{\vero \in \N_{\textrm{in}}} w_{\vero}(\vero) \lambda_{\vero}.
\]

There are two principal reasons to consider average interpolation.
First, we are interested in the approximation of singular functions and thus
Lagrange interpolation cannot be used since point-wise values become meaningless.
In fact, this motivated the introduction of average interpolation (see \cite{Clement,SZ:90}).
In addition, average interpolation has better approximation
properties when narrow elements are used (see \cite{Acosta}).

%Shape regularity shows that $\supp\psi_{\vero} \subset \omega_{\vero}$,
%so that the integral appearing in the definition of the regularized
%average $w_{\vero}$ can be written over the set $\omega_{\vero}$.

Finally, for $\vero \in \N_{\textrm{in}}$, we define the weighted regularized average of $w$ as
\begin{equation}
\label{p0average}
  Q_{\vero}w = \int w(x) \psi_{\vero}(x) \diff x 
  = \int_{\omega_{\vero}} w(x) \psi_{\vero}(x) \diff x.
\end{equation}

\subsubsection{Weighted Poincar\'e inequality}
\label{sub:sub:poincare}

In order to obtain interpolation error estimates in $L^2(\C_\Y,y^{\alpha})$ and
$H^1(\C_\Y,y^{\alpha})$, it is instrumental to have a weighted Poincar\'e-type inequality.
Weighted Poincar\'e inequalities are particularly pertinent in the study
of the nonlinear potential theory of degenerate elliptic equations, see \cite{FKS:82,HKM}.
If the domain is a ball and the weight belongs to $A_p$, with $1 \leq p < \infty$,
this result can be found in \cite[Theorem~1.3 and Theorem~1.5]{FKS:82}.
However, to the best of our knowledge, such a result is not available in the
literature for more general domains. For our specific weight
we present here a constructive proof, \ie not based on a compactness argument. This
allows us to study the dependence of the constant on the domain.

\begin{lemma}[Weighted Poincar\'e inequality I]
\label{le:Poincareweighted}
Let $\omega \subset \R^{n+1}$ be bounded, star-shaped with
respect to a ball $B$, and $\diam \omega \approx 1$.
Let $\chi \in C^0\left(\bar\omega\right)$ with
$\int_{\omega} \chi = 1$, and
$\xa(y) := \left|a|y| + b\right|^{\alpha}$ for $a,b\in\R$.
If $w \in H^1(\omega,\xa(y))$ is such that
$\int_{\omega} \chi w = 0$, then
\begin{equation}
\label{Poincareweighted}
  \|w \|_{L^2(\omega,\xa)}
  \lesssim \| \nabla w \|_{L^2(\omega,\xa)},
\end{equation}
where the hidden constant depends only on $\chi$, $\alpha$ and the
radius $r$ of $B$, but is independent of both $a$ and $b$.
\end{lemma}
\begin{proof}
The fact that $\alpha \in (-1,1)$ implies $\xa \in A_2(\R^{n+1})$
with a Muckenhoupt constant $C_{2,\xa}$ in \eqref{A_pclass}
uniform in both $a$ and $b$. Define
\[
  \widetilde{w} = \xa w - \left(\int_{\omega} \xa w \right)\chi.
\]
Clearly $\widetilde{w} \in L^1(\omega)$ and it has vanishing mean value by
construction.

Since $\int_{\omega} \chi w = 0$ we obtain
\begin{equation}
\label{aux1:1}
  \| w \|^2_{L^2(\omega,\xa)} = \int_{\omega} w \widetilde{w} +
  \left(\int_{\omega} \xa w \right)\int_{\omega} \chi w =
  \int_{\omega} w \widetilde{w}.
\end{equation}
Consequently, given that $\omega$ is star shaped with respect to
$\hat{B}$, and $\xa\in A_2(\R^{n+1})$, there exists
$F \in H_0^1(\omega,\xa)^{n+1}$ such that $- \DIV F = \widetilde{w}$, and
\begin{equation}
\label{Fdivweighted}
  \| F\|_{H_0^1(\omega,\xa^{-1})^{n+1}} \lesssim \| \widetilde{w} \|_{L^2(\omega,\xa^{-1})},
\end{equation}
where the hidden constant in \eqref{Fdivweighted} depends on $r$ and the constant
$C_{2,\xa}$ from Definition~\ref{def:Muckenhoupt}
\cite[Theorem~3.1]{DL:10}. 

Replacing $\widetilde{w}$ by $-\DIV F$ in \eqref{aux1:1},
integrating by parts and using \eqref{Fdivweighted}, we get
\begin{equation}
  \| w \|^2_{L^2(\omega,\xa)} =
  -\int_{\omega} w \,\DIV F = \int_{\omega} \nabla w \cdot F
  \lesssim \|\nabla w \|_{L^2(\omega,\xa)} \| \widetilde{w} \|_{L^2(\omega,\xa^{-1})}.
\label{aux1:2}
\end{equation}
To estimate $\| \widetilde{w} \|_{L^2(\omega,\xa^{-1})}$ we use the 
Cauchy-Schwarz inequality 
and the constant $C_{2,\xi_\alpha}$ from Definition~\ref{def:Muckenhoupt} as follows:
\[
  \|\widetilde{w} \|_{L^2(\omega,\xa^{-1})}^2  \leq
  2 \left( 1 + \int_{\omega} \xa \int_{\omega} \chi^2 \xa^{-1} \right)
  \| w \|^2_{L^2(\omega,\xa)}
  \lesssim \| w \|^2_{L^2(\omega,\xa)}.
\]
Inserting the inequality above into \eqref{aux1:2}, we obtain \eqref{Poincareweighted}.
\end{proof}

We need a slightly more general form of the Poincar\'e inequality for
the applications below. We now relax the geometric assumption on the
domain $\omega$ and let the vanishing mean property hold just in a
subdomain.

\begin{corollary}[Weighted Poincar\'e inequality II]
\label{C:Poincareweighted-2}
Let $\omega=\cup_{i=1}^N\omega_i\subset\R^{n+1}$ be a connected domain and
each $\omega_i$ be a star-shaped domain with respect to a ball
$B_i$. Let $\chi_i\in C^0(\bar\omega_i)$ and $\xa$ be as in Lemma
\ref{le:Poincareweighted}. If $w\in H^1(\omega,\xa)$ and
$w_i:=\int_{\omega_i} w\chi_i$, then
\begin{equation}
\label{Poincareweighted-2}
  \|w - w_i\|_{L^2(\omega,\xa)}
  \lesssim \| \nabla w \|_{L^2(\omega,\xa)}
\qquad\forall 1\le i\le N,
\end{equation}
where the hidden constant depends on $\{\chi_i\}_{i=1}^N$, $\alpha$, the
radius $r_i$ of $B_i$, and the amount of overlap between the subdomains
$\{\omega_i\}_{i=1}^N$, but is independent of both $a$ and $b$.
\end{corollary}
\begin{proof}
This is a consequence of Lemma \ref{le:Poincareweighted} and 
\cite[Theorem 7.1]{DupontScott}. We sketch the proof here for
completeness. It suffices to deal with two subdomains,
$\omega_1,\omega_2$, and the overlapping region
$B=\omega_1\cap\omega_2$. We observe that
\[
  \|w-w_1\|_{L^2(\omega_2,\xa)} \le \|w-w_2\|_{L^2(\omega_2,\xa)}
  + \|w_1-w_2\|_{L^2(\omega_2,\xa)},
\]
together with
$
  \|w_1-w_2\|_{L^2(\omega_2,\xa)} =
    \left( \frac{\int_{\omega_2}\xa}{\int_B\xa} \right)^{1/2} 
    \|w_1-w_2\|_{L^2(B,\xa)}
$
and
\[
\|w_1-w_2\|_{L^2(B,\xa)}  
\lesssim \|w-w_1\|_{L^2(\omega_1,\xa)} + \|w-w_2\|_{L^2(\omega_2,\xa)},
\] 
imply $\|w-w_1\|_{L^2(\omega_2,\xa)} \lesssim \|\nabla w\|_{L^2({\omega_1 \cup \omega_2},\xa)}$.
This, combined with \eqref{Poincareweighted}, gives 
\eqref{Poincareweighted-2} for $i=1$
with a stability constant depending on the ratio $\frac{\int_{\omega_2}\xa}{\int_B\xa}$.
\end{proof}

\subsubsection{Weighted $L^2$-based interpolation estimates}
\label{sub:sub:L2_estimates}
Owing to the weighted Poincar\'e inequality 
of Corollary \ref{C:Poincareweighted-2},
we can adapt the proof of \cite[Lemma~2.3]{DL:05} to obtain interpolation estimates in the
weighted $L^2$-norm. These estimates allow a disparate mesh size on the extended direction, relative to
the coordinate directions $x_i$, $i=1,\dots,n,$ which may in turn be graded.
This is the principal difference with \cite[Lemma~2.3]{DL:05} where, however, the domain must
be a cube.

\begin{lemma}[Weighted $L^2$-based interpolation estimates]
\label{LM:approximation}
Let $\verot \in \N_{~\emph{in}}$. Then, for all $w \in H^1(\omega_{\verot},y^{\alpha})$, we have
\begin{equation}
\label{v-Q_vL2}
  \|w - {Q_{\verot}w} \|_{L^2(\omega_\verot,y^{\alpha})}
  \lesssim   h_{\verot'}  \| \nabla_{x'} w\|_{L^2(\omega_\verot,y^{\alpha})} +
  h_{\verot''}\| \partial_{y} w\|_{L^2(\omega_\verot,y^{\alpha})},
\end{equation}
and, for all $v \in H^2(\omega_{\verot},y^{\alpha})$ and $j=1,\dots,n+1$, we have
\begin{equation}
\label{v-v_vH1}
  \|\partial_{x_j}(w - w_{\verot}) \|_{L^2(\omega_\verot,y^{\alpha})}
	\lesssim h_{\verot'} \sum_{i=1}^n\| \partial^2_{x_j x_i} w\|_{L^2(\omega_\verot,y^{\alpha})}
  + h_{\verot''}\| \partial^2_{x_j y} w\|_{L^2(\omega_\verot,y^{\alpha})},
  \end{equation}
where, in both inequalities, the hidden constant depends only on
$\alpha$, $\sigma_{\Omega}$, $\sigma$ and $\psi$.
\end{lemma}
\begin{proof}
Define by $\Fm_\vero: (x',y) \rightarrow (\bar{x}',\bar{y}) $ the
scaling map
\[
  \bar{x}' = \frac{\vero'-x'}{h_{\vero'}}, \qquad  \bar{y} = \frac{\vero''-y}{h_{\vero''}},
\]
along with $\overline{\omega}_\vero = \Fm_\vero(\omega_\vero)$ and $\bar{w}(\bar x) = w(x)$.
Define also
$
  \bar{Q} \bar{w}= \int\bar{w} \psi,
$
where $\psi$ has been introduced in Section \ref{sub:intweighted}.
Since $\supp \psi \subset \overline{\omega}_{\vero}$
integration takes place only over $\overline\omega_\vero$, and $\int_{\overline\omega_\vero} \psi =1$. 
Then, $\bar{Q} \bar{w}$ satisfies
$
  \bar{Q} \bar{w} = \int_{\overline{\omega}_{\vero}}\bar{w} \psi =
  \int_{\omega_{\vero}} w \psi_{\vero} = Q_{\vero} w,
$
and
\begin{equation}
\label{int=0}
  \int_{\overline{\omega}_{\vero}} (\bar{Q} \bar{w} - \overline{w} )\psi \diff \bar{x}=
  \bar{Q} \bar{w} - \int_{\overline{\omega}_{\vero}}\bar{w} \psi \diff \bar{x}= 0.
\end{equation}
Simple scaling, using the definition of the mapping $\Fm_{\vero}$, yields
\begin{equation}
\label{changevar1}
  \int_{\omega_{\vero}} y^{\alpha} | w - Q_{\vero} w |^2 \diff x=
  h_{\vero'}^n  h_{\vero''}
  \int_{\overline{\omega}_\vero} \xa | \bar{w} - \bar{Q} \bar{w} |^2 \diff \bar{x},
\end{equation}
where $\xa(y) := | \vero'' - \bar{y}h_{\vero''} |^{\alpha}$.
By shape regularity, the mesh sizes $h_{\vero'}, h_{\vero''}$ satisfy
$1/{2\sigma} \leq h_{\bar{\vero}''} \leq 2\sigma$
and $1/2\sigma_{\Omega} \leq h_{\bar{\vero}'} \leq 2\sigma_{\Omega}$, respectively,
and $\diam \overline{\omega}_{\vero}\approx 1$.
In view of \eqref{int=0}, we can apply Lemma~\ref{le:Poincareweighted} with
% $\delta = 2\max\{\sigma_{\Omega},\sigma\}$, 
the weight $\xa$ and $\chi = \psi$, to $\omega = \overline{\omega}_{\vero}$ to obtain
\[
  \| \bar{w} - \bar{Q} \bar{w} \|_{L^2 \left(\bar{\omega}_\vero, \xa \right)} \lesssim
  \| \bar\nabla \bar{w} \|_{L^2 \left(\bar{\omega}_\vero, \xa \right)},
\]
where the hidden constant depends only on $\alpha$, $\sigma_{\Omega}$, 
$\sigma$ and $\psi$, but not on $\vero''$ and $h_{\vero''}$.
Applying this to \eqref{changevar1}, together with a change of variables with 
$\Fm_{\vero}^{-1}$, we get \eqref{v-Q_vL2}. 

The proof of \eqref{v-v_vH1} is similar. Notice that
\begin{align*}
  w_{\vero}(z) &= \int_{\omega_{\vero}} \left(w(x) + \nabla w(x)\cdot
  (z - x ) \right) \psi_{\vero}(x) \diff x \\
  &=
  \int_{\overline{\omega}_{\vero}}\left(\bar{w}(\bar{x}) + \bar{\nabla}\bar{w}(\bar{x})\cdot
      (\bar{z} - \bar{x} ) \right)\psi(\bar{x}) \diff \bar x
  =: \bar{w}_0(\bar{z}).
\end{align*}
Since
$
  \partial_{\bar{z}_i} \bar{w}_0(\bar{z}) =
  \int_{\overline{\omega}_{\vero}} \partial_{\bar{x}_i} \bar{w}(\bar{x})\psi(\bar{x}) \diff \bar x
$
is constant, we have the vanishing mean value property
\begin{align*}
  \int_{\overline{\omega}_{\vero}} \partial_{\bar{z}_i}
      \left( \bar{w}(\bar z)-\bar{w}_0(\bar z) \right) \psi(\bar z) \diff \bar z = 0.
%   \int_{\bar{\omega}_{\vero}} \partial_{\bar{z}_i} \bar{w}(\bar z) \psi(\bar z) \diff \bar z \\
%   &- 
%   \int_{\bar{\omega}_{\vero}} \partial_{\bar{x}_i} \bar{w}(\bar{x})\psi(\bar{x}) \diff \bar x
%   \int_{\bar{\omega}_{\vero}} \psi (\bar{z}) \diff \bar z= 0,
\end{align*}
Finally, applying Lemma~\ref{le:Poincareweighted} to
$\partial_{\bar{x}_i}\left(\bar{w}(\bar{x}) - \bar{w}_0(\bar{x}) \right)$,
and scaling back via the map $\Fm_{\vero}$, we obtain \eqref{v-v_vH1}.
\end{proof}

By shape regularity, for all $\vero \in \N_{~\textrm{in}}$ and $T \subset\omega_\vero$, the
quantities $h_{\vero'}$ and $h_{\vero''}$ are equivalent to $h_K$ and $h_I$,
up to a constant that depends
only on $\sigma_{\Omega}$ and $\sigma$, respectively. This fact leads to
interpolation estimates in the weighted $L^2$-norm.

\begin{theorem}
[Stability and local interpolation estimates in the weighted $L^2$-norm]
\label{TH:v - PivL2}
For all $T \in \T_{\Y}$ and $w \in L^2(\omega_{T},y^{\alpha})$ we have
\begin{equation}
\label{PiL2bounded}
  \|  \Pi_{\T_{\Y}} w \|_{L^2(T,y^{\alpha})} \lesssim \| w \|_{L^2(\omega_T,y^{\alpha})}.
\end{equation}
If, in addition, $w \in H^1(\omega_{T},y^{\alpha})$
\begin{equation}
\label{v - PivL2}
  \| w - \Pi_{\T_{\Y}} w \|_{L^2(T,y^{\alpha})} \lesssim h_{\verot'}  
      \| \nabla_{x'} w\|_{L^2(\omega_T,y^{\alpha})} + 
     h_{\verot''}\| \partial_{y} w \|_{L^2(\omega_T,y^{\alpha})}.
\end{equation}
The hidden constants in both inequalities depend only on
$\sigma_{\Omega}$, $\sigma$, $\psi$ and $\alpha$.
\end{theorem}
\begin{proof}
Let $T$ be an element of $\T_{\Y}$. Assume, for the moment,
that $ \Pi_{\T_{\Y}}$ is uniformly bounded 
as a mapping from $L^2(\omega_T,y^{\alpha})$ to $L^2(T,y^{\alpha})$, \ie \eqref{PiL2bounded}.

Choose an interior node $\vero$ of $T$, \ie a node $\vero$ of $T$ 
such that $\vero\in \N_{~\textrm{in}}$. Since $Q_{\vero}w$ is constant,
we deduce $ \Pi_{\T_{\Y}} Q_{\vero} w = Q_{\vero} w$, whence
\[
  \| w -  \Pi_{\T_{\Y}} w \|_{L^2(T,y^{\alpha})} =
  \| (I- \Pi_{\T_{\Y}})(w - Q_{\vero} w) \|_{L^2(T,y^{\alpha})}
  \lesssim \| w - Q_{\vero} w \|_{L^2(\omega_T,y^{\alpha})},
\]
so that \eqref{v - PivL2} follows from Corollary \ref{C:Poincareweighted-2}.

It remains to show the local boundedness \eqref{PiL2bounded} of $\Pi_{\T_{\Y}}$. By definition,
\[
   \Pi_{\T_{\Y}} w = \sum_{i=1}^{n_T} w_{\vero_i}(\vero_i) \lambda_{\vero_i},
\]
where $\{ \vero_{i} \}_{i=1}^{n_T}$ denotes the set of interior vertices of $T$.
By the triangle inequality
\begin{equation}
\label{PiL2aux}
  \| \Pi_{\T_{\Y}} w \|_{L^2(T,y^{\alpha})} \leq  \sum_{i=1}^{n_T} \| w_{\vero_i} \|_{L^{\infty}(T)}
      \| \lambda_{\vero_i}\|_{L^2(T,y^{\alpha})},
\end{equation}
so that we need to estimate $\| w_{\vero_i} \|_{L^{\infty}(T)}$. This follows from
\eqref{p1average} along with,
\begin{equation}
\label{L2aux1}
  \left| \int_{\omega_{\vero_i}} w \psi_{\vero_i} \right|
  \leq \| w \|_{L^2(\omega_{\vero_i},y^{\alpha})} \|
  \psi_{\vero_i}\|_{L^2(\omega_{\vero_i},y^{-\alpha})},
\end{equation}
and, for $\ell = 1,\ldots,n+1$,
\begin{equation}
\label{L2aux2}
  \left| \int_{\omega_{\vero_i}} \partial_{x_\ell}w(x) (z_\ell-x_\ell) \psi_{\vero_i}(x) \diff x \right|
  \lesssim  \| w \|_{L^2(\omega_{\vero_i},y^{\alpha})}
  \| \psi_{\vero_i}\|_{L^2(\omega_{\vero_i},y^{-\alpha})}.
\end{equation}
We get (\ref{L2aux2}) upon integration by parts, 
$\psi_{\vero_i} =0$ on $\partial \omega_{\vero_i}$, and
$|z_\ell - x_\ell| \lesssim h_{K} \approx h_{\vero'}$ for $\ell=1,\cdots,n$ and
$|z_{n+1} - y| \lesssim h_I \approx h_{\vero''}$.
Replacing \eqref{L2aux1} and \eqref{L2aux2} in \eqref{PiL2aux}, we arrive at
\[
  \| \Pi_{\T_{\Y}} w \|_{L^2(T,y^{\alpha})} \lesssim \| w \|_{L^2(\omega_T,y^{\alpha})}
  \sum_{i=1}^{n_T} \| \lambda_{\vero_i}\|_{L^2(T,y^{\alpha})}
  \| \psi_{\vero_i}\|_{L^2(\omega_{\vero_i},y^{-\alpha})}
                                \lesssim \| w\|_{L^2(\omega_T,y^{\alpha})},
\]
where the last inequality is a consequence of 
$\lambda_{\vero_i}$ and $\psi$ being bounded in $L^{\infty}(\omega_T)$, 
\[
\| \lambda_{\vero_i}\|_{L^2(T,y^{\alpha})}
  \| \psi_{\vero_i}\|_{L^2(\omega_{\vero_i},y^{-\alpha})}\lesssim
  |\omega_{\vero_i}|^{-1}\left( \int_{\omega_{\vero_i}} |y|^\alpha \int_{\omega_{\vero_i}}
  |y|^{-\alpha}\right)^{1/2},
\]
together with $|y|^{\alpha} \in A_2(\R^{n+1})$; see \eqref{A_pclass}.
\end{proof}

\subsubsection{Weighted $H^1$-based interpolation estimates on interior elements}
\label{sub:sub:H1_estimates}
Here we prove interpolation estimates on the first derivatives for interior elements.
The, rather technical, proof is an adaption of \cite[Theorem~2.6]{DL:05} to our particular
geometric setting. In contrast to \cite[Theorem~2.6]{DL:05}, we do not have the symmetries of 
a cube. However, exploiting the Cartesian product structure of the elements $T = K \times I$, 
we are capable of handling the anisotropy in the extended variable $y$
for general shape-regular graded meshes $\T_\Y$.
This is the content of the following result.

\begin{theorem}
[Stability and local interpolation: interior elements]
\label{TH:v - PivH1}
Let $T \in \T_{\Y}$ be such that $\partial T \cap \Gamma_D = \emptyset$.
For all $w \in H^2(\omega_{T},y^{\alpha})$ we have the stability bounds
\begin{align}
\label{dxPIcontinuous}
  \|\nabla_{x'} \Pi_{\T_{\Y}} w \|_{L^2(T,y^{\alpha})} &\lesssim 
  \| \nabla_{x'} w \|_{L^2(\omega_T,y^{\alpha})},
\\
\label{dyPIcontinuous}
  \|\partial_{y}\Pi_{\T_{\Y}} w \|_{L^2(T,y^{\alpha})} &\lesssim 
  \| \partial_y w \|_{L^2(\omega_T,y^{\alpha})},
\end{align}
and, for all $w \in H^2(\omega_{T},y^{\alpha})$ and $j=1,\ldots,n+1$ we have the error estimates
\begin{equation}
\label{v - PivH1}
  \|\partial_{x_j}( w - \Pi_{\T_{\Y}} w) \|_{L^2(T,y^{\alpha})} \lesssim 
  h_{\verot'}  \| \nabla_{x'} \partial_{x_j} w \|_{L^2(\omega_T,y^{\alpha})} +
  h_{\verot''}\| \partial_{y}\partial_{x_j} w \|_{L^2(\omega_T,y^{\alpha})}.
\end{equation}
\end{theorem}
\begin{proof}
\begin{figure}
\centering
\includegraphics[width=0.55\textwidth]{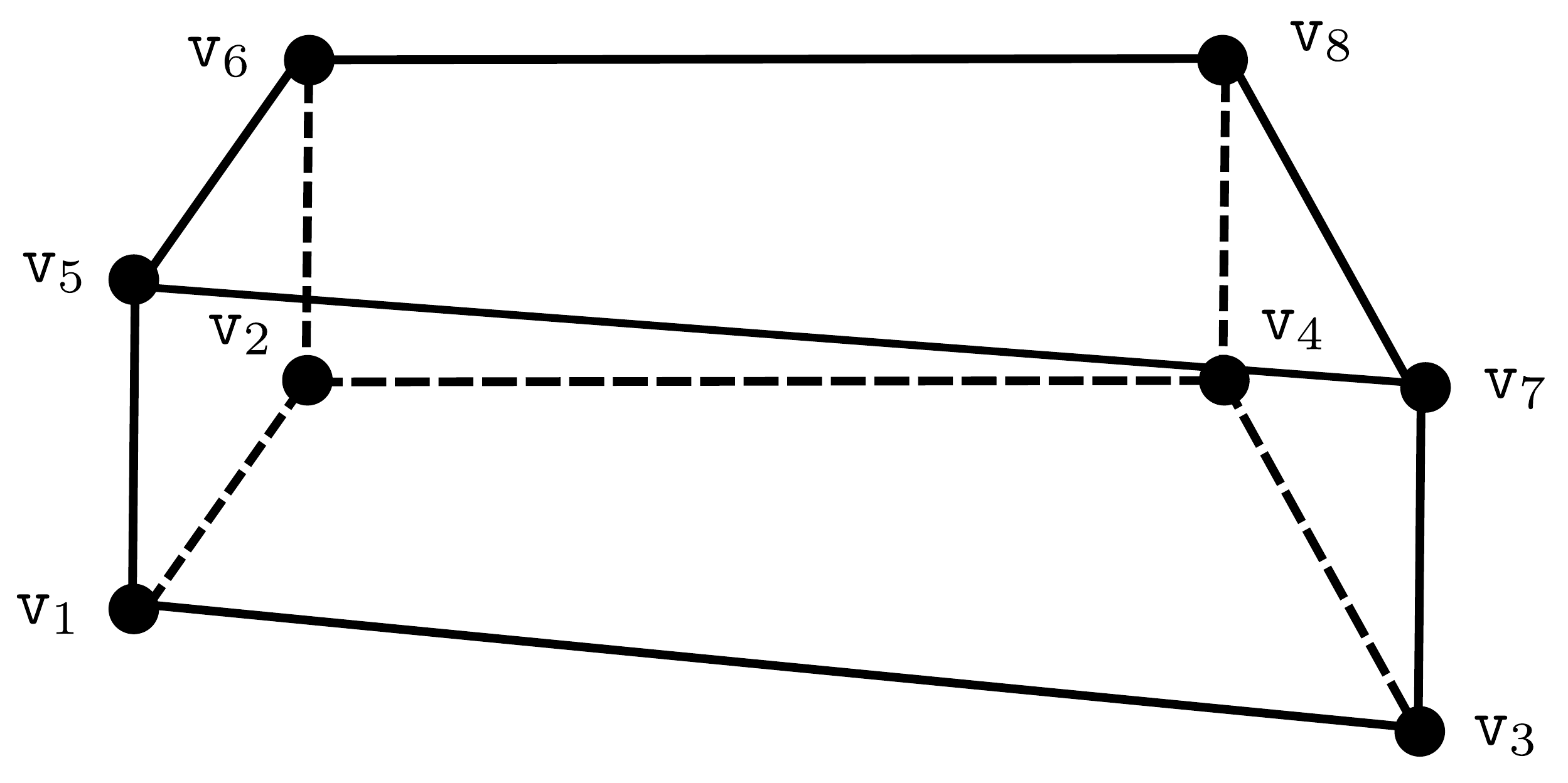}
\caption{A generic element $T= K \times I$ in three dimensions: a quadrilateral prism.}
\label{fig:prism}
\end{figure}
To exploit the particular structure of $T$, we label its vertices
in an appropriate way; see Figure~\ref{fig:prism} for the three-dimensional case.
In general, if $T=K\times[a,b]$, we first assign a numbering $\{\vero_k\}_{k=1,\ldots,2^{n}}$
to the nodes that belong to $K\times\{a\}$. If $(\tilde\vero',b)$ is a vertex in 
$K \times \{b\}$, then
there is a $\vero_k \in K \times \{a\}$ such that $\tilde{\vero}'=\vero_k'$, and we set
$\vero_{k+2^n} = \tilde\vero$. We proceed in three steps.

\noindent \boxed{1} \emph{Derivative $\partial_y$ in the extended dimension.}
We wish to obtain a bound for the norm $\|\partial_{y}( w - \Pi_{\T_{\Y}} w) \|_{L^2(T,y^{\alpha})}$.
Since, $w - \Pi_{\T_{\Y}} w = (w - w_{\vero_1}) + (w_{\vero_1}-\Pi_{\T_{\Y}} w)$
and an estimate for the difference
$w - w_{\vero_1}$
is given in Lemma~\ref{LM:approximation}, it suffices to consider
$q :=  w_{\vero_1}-\Pi_{\T_{\Y}} w \in \Q_1$. Thanks to the special
labeling of the nodes and the tensor product structure of the elements,
\ie $\partial_y \lambda_{\vero_{i+2^n}} = -\partial_y\lambda_{\vero_{i}}$, we get
\[
  \partial_y q = \sum_{i=1}^{2^{n+1}} q(\vero_i) \partial_y \lambda_{\vero_i} =
  \sum_{i=1}^{2^n} ( q(\vero_i) - q(\vero_{i+2^n})) \partial_y \lambda_{\vero_i},
\]
so that
\begin{equation}
\label{partial_yq}
  \| \partial_y q \|_{L^2(T,y^{\alpha})} \leq
  \sum_{i=1}^{2^n} | q(\vero_i) - q(\vero_{i+2^n}) |
  \| \partial_y \lambda_{\vero_i}\|_{ L^2(T,y^\alpha) }.
\end{equation}
To estimate the differences $| q(\vero_i) - q(\vero_{i+2^n}) |$ for $i=1,\cdots,2^n$ we may,
without loss of generality, set $i=1$.
By the definitions of $\Pi_{\T_{\Y}}$ and $q$, we have
$\Pi_{\T_{\Y}}w(\vero_1) = w_{\vero_1}(\vero_1)$, whence
\[
  \delta q(\vero_1):= q(\vero_1) - q(\vero_{1+2^n})
%   = w_{\vero_1}(\vero_1)-\Pi w(\vero_1) - (w_{\vero_1}(\vero_{1+2^n})-\Pi w(\vero_{1+2^n}))
   = w_{\vero_{1+2^n}}(\vero_{1+2^n}) - w_{\vero_1}(\vero_{1+2^n}),
\]
and by the definition \eqref{p1average}
of the averaged Taylor polynomial, we have
\begin{equation}
\label{int:P-P,1}
   \delta q(\vero_1)=
  \int P(x,\vero_{1+2^n}) \psi_{\vero_{1+2^n}}(x) \diff x 
  - \int P(x,\vero_{1+2^n}) \psi_{\vero_{1}}(x) \diff x.
\end{equation}

Recalling the operator $\odot$, introduced in \eqref{eq:defcolon}, we notice that, for
$h_\vero = (h_{\vero'},h_{\vero''})$ and $z\in \R^{n+1}$,
the vector $h_\vero \odot z$ is uniformly equivalent to $(h_K z', h_I z'')$
for all $T = K \times I$ in the star $\omega_{\vero}$.
Changing variables in \eqref{int:P-P,1} yields
\begin{equation}
  \label{int:P-P,2}
  \delta q(\vero_1)=
  \int \left( P(\vero_{1+2^n}- h_{\vero_{1+2^n}} \odot z, \vero_{1+2^n}) -
   P(\vero_1 - h_{\vero_1} \odot z,\vero_{1+2^n}) \right) \psi(z) \diff z.
\end{equation}
To estimate this expression define
\begin{eqnarray}
\label{theta}
  \theta = (0,\theta'') = 
  \left(0,\vero_{1+2^n}''
  - \vero_1''+( {h_{\vero_1''}-h_{ \vero_{1+2^n}'' }} )z'' \right),
\end{eqnarray}
and $F_z(t) = P(\vero_1 - h_{\vero_1} \odot z + t\theta,\vero_{1+2^n})$. Using that
$\vero_1'=\vero_{1+2^n}'$ and $h_{\vero_1'} = h_{\vero_{1+2^n}'}$, 
we easily obtain
\[
  P(\vero_{1+2^n}- h_{\vero_{1+2^n}} \odot z, \vero_{1+2^n})-
  P(\vero_1 - h_{\vero_1} \odot z,\vero_{1+2^n}) = F_z(1) - F_z(0).
\]
Consequently,
\begin{equation}
\label{wIoft}
  {\delta q(\vero_1)} = \int \int_0^1 F_z'(t) {\psi(z)} \diff t \diff z
  = \int_0^1 \int F_z'(t)\psi(z) \diff z \diff t,
\end{equation}
and since $\psi$ is bounded in $L^{\infty}$ and $\supp \psi =D \subset B_1 \times (-1,1)$,
we need to estimate the integral
\[
  I(t) = \int_{D} | F_z'(t) | \diff z, \quad 0 \leq t \leq 1.
\]
Invoking the definitions of $F_{z}$ and $P(x,y)$, we deduce
\begin{equation*}
  F_z'(t) = \nabla_x P(\vero_1 - h_{\vero_1} \odot z + t\theta,\vero_{1+2^n}) \cdot \theta,
\end{equation*}
and
\begin{equation*}
  \nabla_x P(x,\vero) = D^2 w(x)\cdot(\vero - x).
\end{equation*}
Using these two expressions, we arrive at
\begin{align*}
  I(t) &\leq \int_D \left(
          \left| \partial_{yy}^2 w(\vero_1 - h_{\vero_1}\odot z + t\theta) \right|
          \left| \vero_{1+2^n}''- \vero_1'' + {h_{\vero_1''}} z'' - t\theta'' \right| \right. \\
          &+
         \left. \left| \partial_y\nabla_{x'} w(\vero_1 - h_{\vero_1} \odot z + t\theta) \right|
         |\vero_{1+2^n}' - \vero_1' + h_{\vero_1'}z' |
                      \right) |\theta''|\diff z, 
\end{align*}
Now, since $|z'|, |z''| \leq 1$ and $0 \leq t \leq 1$, we see that
\[
  | \vero_{1+2^n}' - \vero_1' + h_{\vero_1'}z' | \lesssim
  h_{\vero_1'}, \qquad | \vero_{1+2^n}'' - \vero_1'' + {h_{\vero_1''}z'' - t\theta'' }|  
  \lesssim {h_{\vero_1''}}.
\]
Consequently,
\begin{multline*}
  I(t) \lesssim
    \int_{D} \left(
      \left| \partial_{yy}^2 w(\vero_1 - h_{\vero_1} \odot z + t\theta) \right| {h_{\vero_1''}^2}
     \right. \\
     \left.
      + \left| \partial_y \nabla_{x'} w (\vero_1 - h_{\vero_1} \odot z + t\theta) \right|
          h_{\vero_1'} {h_{\vero_1''}}
    \right)\diff z.
\end{multline*}
Changing variables, via $\tau = \vero_1 - h_{\vero_1} \odot z + t \theta$, 
we obtain
\begin{equation}
\label{Iaux}
  I(t) \lesssim \int_{\omega_T} \left(
    \frac{h_{\vero_1''}}{ {h_{\vero_1'}^n}}\left|\partial_{yy}^2 w(\tau) \right|
    + \frac{1}{{ h_{\vero_1'}^{n-1} } }\left| \partial_y \nabla_{x'} w(\tau) \right| \right) \diff \tau,
\end{equation}
because the support $D$ of $\psi$ is contained in 
$B_{1/\sigma_{\Omega}} \times (-1/\sigma_{\Y},1/\sigma_{\Y})$,
and so is mapped into $\omega_{\vero_1}\subset\omega_T$. Notice also that
$h_{\vero_1''} \lesssim (1-t)h_{\vero_1''} + th_{\vero_{1+2^n}''}$.
This implies
\begin{equation}
 \label{Ioft}
  I(t) \lesssim \left(
    \frac{h_{\vero_1''}}{{h_{\vero_1'}^n}}\| \partial_{yy}^2 w \|_{L^2(\omega_T,y^{\alpha})} +
    \frac{1}{{ h_{\vero_1'}^{n-1} } }\| \nabla_{x'} \partial_y w \|_{L^2(\omega_T,y^{\alpha})}
                  \right)
  \| 1 \|_{L^2(\omega_T,y^{-\alpha})},
\end{equation}
which, together with \eqref{wIoft}, yields
\begin{equation}
\label{w-w}
  \begin{aligned}
    |\delta q(\vero_1)| \| \partial_y \lambda_{\vero_1} \|_{L^2(T,y^{\alpha})} &\lesssim
    \left(
       \frac{h_{\vero_1''}}{{h_{\vero_1'}^n}} \|\partial_{yy}^2 w \|_{L^2(\omega_T,y^{\alpha})}
      + \frac{1}{{ h_{\vero_1'}^{n-1} } }\| \nabla_{x'} \partial_y w \|_{L^2(\omega_T,y^{\alpha})}
    \right) \\
    &\cdot \| 1 \|_{L^2(\omega_T,y^{-\alpha})}
    \| \partial_y \lambda_{\vero_1}\|_{L^2(T,y^{\alpha})}.
  \end{aligned}
\end{equation}
Since $|y|^\alpha \in A_2(\R^{n+1})$, we have
\begin{equation}
  \| 1 \|_{L^2(\omega_T,y^{-\alpha})} \|\partial_y \lambda_{\vero_1} \|_{L^2(T,y^{\alpha})}
  \lesssim 
  {h_{\vero_1'}^{n}} {\frac{1}{h_{\vero_1''}}}
  \left(\int_{I} y^{-\alpha} \right)^{\srn}
  \left(\int_{I} y^{\alpha} \right)^{\srn}
  \lesssim {h_{\vero_1'}^n}.
\end{equation}
Replacing this into \eqref{w-w}, we obtain
\begin{equation}
  {| \delta q(\vero_1)|} \|\partial_y \lambda_{\vero_1} \|_{L^2(T,y^{\alpha})}
  \lesssim h_{\vero_1}'
  \| \nabla_{x'} \partial_y w \|_{L^2(\omega_T,y^{\alpha})}
  + {h_{\vero_1''}} \|\partial_{yy}^2 w \|_{L^2(\omega_T,y^{\alpha})},
\end{equation}
which, in this case, implies \eqref{v - PivH1}.

\noindent \boxed{2} \emph{Derivatives $\nabla_{x'}$ in the domain $\Omega$.} To prove an estimate for 
$\nabla_{x'}(w-\Pi_{\T_{\Y}} w)$ we notice that,
given a vertex $\vero$, the associated basis function $\lambda_{\vero}$ 
can be written as $\lambda_{\vero}(x) = \Lambda_{\vero'}(x') \mu_{\vero''}(y)$,
where $\Lambda_{\vero'}$ is the canonical $\Q_1$ basis function
on the variable $x'$ associated to the node $\vero'$ in the triangulation $\T_\Omega$,
and $\mu_{\vero''}$ corresponds to the piecewise $\mathbb{P}_1$ basis function 
associated to the node ${\vero''}$.
Recall that, by construction, the basis $\{ \Lambda_i \}_{i=1}^{2^n}$
possesses the so-called partition of unity property, \ie
\[
  \sum_{i=1}^{2^n} \Lambda_{\vero}(x') = 1 \quad \forall x' \in K,
  \qquad \Longrightarrow \qquad
  \sum_{i=1}^{2^n} \nabla_{x'}\Lambda_{\vero}(x') = 0 \quad \forall x' \in K.
\]
This implies that, for every $q\in \Q_1(T)$,
\begin{align*}
  \nabla_{x'}q &= \sum_{i=1}^{2^{n+1}} q(\vero_i) \nabla_{x'}\lambda_{\vero_i}
  = \sum_{i=1}^{2^{n}} \left(
        q(\vero_i)\mu_{\vero_i''}(y) + q(\vero_{i+2^n})\mu_{\vero_{i+2^n}''}(y) 
                        \right) \nabla_{x'}\Lambda_{\vero_i'}(x') \\
  &= \sum_{i=1}^{2^{n}} \left[ (q(\vero_i) - q(\vero_1)) \mu_{\vero_{i}''}(y) +
    ( q(\vero_{i+2^n}) - q(\vero_{1+2^n})) \mu_{\vero_{i+2^n}''}(y)
    \right] \nabla_{x'}\Lambda_{\vero_i}(x'),
\end{align*}
so that
\begin{align*}
  \| \nabla_{x'} q \|_{L^2(T,y^{\alpha})} & \lesssim
  \sum_{i=1}^{2^{n}} |q(\vero_i) - q(\vero_1)| 
      \|\mu_{\vero_{i}''}\nabla_{x'}\Lambda_{\vero_i}\|_{L^2(T,y^{\alpha})} \\
      & + \sum_{i=1}^{2^{n}} |q(\vero_{1+2^n})-q(\vero_{i+2^n})|
      \|\mu_{\vero_{i+2^n}''}\nabla_{x'}\Lambda_{\vero_i}\|_{L^2(T,y^{\alpha})}.
\end{align*}
This expression shows that the same techniques developed for the previous step allows us to 
obtain \eqref{v - PivH1}.

\noindent \boxed{3} \emph{Stability}. 
It remains to prove \eqref{dxPIcontinuous} and \eqref{dyPIcontinuous}. 
By the triangle inequality,
\begin{equation*}
 \| \partial_{y}\Pi_{\T_{\Y}} w\|_{L^2(T,y^{\alpha})} \leq 
 \| \partial_{y}(w- \Pi_{\T_{\Y}} w)\|_{L^2(T,y^{\alpha})} 
   +  \| \partial_y w\|_{L^2(T,y^{\alpha})},
\end{equation*}
so that it suffices to estimate the first term. Add and subtract $w_{\vero_1}$,
\begin{equation}
\label{dyw-Piw}
  \| \partial_{y}(w- \Pi_{\T_{\Y}} w)\|_{L^2(T,y^{\alpha})} \leq
  \| \partial_{y}(w- w_{\vero_1})\|_{L^2(T,y^{\alpha})} 
  + \| \partial_{y}(w_{\vero_1}- \Pi_{\T_{\Y}} w)\|_{L^2(T,y^{\alpha})}.
\end{equation}
Let us estimate the first term. The definition of 
$\psi_{\vero_1}$, together with $|y|^{\alpha} \in A_2(\R^{n+1})$ implies
$
  \| \psi_{\vero_1}\|_{L^2(\omega_{\vero_1},y^{-\alpha})} 
\| 1\|_{L^2(\omega_{\vero_1},y^{\alpha})} \lesssim 1,
$
whence invoking the definition \eqref{p1average} of the regularized 
Taylor polynomial $w_{\vero_1}$ yields
\[
  \|\partial_y w_{\vero_1}\|_{L^2(T,y^{\alpha})} \leq 
  \|\partial_y w\|_{L^2(\omega_{\vero_1},y^{\alpha})},
\]
and
\begin{equation}
\label{w_v1H1,2}
 \|\partial_{y}(w-w_{\vero_1}) \|_{L^2(T,y^{\alpha})}
\lesssim { \|\partial_{y}w \|_{L^2(T,y^{\alpha})}.}
\end{equation}

To estimate the second term of the right hand side of
\eqref{dyw-Piw}, we repeat the steps used to obtain \eqref{v - PivH1},
starting from \eqref{int:P-P,1}.
Integrating by parts and using that 
$\psi_{\vero_i} =0$ on $\partial \omega_{\vero_i}$, we get,
for $\ell = 1,\ldots,n+1$,
\begin{multline*}
   \int_{\omega_{\vero_i}} \partial_{x_\ell}w(x) (z_\ell-x_\ell) \psi_{\vero_i}(x) \diff x 
    = \int_{\omega_{\vero_i}} w(x) \psi_{\vero_i}(x)\diff x \\
    - \int_{\omega_{\vero_i}} w(x) (z_\ell-x_\ell) \partial_{x_\ell} \psi_{\vero_i}(x)\diff x,
\end{multline*}
whence
\begin{equation}
\label{I_1+I_2}
  \begin{aligned}
    \delta q(\vero_1) & = (n+2)\left(\int w(x) \psi_{\vero_{1+2^n}}\diff x - 
    \int w(x) \psi_{\vero_{1}}\diff x\right) \\
    & - \int w(x)( \vero_{1+2^n}- x)\cdot \nabla \psi_{\vero_{1+2^n}}(x)\diff x
    + \int w(x)( \vero_{1}- x) \cdot \nabla \psi_{\vero_{1}}(x)\diff x  \\
    & = I_1 + I_2.
  \end{aligned}
\end{equation}
To estimate $I_1$ we consider the same change of variables used to obtain \eqref{int:P-P,2}.
Define $G_z(t) = (n+2) \cdot w(\vero_1 - h_{\vero_1}\odot z+ t\theta)$, with $\theta$ as in 
\eqref{theta}, and observe that
\[
  I_1 = \int_0^1 \int  G_z'(t)\psi(z) \diff z \diff t
  = (n+2)\int_0^1 \int \partial_{y}w(\vero_1 - h_{\vero_1}\odot z+ t\theta) 
  \theta''\psi(z) \diff z \diff t.
\]
Introducing the change of variables $\tau = \vero_1 - h_{\vero_1} \odot z + t \theta$, we obtain
\begin{equation}
\label{I_1}
| I_1 | \lesssim \int_{\omega_T}  { \frac{1}{h_{\vero_1'}^n} }|\partial_{y}w(\tau)| \diff \tau
\leq  { \frac{1}{h_{\vero_1'}^n} } \|\partial_{y}w \|_{L^2(\omega_T,y^{\alpha})}  
\|1\|_{L^2(\omega_T,y^{-\alpha})}.
\end{equation}
We now estimate $I_2$. Changing variables,
\begin{align*}
  I_2 & = \int \left( 
  w(\vero_{1+2^n} - h_{\vero_{1+2^n}}\odot z) - w(\vero_{1} - h_{\vero_{1}}\odot z) \right) 
  z' \cdot \nabla_{x'} \psi(z) \diff z
  \\
  & + \int \left( w(\vero_{1+2^n} - h_{\vero_{1+2^n}}\odot z)z'' - 
   w(\vero_{1} - h_{\vero_{1}}\odot z)(\vartheta + z'') \right) \partial_y \psi(z) \diff z
  \\
  & = I_{2,1} + I_{2,2},
\end{align*}
where $\vartheta = (\vero_{1+2^n}^{n+1} - \vero_{1}^{n+1})/{ h_{\vero_1''} }$.
Arguing as in the derivation of \eqref{I_1} we obtain
\begin{align}
\label{I_2,1,2}
| I_{2,1} |,  | I_{2,2} |  \lesssim \int_{\omega_T}  
{ \frac{1}{h_{\vero_1'}^n} }|\partial_{y}w(\tau)| \diff \tau
\leq  { \frac{1}{h_{\vero_1'}^n} } \|\partial_{y}w \|_{L^2(\omega_T,y^{\alpha})}  
\|1\|_{L^2(\omega_T,y^{-\alpha})}.
\end{align}
Inserting \eqref{I_1} and \eqref{I_2,1,2} in \eqref{I_1+I_2} we deduce
\[
 |\delta q(\vero_1)| \lesssim { \frac{1}{h_{\vero_1'}^n} } \|\partial_{y}w \|_{L^2(\omega_T,y^{\alpha})}
 \|1\|_{L^2(\omega_T,y^{-\alpha})},
\]
whence
\begin{equation}
\label{deltaqv1}
 |\delta q(\vero_1)|\| \partial_y \lambda_{\vero_1}\|_{ L^2(T,y^\alpha) } \lesssim 
 \|\partial_{y}w \|_{L^2(\omega_T,y^{\alpha})},
\end{equation}
because $h_{\vero'_1}^{-n}  \|\partial_y\lambda_{\vero_1}\|_{L^2(\omega_T,y^\alpha)}
\|1\|_{L^2(\omega_T,y^{-\alpha})}\le C$.
Replacing \eqref{deltaqv1} in \eqref{partial_yq}, we get 
\[
  \| \partial_{y}(w_{\vero_1}- \Pi_{\T_{\Y}} w)\|_{L^2(T,y^{\alpha})} \lesssim 
  \| \partial_y w\|_{L^2(\omega_T,y^{\alpha})},
\]
which, together with \eqref{dyw-Piw} and \eqref{w_v1H1,2}, imply the desired result \eqref{dyPIcontinuous}.
Similar arguments are used to prove the stability bound \eqref{dxPIcontinuous}.
\end{proof}

\subsubsection{Weighted $H^1$-based interpolation estimates on boundary elements}
\label{sub:sub:H1_boundary_estimates}
Let us now extend the interpolation estimates of \S~\ref{sub:sub:H1_estimates}
to elements that intersect the Dirichlet boundary, where the functions to be
approximated vanish. To do so, we adapt the results of \cite[Theorem 3.1]{DL:05} to our particular case.

We consider, as in \cite[Section 3]{DL:05}, different cases according to the 
relative position of the element $T$ in $\T_{\Y}$. We define
the non-overlapping sets
% \begin{align*}
%   \EO{\C_1} &= \left\{ T \in \T_{\Y}: \partial T \cap \Gamma_D = \emptyset \right\}, \\
%   \EO{\C_2} &=  \left\{ T \in \T_{\Y}: \partial T \cap \partial_L \C_{\Y}  \neq \emptyset, \
%   
%             \partial T \cap \left( \bar\Omega \times \{ \Y \} \right) = \emptyset \right\}, \\
%   \EO{\C_3} &=  \left\{ T \in \T_{\Y}: \partial T \cap \partial_L \C_{\Y}  \neq \emptyset, \
%             \partial T \cap \left( \partial \Omega \times \{ \Y \} \right) \neq \emptyset \right\}, \\
%   \C_4 &=  \left\{T \in \T_{\Y}: \partial T \cap \partial_L \C_{\Y}  = \emptyset, \
%             \partial T \cap \left( \partial \Omega \times \{ \Y \} \right) \neq \emptyset \right\}.
% \end{align*}
\begin{align*}
 \C_1 &= \left\{ T \in \T_{\Y}: \partial T \cap \Gamma_D = \emptyset \right\}, \\
  \C_2 &=  \left\{ T \in \T_{\Y}: \partial T \cap \partial_L
    \C_{\Y}  \neq \emptyset \right\}, \\
  \C_3 &=  \left\{T \in \T_{\Y}: 
            \partial T \cap \left( \partial \Omega \times \{ \Y \}
            \right) \neq \emptyset \right\}.
\end{align*}
The elements in $\C_1$ are interior, so the corresponding interpolation estimate is given in 
Theorem~\ref{TH:v - PivH1}.
Interpolation estimates on elements in $\C_3$ are a direct consequence of 
\cite[Theorem 3.1]{DL:05} and Theorem \ref{TH:boundary_estimates} below.
This is so due to the fact that, since $\Y > 1$, the weight $y^\alpha$ 
over $\C_3$ is no longer singular nor degenerate. It remains only to provide interpolation estimates
for elements in $\C_2$.

\begin{theorem}[Local error interpolation estimate: elements in $\C_2$]
\label{TH:boundary_estimates}
Let $T \in \C_{2}$ and $w \in H^1(\omega_T,y^\alpha)$ vanish on
$\partial T \cap \partial_L \C_{\Y}$.
Then, we have the stability bounds
\begin{align}
\label{stab-boundx'}
  \|\nabla_{x'} \Pi_{\T_{\Y}} w \|_{L^2(T,y^{\alpha})} &\lesssim 
  \| \nabla_{x'} w \|_{L^2(\omega_T,y^{\alpha})},
  \\
\label{stab-boundy}
  \|\partial_{y}\Pi_{\T_{\Y}} w \|_{L^2(T,y^{\alpha})} &\lesssim 
  \| \partial_y w \|_{L^2(\omega_T,y^{\alpha})},
\end{align}
If, in addition, $w \in H^2(\omega_T,y^\alpha)$, then,
for $j=1,\ldots,n+1$,
\begin{equation}
\label{C_2_element}
   \| \partial_{x_j}(w - {\Pi_{\T_{\Y}}} w)\|_{L^2(T,y^{\alpha})} \lesssim h_{\verot'}
   \| \partial_{x_j}\nabla_{x'} w \|_{L^2(\omega_T,y^{\alpha})}
  + h_{\verot''}\| {\partial_{x_jy}} w \|_{L^2(\omega_T,y^{\alpha})}.
\end{equation} 
\end{theorem}
\begin{proof}
For simplicity we present the proof in two dimensions. Let $T = (0,a) \times (0,b) \in \C_2$. 
Notice that over such an element the weight
becomes degenerate or singular. Recall the local enumeration of vertices introduced in the proof of
Theorem~\ref{TH:v - PivH1} (see also Figure~\ref{fig:prism}). By
the definition of $\Pi_{\T_{\Y}}$ we have
\begin{equation}\label{Pi-bound}
 \Pi_{\T_{\Y}} w |_{T} = w_{\vero_2}(\vero_2) \lambda_{\vero_2} + 
  w_{\vero_4}(\vero_4) \lambda_{\vero_4},
\end{equation}
The proofs of \eqref{stab-boundx'} and \eqref{stab-boundy} are similar to Step 3 of Theorem
\ref{TH:v - PivH1}. To show \eqref{C_2_element}, we write the local difference between a function and 
its interpolant as
$(w - \Pi_{\T_{\Y}} w) |_{T} = (w - w_{\vero_2})|_{T} + (w_{\vero_2} - \Pi_{\T_{\Y}})|_{T}$. 
Proceeding as in the proof of Lemma~\ref{LM:approximation}, 
we can bound $\partial_{x_j}(w - w_{\vero_2})|_{T}$ for $j=1,2$, in the $L^2(T,y^\alpha)$-norm, by the
right hand side of \eqref{C_2_element}
because this is independent of the trace of $w$. It remains then to derive a bound for
$(w_{\vero_2} - \Pi_{\T_{\Y}}w)|_{T}$, for which we consider two separate cases.

\noindent \boxed{1} \emph{Derivative in the extended direction.} We use
$w_{\vero_2}\in \Q_1$, \eqref{Pi-bound} and $\Pi_{\T_{\Y}}w(\vero_1)=\Pi_{\T_{\Y}}w(\vero_3)=0$,
to write
\[
  \partial_{y}(w_{\vero_2} - \Pi_{\T_{\Y}}w)|_{T} = 
  \left( w_{\vero_2}(\vero_3) - w_{\vero_2}(\vero_1) \right) \partial_y \lambda_{\vero_3} + 
  \left( w_{\vero_2}(\vero_4) - w_{\vero_4}(\vero_4) \right) \partial_y \lambda_{\vero_4}.
\]
Since $w \equiv 0 $ on $\{0\}\times(0,b)$, then $\partial_y w \equiv 0$ on $\{0\}\times(0,b)$.
By the definition of the Taylor polynomial $P$, given in \eqref{Ptaylor},
and the fact that $\vero_1' = \vero_3'$, we obtain
\begin{align*}
  w_{\vero_2}(\vero_3) - w_{\vero_2}(\vero_1) & = (\vero_3'' - \vero_1'')
  \int_{{\omega_T}} \partial_y w(x) \psi_{\vero_2}(x) \diff x \\
  & = (\vero_3'' - \vero_1'')
  \int_{{\omega_T}} \int_{0}^{x'} \partial_{x'y} w(\sigma,y) \psi_{\vero_2}(x',y)
  \diff \sigma \diff x'\diff y.
\end{align*}
Therefore
\begin{align*}
|w_{\vero_2}(\vero_3) - w_{\vero_2}(\vero_1)| & \lesssim h_{\vero_1''}
h_{\vero_1'}\| \partial_{x'y} w \|_{L^2(\omega_T,y^{\alpha})}
\| \psi_{\vero_2} \|_{L^2(\omega_T,y^{-\alpha})}
\\
& \lesssim h_{\vero_1''}
h_{\vero_1'} \frac{h_{\vero_1'}^{\sr}}{h_{\vero_2'} h_{\vero_2''}}
\left( \int_{0}^{b} y^{-\alpha} \diff y\right)^{\sr}
\| \partial_{x'y} w \|_{L^2(\omega_T,y^{\alpha})}.
\end{align*}
Since, in view of the weak shape regularity assumption on the mesh $\T_{\Y}$,
$h_{\vero_1'}\approx h_{\vero_2'}$, $h_{\vero_1''}= h_{\vero_2''}$,  and 
$y^{\alpha} \in A_2(\R_{+}^{n+1})$, we conclude that
\begin{equation}
\label{boundary_aux_1}
  \begin{aligned}
    |w_{\vero_2}(\vero_3) - w_{\vero_2}(\vero_1)|
    \| \partial_y \lambda_{\vero_3} \|_{L^2(T,y^{\alpha})}  & \lesssim 
    \frac{h_{\vero_1'}}{h_{\vero_1''}}
    \left( \int_{0}^{b} y^{-\alpha} \diff y 
    \int_{0}^{b} y^{\alpha} \diff y \right)^{\sr} \times \\
    &\times
    \| \partial_{x'y} w \|_{L^2(\omega_T,y^{\alpha})} \\
     &\lesssim h_{\vero_1'}\| \partial_{x'y} w \|_{L^2(\omega_T,y^{\alpha})}.
  \end{aligned}
\end{equation}
Finally, to bound $w_{\vero_2}(\vero_4) - w_{\vero_4}(\vero_4)$, we proceed as in Step 1 of the proof of
Theorem~\ref{TH:v - PivH1}, which is valid regardless of the trace of
$w$, and deduce
\begin{equation*}
|w_{\vero_2}(\vero_4) - w_{\vero_4}(\vero_4)|\| \partial_y \lambda_{\vero_3} \|_{L^2(T,y^{\alpha})} 
\lesssim h_{\vero_1'}\| \partial_{x'y} w \|_{L^2(\omega_T,y^{\alpha})}
+ h_{\vero_1''}\| \partial_{yy} w \|_{L^2(\omega_T,y^{\alpha})}.
\end{equation*}
This, in conjunction with the previous estimate, yields \eqref{C_2_element}
for the derivative in the extended direction.

\noindent \boxed{2} \emph{Derivative in the $x'$ direction.} To estimate 
$ \partial_{x'}(w_{\vero_2} - \Pi_{\T_{\Y}}w)|_{T}$
we proceed as in Theorem~\ref{TH:v - PivH1} and \cite[Theorem 3.1]{DL:05}, 
but we cannot exploit the symmetry of the tensor product structure now. For brevity, 
we shall only point out the main technical differences. Using, again, that
$(w_{\vero_2} - \Pi_{\T_{\Y}}w) \in \Q_1$, 
\begin{align*}
 \partial_{x'}(w_{\vero_2} - \Pi_{\T_{\Y}}w)|_T & = w_{\vero_2}(\vero_1)\partial_{x'}\lambda_{\vero_1}
+ w_{\vero_2}(\vero_3)\partial_{x'}\lambda_{\vero_3}
+ (w_{\vero_2}(\vero_4) - w_{\vero_4}(\vero_4))\partial_{x'}\lambda_{\vero_4}
\\
& =  w_{\vero_2}(\vero_1)\partial_{x'}\lambda_{\vero_1}
+ (w_{\vero_2}(\vero_4) - w_{\vero_2}(\vero_3))\partial_{x'}\lambda_{\vero_4}
\\
& - (w_{\vero_4}(\vero_4) - w_{\vero_4}(\vero_3))\partial_{x'}\lambda_{\vero_4}
- w_{\vero_4}(\vero_3) \partial_{x'}\lambda_{\vero_4}
\\
& = J(w_{\vero_2},w_{\vero_4}) \partial_{x'}\lambda_{\vero_4} 
+ w_{\vero_2}(\vero_1)\partial_{x'}\lambda_{\vero_1}
- w_{\vero_4}(\vero_3) \partial_{x'}\lambda_{\vero_4},
\end{align*}
where
\[
 J(w_{\vero_2},w_{\vero_4}) = \left( w_{\vero_2}(\vero_4) - w_{\vero_2}(\vero_3)\right)
 - \left( w_{\vero_4}(\vero_4) - w_{\vero_4}(\vero_3) \right).
\] 
Define $\theta=(0,\theta'')=(0,{ \vero_4^2-\vero_2^2 - (h_{\vero_4''}-h_{\vero_2''})z'' })$,
and rewrite $J(w_{\vero_2},w_{\vero_4})$ as follows:
\begin{align*}
J(w_{\vero_2},w_{\vero_4}) & = (\vero_4' - \vero_3') \int_D
 \left( \partial_{x'}w(\vero_2 - h_{\vero_2}\odot z)-
\partial_{x'}w(\vero_4 - h_{\vero_4}\odot z) \right) \psi(z) \diff z 
\\
  & = -  (\vero_4' - \vero_3') \int_D \int_0^1
  \partial_{x'y} w(\vero_2 - h_{\vero_2}\odot z + \theta t) \theta''
  \psi(z) \diff t \diff z,
\end{align*}
where $D=\supp\psi$. Denote
\[
 I(t) =  \int |\partial_{x'y} w(\vero_2 - h_{\vero_2}\odot z + \theta t) \theta'' | \diff z.
\]
Using the change of variables $z \mapsto \tau = \vero_2 - h_{\vero_2}\odot z + \theta t$, results in
\begin{align*}
 | I(t)|  & \lesssim  
 \frac{1}{h_{\vero_2'}} \int_{\omega_T} | \partial_{x'y} w(\tau) | \psi(\tau) \diff \tau
  \lesssim  \frac{1}{h_{\vero_2'}} \| \partial_{x'y} w \|_{L^2(\omega_T,y^{\alpha})}
  \| 1 \|_{L^2(\omega_T,y^{-\alpha})} \\
  & \lesssim {h_{\vero_2'}^{-\sr}} \| \partial_{x'y} w \|_{L^2(\omega_T,y^{\alpha})}
  \left(\int_{0}^{b}y^{-\alpha} \diff y \right)^{\sr},
\end{align*}
whence
$
  \left|J(w_{\vero_2},w_{\vero_4})\right|\lesssim
  h_{\vero_2'}^{\sr} \| \partial_{x'y} w \|_{L^2(\omega_T,y^{\alpha})}
  \left(\int_{0}^{b}y^{-\alpha} \diff y \right)^{\sr}
$. This implies
\begin{align*}
\| {J(w_{\vero_2},w_{\vero_4})} \partial_{x'} \lambda_{\vero_4} \|_{L^2(T,y^{\alpha})}
& \lesssim 
% h_{\vero_2''} \frac{1}{h_{\vero_2''}}
\left(\int_{0}^{b}y^{-\alpha} \diff y \right)^{\sr}
\left(\int_{0}^{b}y^{\alpha} \diff y \right)^{\sr}
\| \partial_{x'y} w \|_{L^2(\omega_T,y^{\alpha})}
\\
& \lesssim h_{\vero_2''}\| \partial_{x'y} w \|_{L^2(\omega_T,y^{\alpha})},
\end{align*}
which follows from the fact that $y^\alpha \in A_2(\R^+)$, and then \eqref{C_2_element} holds true.
  
The estimate of $w_{\vero_2}(\vero_1)\partial_{x'}\lambda_{\vero_2}$
exploits the fact that the trace of $w$ vanishes on
$\partial_L\C_\Y$; the same happens with 
$w_{\vero_4}(\vero_3) \partial_{x'}\lambda_{\vero_4}$. In fact, we can
write
\begin{align*}
w_{\vero_2}(\vero_1) &= \int_{\omega_{\vero_2}} \int_0^{x'} 
\left( \partial_{x'} w(\tau,y) - \partial_{x'} w(x',y) \right)
\psi_{\vero_2}(x',y) \diff \tau \diff x' \diff y
\\
& + \int_{\omega_{\vero_2}}  \left( \partial_y w(0,y)
- \partial_y w(x',y) \right) y 
\psi_{\vero_2}(x',y) \diff x' \diff y.
\end{align*}

To derive \eqref{C_2_element} we finally proceed as 
in the proofs of Theorem~\ref{TH:v - PivH1}
and \cite[Theorem 3.1]{DL:05}. We omit the details.
\end{proof}

\section{Error estimates}
\label{sec:ErrorEstimate}
The estimates of \S~\ref{sub:sub:H1_estimates} and \S~\ref{sub:sub:H1_boundary_estimates}
are obtained under the local assumption that $w \in H^2(\omega_T,y^{\alpha})$.
However, the solution $\ve$ of \eqref{alpha_harmonic_extension_weak}
satisfies $\ve_{yy} \in L^2(\C,y^\beta)$ only when
$\beta > 2\alpha + 1$, according to Theorem~\ref{TH:regularity}. 
For this reason, in this section we 
derive error estimates for both quasi-uniform and graded
meshes. The estimates of \S~\ref{sub:quasi} for quasi-uniform meshes
are quasi-optimal in terms of regularity but suboptimal in terms of order.
The estimates of \S~\ref{sub:graded} for graded meshes are, instead,
quasi-optimal in both regularity and order. Mesh anisotropy is able
to capture the singular behavior of the solution and restore optimal decay rates.

\subsection{Quasi-uniform meshes}
\label{sub:quasi}
We start with a simple one dimensional case ($n=1$) and assume that we need to approximate
over the interval $[0,\Y]$ the function $w(y) = y^{1-\alpha}$.
Notice that $w_y(y) \approx y^{-\alpha}$ as $y \approx 0^+$ has the same behavior as the 
derivative in the extended direction of the $\alpha$-harmonic extension $\ve$.

Given $M \in \mathbb{N}$ we consider the uniform partition of the interval $[0,\Y]$ 
\begin{equation}
\label{unif_mesh}
  y_k = \frac{k}{M} \Y, \quad k=0,\dots,M.
\end{equation}
and corresponding elements $I_k =[y_k,y_{k+1}]$ of size $h_k = h = \Y/M$ for $k=0,\dots,M-1$.

We can adapt the definition of $\Pi_{\T_{\Y}}$ of \S~\ref{sub:intweighted} to this setting, and
bound the local interpolation errors
$E_k = \|\partial_y( w - \Pi_{\T_{\Y}}w)\|_{L^2(I_k,y^{\alpha})}$.
For $k=2,\ldots,M-1$, since $y \geq h$ and $\alpha < 2\alpha + 1 < \beta$, \eqref{v - PivH1}
implies
% Theorem~\ref{TH:v - PivH1} and the fact that, for all 
% ${\alpha \in (-1,1)\setminus \{0\}}$,
% we have $\alpha - \beta < 0$ imply
\begin{equation}
\label{E_k_u} 
  E_k^2 \lesssim h^2 \int_{\omega_{I_k}} y^{\alpha} |w_{yy}|^2 \diff y
  \lesssim h^{2 + \alpha -\beta} \int_{\omega_{I_k}} y^{\beta} |w_{yy}|^2 \diff y,
\end{equation}
because $\left(\frac{y}{h}\right)^\alpha \le \left(\frac{y}{h}\right)^\beta.$
The estimate for $E_0^2+E_1^2$ follows from from the stability
of the operator $\Pi_{\T_{\Y}}$ \eqref{dyPIcontinuous} and \eqref{stab-boundy}:
\begin{equation}
\label{E_1_u}
  E_0^2+ E_1^2 \lesssim \int_0^{3h} y^\alpha |w_y|^2 \lesssim h^{1-\alpha},
\end{equation}
because $w(y)\approx y^{-\alpha}$ as $y\approx 0^+$.
Using \eqref{E_k_u} and \eqref{E_1_u} 
in conjunction with $2+\alpha-\beta<1-\alpha$, we obtain a global interpolation estimate
\begin{equation}
\| \partial_{y}(w -\Pi_{\T_{\Y}}w ) \|_{L^2((0,\Y),y^{\alpha})} 
\lesssim h^{(2+\alpha-\beta)/2}.
\end{equation}

These ideas can be extended to prove an error estimate for $\ve$ on uniform meshes.

\begin{theorem}[Error estimate for quasi-uniform meshes]
\label{thm:quasiuniform}
Let $\ve$ solve \eqref{alpha_harmonic_extension_weak}, and $V_{\T_{\Y}}$ be the 
solution of \eqref{harmonic_extension_weak}, constructed over a quasi-uniform mesh
of size $h$. If $\Y \approx |\log h |$, then for all $\epsilon >0$
\begin{equation}
  \label{sub_optimal}
   \| \nabla(\ve - V_{\T_{\Y}} ) \|_{L^2(\C_{\Y},y^{\alpha})}  
   \lesssim { h^{s-\varepsilon} }\| f \|_{\Ws}.
\end{equation}
where the hidden constant blows up if $\varepsilon$ tends to $0$.
\end{theorem}
\begin{proof}
Use first Theorem~\ref{Thm:v-v^T} and Theorem~\ref{TH:upper_bound_1},
combined with \eqref{estimate_ve0}, to reduce the approximation
error to the interpolation error of $\ve$. Repeat next the steps leading to \eqref{E_k_u}--\eqref{E_1_u}, but
combining the interpolation estimates  of Theorems \ref{TH:v - PivH1} and \ref{TH:boundary_estimates} with
the regularity results of Theorem~\ref{TH:regularity}.
\end{proof}

\begin{remark}[Sharpness of \eqref{sub_optimal} for $s\ne\frac12$]
\label{optimal-reg} \rm
According to \eqref{asymptoticve} and \eqref{reginx}, $\partial_y\ve\approx y^{-\alpha}$, 
and this formally implies $\partial_y \ve \in H^{s-\varepsilon}(\C,y^{\alpha})$ for
all $\varepsilon>0$ provided $f\in\Ws$. In this sense
\eqref{sub_optimal} appears to be sharp with respect to regularity
even though it does not exhibit the optimal rate. We verify this
argument via a simple numerical illustration for dimension $n=1$. 
We let $\Omega = (0,1)$, $s=0.2$, 
right hand side $f = \pi^{2s} \sin(\pi x)$, and note that
$u(x)=\sin(\pi x)$, and the solution $\ve$ to \eqref{alpha_harm_intro}
is
\[
   \ve(x,y) = \frac{2^{1-s}\pi^{s}}{\Gamma(s)} \sin (\pi x) K_{s}(\pi y).
\]
Figure~\ref{fig:0.2_uniform} shows the rate of convergence for the 
$H^1(\C_\Y,y^{\alpha})$-seminorm.
Estimate \eqref{sub_optimal} predicts a rate of $h^{-0.2-\varepsilon}$.
We point out that for the $\alpha$-harmonic extension we are solving a two dimensional
problem and, since the mesh $\T_{\Y}$ is quasi-uniform, $\# \T_{\Y} \approx h^{-2}$. 
In other words the rate of convergence, when measured in term of degrees of freedom, 
is $(\# \T_{\Y})^{-0.1-\varepsilon}$,
which is what Figure~\ref{fig:0.2_uniform} displays.
\end{remark}

\begin{remark}[Case $s = \srn$]
\rm Estimate \eqref{sub_optimal} does not hold for $s=\srn$.
In this case there is no weight and the scaling issues in \eqref{E_k_u} are no longer present,
so that $E_k \lesssim h \| v \|_{H^2(I_k)}$. We thus obtain the optimal error
estimate
\[
   \| \nabla(\ve - V_{\T_{\Y}} ) \|_{L^2(\C_{\Y})}  
   \lesssim h\| f \|_{H^{1/2}_{00}(\Omega)}.
\]
\end{remark}

\begin{figure}
\label{fig:0.2_uniform}
  \centering
  \includegraphics[width=0.6\textwidth]{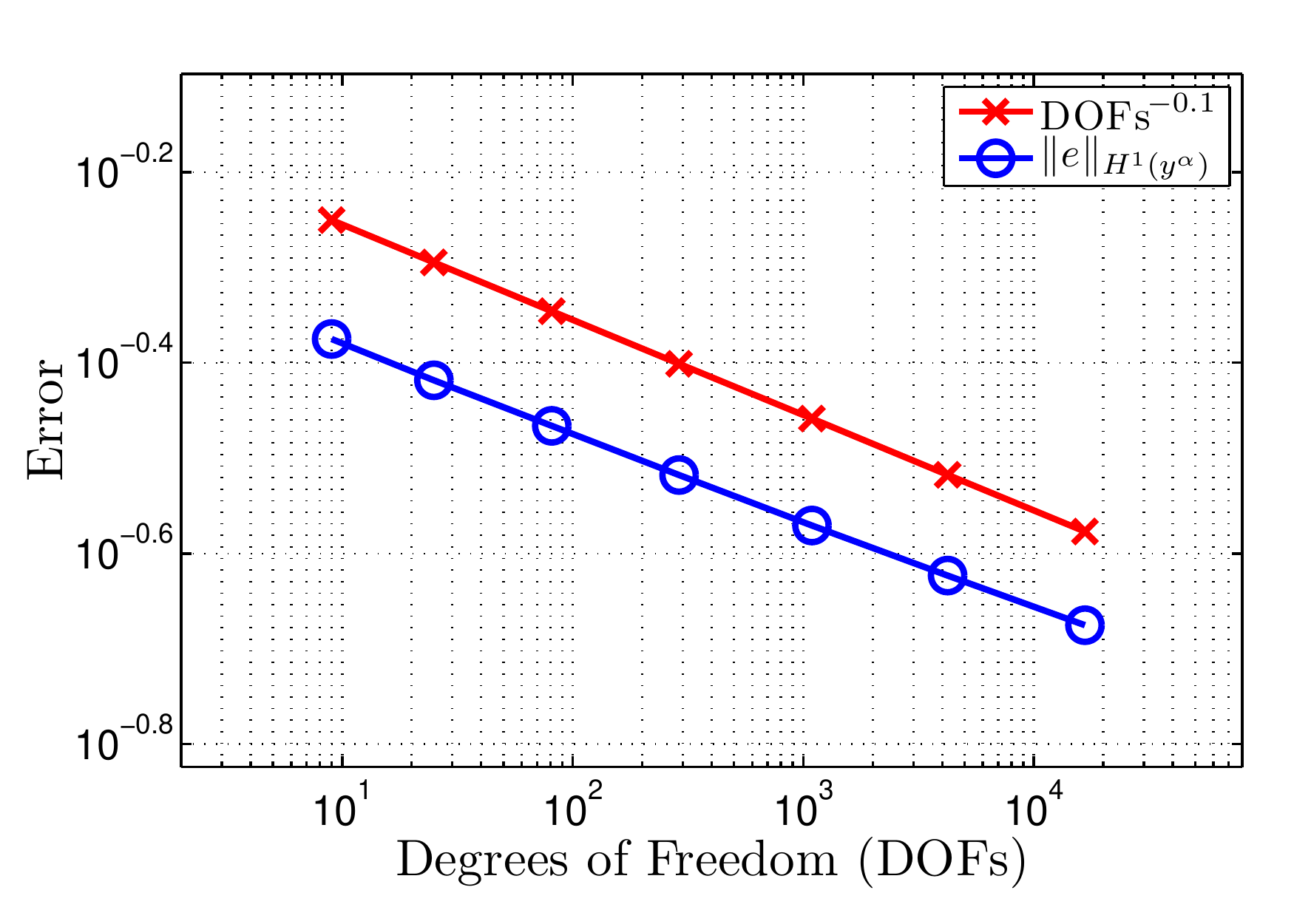}
  \caption{Computational rate of convergence for quasi-uniform meshes, $s=0.2$, and $n=1$.}
\end{figure}

\subsection{Graded meshes}
\label{sub:graded}
The estimate \eqref{sub_optimal} can be written equivalently
\[
   \| \nabla(\ve - V_{\T_{\Y}} ) \|_{L^2(\C_{\Y},y^{\alpha})}  
   \lesssim (\#\T_\Y)^{-\frac{s-\varepsilon}{n+1}} \| f \|_{\Ws},
\]
for quasi-uniform meshes in dimension $n+1$. We now show how to 
compensate the singular behavior in the extended variable $y$ by 
anisotropic meshes and restore the optimal convergence rate
$-1/(n+1)$.

As in \S~\ref{sub:quasi} we start the discussion in dimension $n=1$ with the function
$w(y) = y^{1-\alpha}$ over $[0,\Y]$.
We consider the graded partition $\T_\Y$ of the interval $[0,\Y]$
\begin{equation}
\label{graded_mesh}
  y_k = \left( \frac{k}{M}\right)^{\gamma} \Y, \quad k=0,\dots,M,
\end{equation}
where $\gamma = \gamma(\alpha) > 3/(1-\alpha)> 1$. If we denote by $h_k$ the length of the
interval
\[
  I_{k}=[y_{k},y_{k+1}] =  
    \left[\left( \frac{k}{M}\right)^{\gamma}\Y, \left( \frac{k+1}{M}\right)^{\gamma}\Y \right],
\]
then
\[
  h_k = y_{k+1} - y_{k}  \lesssim \frac{\Y}{M^{\gamma}}k^{\gamma-1}, \quad k=1,\dots, M-1.
\]
We again consider the operator $\Pi_{\T_{\Y}}$ of
\S~\ref{sub:intweighted} on the one dimensional mesh
$\T_\Y$ and wish to bound the local 
interpolation errors $E_k$ of \S~\ref{sub:quasi}.
We apply estimate \eqref{v - PivH1} to interior elements
to obtain that, for $k=2,\dots, M-1$,
\begin{equation}
\label{gm1}
  \begin{aligned}
  E_k^2 &\lesssim h_k^2\int_{\omega_{I_k}} y^{\alpha} |w_{yy}|^2 \diff y
    \lesssim \Y^2\frac{k^{2(\gamma-1)}}{M^{2\gamma}}\int_{\omega_{I_k}} y^{\alpha} |w_{yy}|^2 \diff y \\
  &\lesssim \Y^{2+\alpha-\beta} \frac{k^{2(\gamma-1)}}{M^{2\gamma}}
    \left( \frac{k}{M} \right)^{\gamma(\alpha - \beta)}
    \int_{\omega_{I_k}} y^{\beta} |w_{yy}|^2 \diff y
    \lesssim  \Y^{1-\alpha} \frac{k^{{\gamma(1-\alpha)-3}}}{M^{{\gamma(1-\alpha)}}}.
  \end{aligned}
\end{equation}
because $y^\alpha \lesssim
\left(\frac{k}{M}\right)^{\gamma(\alpha-\beta)} \Y^{\alpha-\beta} y^\beta$.
Adding \eqref{gm1} over $k = 2,\dots,M-1$, and using that
$\gamma(1-\alpha)>3$, we arrive at
\begin{equation}
\label{gm2}
  \| \partial_y(w - \Pi_{\T_\Y} w)\|^2_{L^2\left( (y_2,\Y),y^\alpha \right)}
  \lesssim \Y^{1-\alpha}M^{-2}.
\end{equation} 
For the errors $E_0^2,E_1^2$ we resort to the stability bounds
\eqref{dyPIcontinuous} and \eqref{stab-boundy} to write
\begin{equation}
\label{gm3}
  \| \partial_y (w-\Pi_{\T_\Y}w) \|^2_{L^2\left((0,y_3),y^{\alpha}\right)} 
  \lesssim \int_0^{\left(\frac{3}{M}\right)^\gamma \Y} y^{-\alpha} \diff y
  \lesssim \frac{\Y^{1-\alpha}}{M^{\gamma(1-\alpha)}},
\end{equation}
where we have used \eqref{graded_mesh}. Finally, adding 
\eqref{gm2} and \eqref{gm3} gives
\[
  \|\partial_{y}(w-\Pi_{\T_{\Y}}w) \|^2_{L^2((0,\Y),y^{\alpha})}
  \lesssim \Y^{1-\alpha} M^{-2}, 
\]
and shows that the interpolation error exhibits optimal decay rate.

We now apply this idea to the numerical solution
of problem \eqref{alpha_harmonic_extension_weak_T}.
We assume $\T_\Omega$ to be quasi-uniform in $\Tr_\Omega$ with $\# \T_\Omega \approx M^n$ and 
construct $\T_\Y \in \Tr$ as the tensor product of $\T_\Omega$ and the partition 
given in \eqref{graded_mesh}, with $\gamma > 3/(1-\alpha)$.
Consequently, $\#\T_\Y = M \cdot \# \T_\Omega \approx M^{n+1}$. 
Finally, we notice that since $\T_\Omega$ is shape regular and quasi-uniform,
$h_{\T_{\Omega}} \approx (\# \T_{\Omega})^{-1/n} \approx M^{-1}$.

\begin{theorem}[Error estimate for graded meshes]
\label{TH:error_estimates}
Let $V_{\T} \in \V(\T_\Y)$ solve \eqref{harmonic_extension_weak} and
$U_{\T_{\Omega}} \in \U(\T_{\Omega})$ be defined as in \eqref{eq:defofU}.
Then
\begin{equation}
\label{uboundVandU1}
  \| \ve - V_{\T_{\Y}} \|_{\HLn(\C,y^\alpha)} \lesssim 
  e^{-\sqrt{\lambda_1}\Y/4} \|f \|_{\Hs'} + \Y^{(1-\alpha)/2} (\# \T_{\Y})^{-1/(n+1)} \|f \|_{\Ws},
\end{equation} 
\end{theorem}
\begin{proof}
In light of \eqref{estimate_ve0}, with $\epsilon\approx e^{-\sqrt{\lambda_1}\Y/4}$,
it suffices to bound the interpolation
error $\ve - \Pi_{\T_{\Y}} \ve$ on the mesh $\T_\Y$.
To do so we, first of all, notice that if 
$I_1$ and $I_2$ are neighboring cells on the partition of $[0,\Y]$, 
then there is a constant $\sigma=\sigma(\gamma)$ such that
$h_{I_1}\le \sigma h_{I_2}$, whence the weak regularity condition $(c)$ holds.
We can thus apply the polynomial interpolation theory of \S~\ref{sub:intweighted}.
We decompose the mesh $\T_\Y$ into the sets
\begin{align*}
  \Tm_0 := \left\{ T \in \T_\Y: \ \omega_T \cap (\bar\Omega \times \{0\} ) 
    = \emptyset \right\}, \quad
  \Tm_1 := \left\{ T \in \T_\Y: \ \omega_T \cap (\bar\Omega \times \{0\} )
    \neq \emptyset \right\}.
\end{align*}
We observe that for all $T=K\times I_k\in\Tm_0$ we have $k\ge 2$ and
$y^\alpha\lesssim \left(\frac{k}{M}\right)^{\gamma(\alpha-\beta)}
\Y^{\alpha-\beta} y^\beta$.
Applying Theorem~\ref{TH:v - PivH1} and Theorem~\ref{TH:boundary_estimates}
to elements in $\Tm_0$ we obtain
\begin{multline*}
  \sum_{ T \in \Tm_0 } \| \nabla( \ve - \Pi_{\T_{\Y}} \ve) \|_{L^2(T,y^{\alpha})}^2 \lesssim
  \sum_{ T = K \times I \in \Tm_0 }
  \left(
    h_K^2 \| \nabla_{x'}\nabla \ve \|_{L^2(\omega_T,y^\alpha)}^2
  \right. \\
  \left.
  + h_I^2 \|\partial_y \nabla_{x'} \ve \|_{L^2(\omega_T,y^\alpha)}^2
  + h_I^2 \|\partial_{yy} \ve \|_{L^2(\omega_T,y^\beta)}^2 \right) = S_1 + S_2 + S_3.
\end{multline*}
We examine first the most problematic third term $S_3$, which we rewrite as
follows:
\[
S_3 \lesssim \sum_{k=2}^M \Y^{2+\alpha-\beta}\frac{k^{2(\gamma-1)}}{M^{2\gamma}}
    \left( \frac{k}{M} \right)^{\gamma(\alpha - \beta)} \int_{a_k}^{b_k} y^\beta 
    \int_\Omega |\partial_{yy}\ve|^2 \diff x' \diff y,
\]
with $a_k=\left(\frac{k-1}{M}\right)^\gamma \Y$ and $b_k=\left(\frac{k+1}{M}\right)^\gamma \Y$.
We now invoke the local estimate \eqref{local_beta}, as well as the fact
that $b_k-a_k \lesssim \left(\frac{k}{M}\right)^{\gamma-1} \frac{\Y}{M}$,
to end up with
\[
  S_3 \lesssim \sum_{k=2}^M
  \Y^{1-\alpha} \frac{k^{\gamma(1-\alpha)-3}}{M^{\gamma(1-\alpha)}}\|f\|_{L^2(\Omega)}^2
  \lesssim \Y^{1-\alpha} M^{-2} \|f\|_{L^2(\Omega)}^2.
\]
We now handle the middle term $S_2$ with the help of \eqref{local_alpha},
which is valid for $b_k\le1$. This imposes the restriction 
$k\le k_0\le M\Y^{-1/\gamma}$, whereas for $k>k_0$ we know that the
estimate decays exponentially. We thus have
\[
  S_2 \lesssim \|f\|_{\Ws}^2 \sum_{k=2}^{k_0}
  \left(\left(\frac{k}{M}\right)^{\gamma-1} \frac{\Y}{M} \right)^3
  \lesssim \frac{\Y^{2/\gamma}}{M^2} \|f\|_{\Ws}^2 
  \lesssim \frac{\Y^{1-\alpha}}{M^2} \|f\|_{\Ws}^2.
\]
The first term $S_1$ is easy to estimate. Since $h_K\le M^{-1}$ for all $K\in\T_\Omega$, we get
\begin{equation*}
 S_1 \lesssim M^{-2} 
  \| \nabla_{x'}\nabla v \|_{L^2(\C_\Y,y^\alpha)}^2 
  \lesssim M^{-2} \|f\|_{\Ws}^2 \lesssim \Y^{1-\alpha} M^{-2} \|f\|_{\Ws}^2.
\end{equation*}

For elements in $\Tm_1$, we rely on the stability estimates
\eqref{dxPIcontinuous}, \eqref{dyPIcontinuous}, \eqref{stab-boundx'} and \eqref{stab-boundy}
of $\Pi_{\T_\Y}$ and thus repeat the arguments used to derive 
\eqref{gm2} and \eqref{gm3}. Adding the estimates for $\Tm_0$ 
and $\Tm_1$ we obtain the assertion.
\end{proof}

\begin{remark}[Choice of $\Y$]
\label{re:ve_quasi_optimal} \rm
A natural choice of $\Y$ comes from equilibrating the two terms on the 
right-hand side of \eqref{uboundVandU1}: 
\[
 \epsilon \approx \# (\T_\Y)^{-\frac1{n+1}}
 \Leftrightarrow
 \Y \approx \log(\# (\T_\Y)).
\]
This implies the near-optimal estimate
\begin{equation}
\label{ve_quasi_optimal}
 \| \ve - V_{\T_{\Y}} \|_{\HLn(\C,y^\alpha)} \lesssim 
 |\log(\# \T_\Y )|^s \cdot (\# \T_{\Y})^{-1/(n+1)} \|f \|_{\Ws}.
\end{equation}
\end{remark}

\begin{remark}[Estimate for $u$] \rm
In view of \eqref{upper_bound_U}, we deduce the energy estimate
\[
  \| u - U_{\T_\Omega} \|_{\Hs} \lesssim 
  \left| \log(\# \T_\Y ) \right|^s\cdot (\# {\T_{\Y}})^{-1/(n+1)} \|f \|_{\Ws}.
\]
We can rewrite this estimate in terms of regularity 
$u\in\mathbb{H}^{1+s}(\Omega)$ and $\#\T_\Omega$ as
\[
  \| u - U_{\T_\Omega} \|_{\Hs} \lesssim  
  \left|\log(\# \T_\Omega) \right|^s \cdot (\# \T_{\Omega})^{-1/n} 
  \|u \|_{\mathbb{H}^{1+s}(\Omega)}.
\]
and realize that the order is near-optimal given the regularity shift
from left to right. However, our PDE approach does not allow for a
larger rate $(\#\T_\Omega)^{(2-s)/n}$ that would still be compatible 
with piecewise bilinear polynomials but not with \eqref{ve_quasi_optimal}.
\end{remark}

\begin{remark}[Computational complexity]
\rm The cost of solving the discrete problem
\eqref{harmonic_extension_weak} is related to $\#\T_\Y$, and not to
$\#\T_\Omega$, but the resulting system is sparse. The structure of 
\eqref{harmonic_extension_weak} is so that fast multilevel solvers can be
designed with complexity proportional to $\#\T_\Y$. 
On the other hand, using an integral formulation requires 
sparsification of an otherwise dense matrix with associated cost 
$(\#\T_\Omega)^2$.
\end{remark}

\begin{remark}[Fractional regularity]
\rm The function $\ve$, solution the $\alpha$-harmonic extension problem,
may also have singularities in the direction of the $x'$-variables
and thus exhibit fractional regularity. This depends
on $\Omega$ and the right hand side $f$ 
(see Remark \ref{R:compatibitity}). The characterization of
such singularities is as yet an open problem to us. The polynomial interpolation  
theory developed in \S~\ref{sub:intweighted}, however, applies to
shape-regular but graded mesh $\T_\Omega$, which can resolve such
singularities, provided we maintain the Cartesian structure of
$\T_\Y$. The corresponding a posteriori error analysis is an entirely
different but important direction currently under investigation.
\end{remark}

\section{Numerical experiments for the fractional Laplacian}
\label{sec:numerics}
To illustrate the proposed techniques here we present a couple of numerical examples.
The implementation has been carried out with the help of the \texttt{deal.II} library
(see \cite{dealii,dealii2}) which, by design, is based on tensor product elements and thus
is perfectly suitable for our needs. The main concern while developing the code was correctness
and, therefore, integrals are evaluated numerically with Gaussian quadratures of sufficiently
high order and linear systems are solved using CG with ILU preconditioner with the exit
criterion being that the $\ell^2$-norm of the residual is less than $10^{-12}$.
More efficient techniques for quadrature and preconditioning are currently under investigation. 

\subsection{A square domain}
Let $\Omega = (0,1)^2$. It is common knowledge that
\[
  \varphi_{m,n}(x_1,x_2) = \sin(m \pi x_1)\sin(n \pi x_2),
  \quad
  \lambda_{m,n} = \pi^2 \left( m^2 + n^2 \right),
  \qquad m,n \in \mathbb{N}.
\]
If $f(x_1,x_2) = ( 2\pi^2)^{s} \sin(\pi x_1)\sin(\pi x_2)$, by \eqref{def:s_in_O} we have
\[
  u(x_1,x_2) = \sin(\pi x_1)\sin(\pi x_2),
\]
and, by \eqref{exactforms},
\[
  \ve(x_1,x_2,y) = \frac{2^{1-s}}{\Gamma(s)}(2\pi^2)^{s/2}
     \sin(\pi x_1)\sin(\pi x_2) y^{s}K_s(\sqrt{2}\pi y).
\] 

We construct a sequence of  meshes $\{\T_{\Y_k} \}_{k\geq1}$,
where the triangulation of $\Omega$ is obtained by uniform refinement and the partition
of $[0,\Y_k]$ is as in \S~\ref{sub:graded}, \ie $[0,\Y_k]$ is 
divided with mesh points given by \eqref{graded_mesh} with the
election of the parameter $\gamma > 3/(1-\alpha)$.
On the basis of Theorem~\ref{Thm:v-v^T}, for each mesh the truncation parameter $\Y_k$ is chosen so that 
$\epsilon \approx (\# \T_{\Y_{k-1}})^{-1/3}$. This can be achieved, for instance, by setting
\[
  \Y_k \geq \Y_{0,k} = \frac{2}{\sqrt{\lambda_1}} ( \log C - \log \epsilon ).
\] 
With this type of meshes,
\[
  \| u - U_{\T_{\Omega,k}} \|_{\Hs} \lesssim 
  \| \ve -V_{\T_{\Y_k}} \|_{\HLn(\C,y^{\alpha})} \lesssim 
  {|\log( \# \T_{\Y_k} )|^s} \cdot (\# \T_{\Y_k})^{-1/3},
\]
which is near-optimal in $\ve$ but suboptimal in $u$, since we should expect
(see \cite{BrennerScott})
\[
  \| u - U_{\T_{\Omega,k}}  \|_{\Hs} \lesssim h_{\T_\Omega}^{2-s} \lesssim (\# \T_{\Y_k} )^{-(2-s)/3}.
\] 

\begin{figure}
\label{fig:02}
\begin{center}$
\begin{array}{cc}
\includegraphics[width=2.42in]{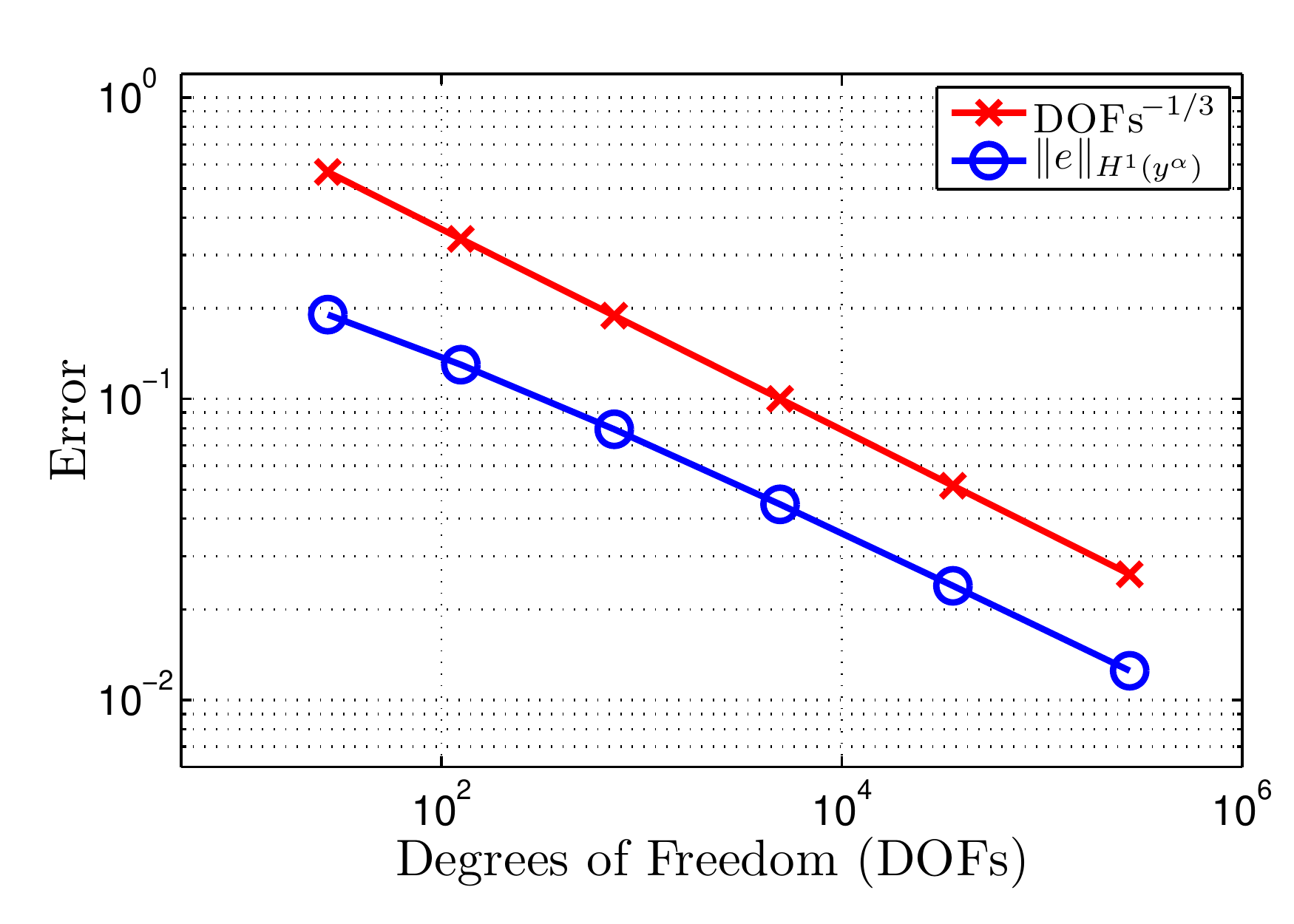} &
\includegraphics[width=2.42in]{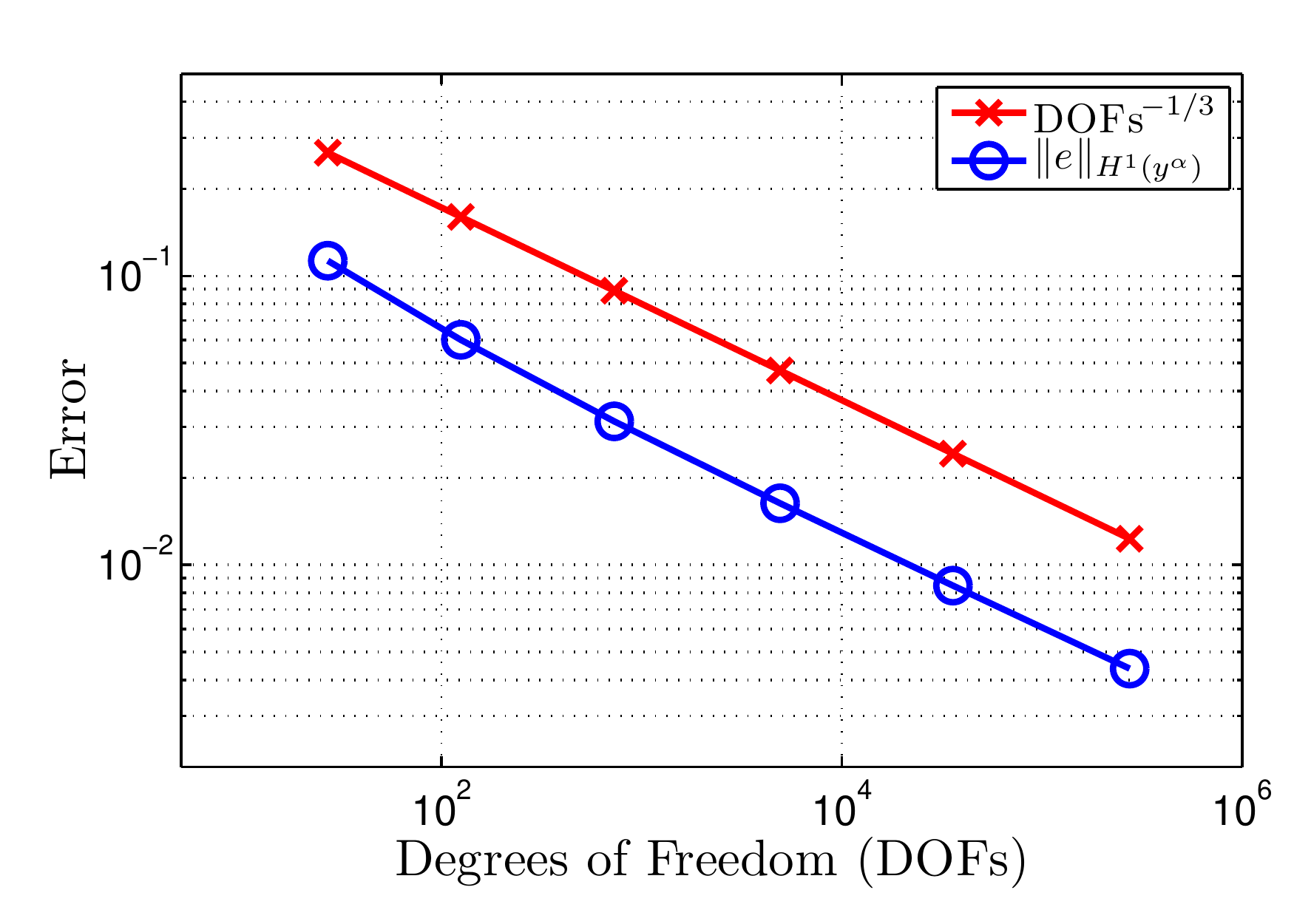}
\end{array}$
\end{center}
\caption{Computational rate of convergence for the approximate solution of the fractional Laplacian over 
a square with
graded meshes on the extended dimension. The left panel shows the rate for $s=0.2$ and the right one 
for $s=0.8$. In both cases, the rate is $\approx (\# \T_{\Y_k} )^{-1/3}$ in agreement
with Theorem \ref{TH:error_estimates} and Remark \ref{re:ve_quasi_optimal}}
\end{figure}

Figure~\ref{fig:02} shows the rates of convergence for $s=0.2$ and $s=0.8$ respectively.
In both cases, we obtain the rate given by Theorem~\ref{TH:error_estimates} and 
Remark~\ref{re:ve_quasi_optimal}.

\subsection{A circular domain}
Let $\Omega = \{ |x'| \in \R^2 : |x'|<1 \}$. Using polar coordinates it can be shown that
\begin{equation}
\label{circularphi}
 \varphi_{m,n}(r,\theta) = 
 J_{m}(j_{m,n}r)
\left( A_{m,n} \cos (m \theta) + B_{m,n} \sin(m \theta)\right), 
\end{equation}
where $J_m$ is the $m$-th Bessel function of the first kind;
$j_{m,n}$ is the $n$-th zero of $J_m$ and
$A_{m,n}$, $B_{m,n}$ are real normalization constants that ensure
$\| \varphi_{m,n} \|_{L^2(\Omega)}=1$ for all $m,n \in \mathbb{N}$.
It is also possible to show that $\lambda_{m,n}=\left(j_{m,n} \right)^2$.

If $f = ( \lambda_{1,1})^{s} \varphi_{1,1}$, then \eqref{def:s_in_O} and \eqref{exactforms} show that
$u = \varphi_{1,1}$ and
\[
  \ve(r,\theta,y) = \frac{2^{1-s}}{\Gamma(s)}(\lambda_{1,1})^{s/2}
         \varphi_{1,1}(r,\theta) y^{s}K_s(\sqrt{2}\pi y).
\]
From \cite[Chapter 9]{Abra}, we have that $j_{1,1} \approx 3.8317$.

We construct a sequence of  meshes $\{\T_{\Y_k} \}_{k\geq1}$,
where the triangulation of $\Omega$ is obtained by quasi-uniform refinement and the partition
of $[0,\Y_k]$ is as in \S~\ref{sub:graded}. The parameter 
$\Y_k$ is chosen so that $\epsilon \approx (\# \T_{\Y_{k-1}})^{-1/3}$.
With these meshes
\begin{equation}
\label{numerical_experiment_2_ve}
  \| \ve -V_{\T_{\Y_k}} \|_{\HLn(\C,y^{\alpha})} \lesssim 
  |\log(\# \T_{\Y_k})|^s \cdot (\# \T_{\Y_k})^{-1/3},
\end{equation}
which is near-optimal. 

Figure~\ref{fig:03} shows the errors of $\|\ve - V_{\T_{k,\Y}} \|_{H^1(y^{\alpha},\C_{\Y_{k}})}$
for $s = 0.3$ and $s = 0.7$. The results, again, are in agreement with
Theorem~\ref{TH:error_estimates} and Remark~\ref{re:ve_quasi_optimal}.

\begin{figure}
\label{fig:03}
\begin{center}$
\begin{array}{cc}
\includegraphics[width=2.42in]{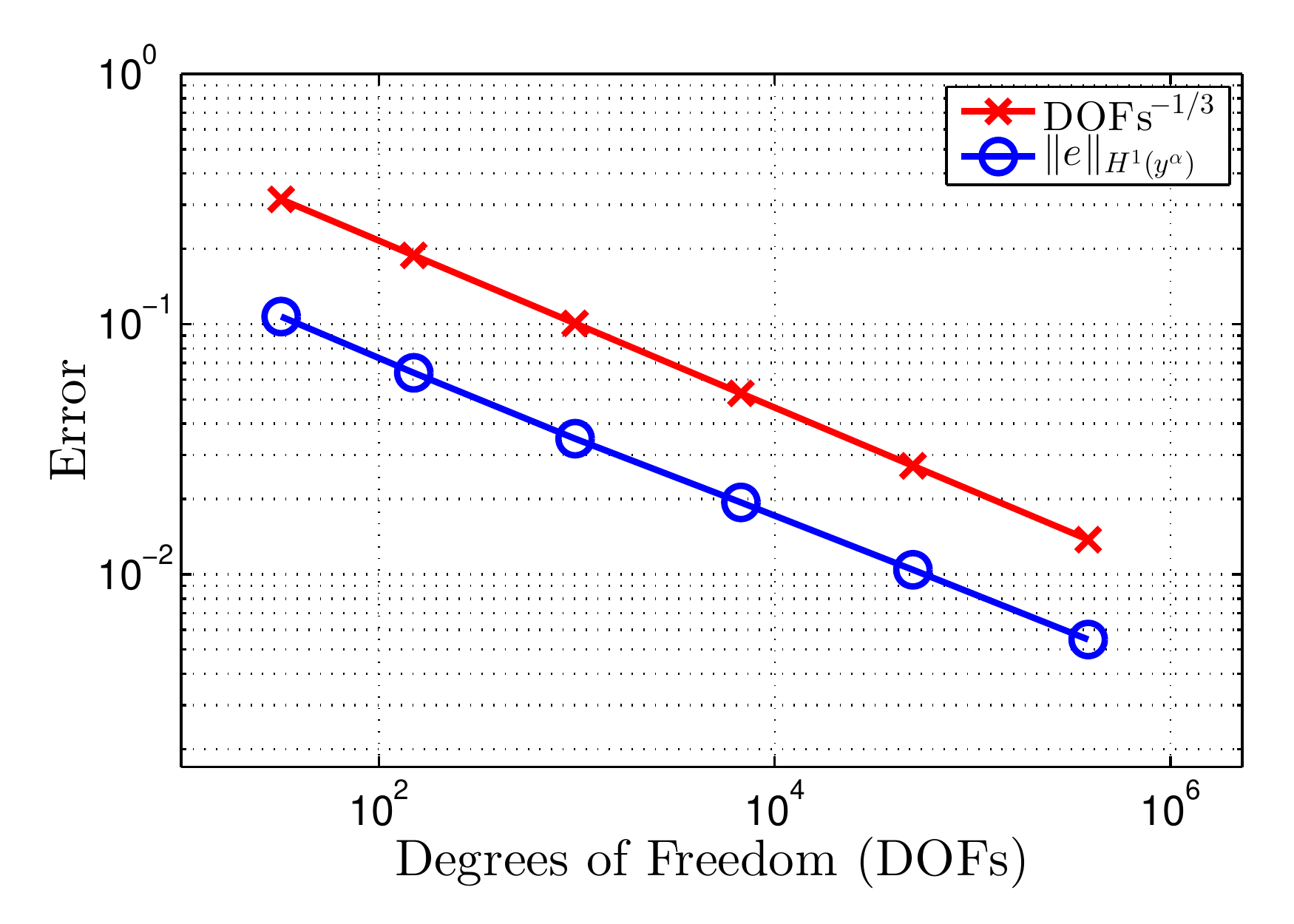} &
\includegraphics[width=2.42in]{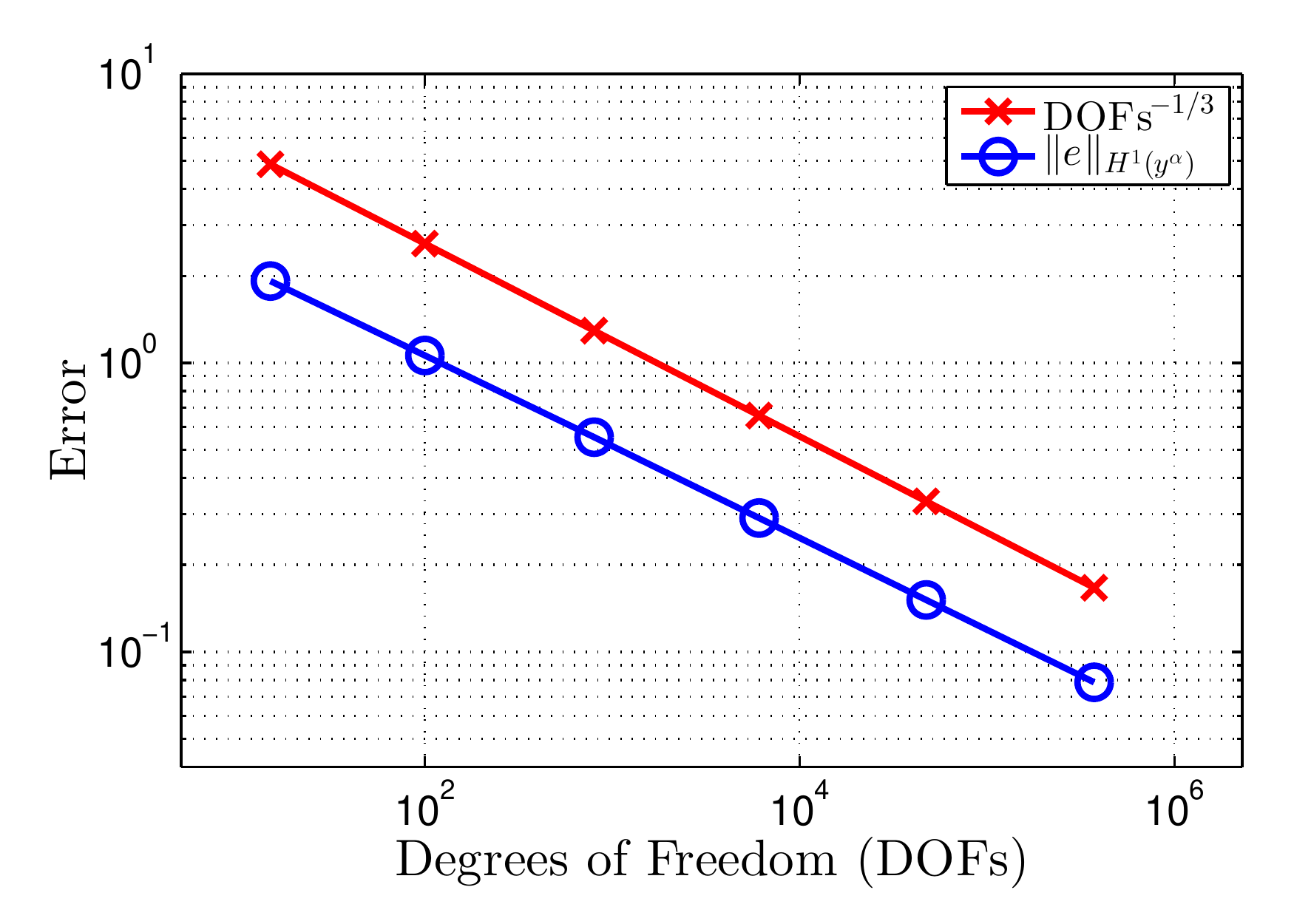}
\end{array}$
\end{center}
\caption{Computational rate of convergence for the approximate solution of the fractional Laplacian 
over a circle with
graded meshes on the extended dimension. The left panel shows the rate for $s=0.3$ and the right one 
for $s=0.7$. In both cases, the rate is $\approx (\# \T_{\Y_k} )^{-1/3}$ in agreement
with Theorem \ref{TH:error_estimates} and Remark \ref{re:ve_quasi_optimal}}
\end{figure}

\section{Fractional powers of general second order elliptic operators}
\label{sec:general}
Let us now discuss how the methodology developed in previous sections extends to a general
second order, symmetric and uniformly elliptic operator. 
This is an important property of our PDE approach. 
Recall that, in \S~\ref{sub:CaffarelliSilvestre}, we discussed how
the fractional Laplace operator can be realized as a Dirichlet to Neumann map via an extension
problem posed on the semi-infinite cylinder $\C$. 
In the work of Stinga and Torrea \cite{ST:10}, the same type of characterization
has been developed for the fractional powers of second order elliptic operators.

Let $\mathcal{L}$ be a second order symmetric differential operator of the form
\begin{equation}
\label{second_order}
 \mathcal{L} w = - \DIV_{x'} (A \nabla_{x'} w ) + c w,
\end{equation}
where $c \in L^\infty(\Omega)$ with $c\geq 0$ almost everywhere,
$A \in \C^{0,1}(\Omega,\GL(n,\R))$ is symmetric and positive definite,
and $\Omega$ is Lipschitz.
Given $f \in L^2(\Omega)$, the Lax-Milgram lemma shows that there is a unique $w \in H^1_0(\Omega)$
that solves
\[
 \mathcal{L} w = f \text{ in } \Omega, \qquad w = 0 \text{ on } \partial \Omega.
\]
In addition, if $\Omega$ has a $\C^{1,1}$ boundary, \cite[Theorem 2.4.2.6]{Grisvard} 
shows that $w\in H^2(\Omega) \cap H_0^1(\Omega)$.
Since $\mathcal{L}^{-1}: L^2(\Omega)\to L^2(\Omega)$ is compact and symmetric,
its spectrum 
is discrete, positive and accumulates at zero. Moreover, there exists
$\{ \lambda_k,\varphi_k \}_{k\in \mathbb N} \subset \R_+\times H^1_0(\Omega)$
such that $\{\varphi_k \}_{k\in \mathbb N}$ is an orthonormal basis of $L^2(\Omega)$ and
for, $k\in\mathbb N$,
\begin{equation}
  \label{eigenvalue_problem_L}
    \mathcal{L} \varphi_k = \lambda_k \varphi_k  \text{ in } \Omega,
    \qquad
    \varphi_k = 0 \text{ on } \partial\Omega,
\end{equation}
and $\lambda_k\to \infty$ as $k\to\infty$.
For $u \in C_0^{\infty}(\Omega)$ we then define the fractional powers of $\mathcal{L}$ as
\begin{equation}
  \label{def:second_frac}
  \mathcal{L}^s u = \sum_{k=1}^\infty u_k \lambda_k^{s}\varphi_k,
\end{equation} 
where $u_k = \int_{\Omega} u \varphi_k $. By density the operator $\mathcal{L}^s$ can be extended 
again to $\Hs$.
This discussion shows that it is legitimate to study the following problem: given
$s\in (0,1)$ and $f \in \Hs'$, find $u \in \Hs$ such that
\begin{equation}
\label{gl=f_bdddom}
  \mathcal{L}^s u= f \text{ in } \Omega.
\end{equation}

To realize the operator $\mathcal{L}^s$ as the Dirichlet to Neumann map of an extension problem 
we use the generalization of the result by Caffarelli and Silvestre presented in \cite{ST:10}.
We seek a function $\ve: \C \rightarrow \R$ that solves
\begin{equation}
\label{alpha_harm_L}
\begin{dcases}
  -\mathcal{L}\ve + \frac{\alpha}{y}\partial_{y}\ve + \partial_{yy}\ve = 0, & \text{in } \C, \\
  \ve = 0, & \text{on } \partial_L \C, \\
  \frac{ \partial \ve }{\partial \nu^\alpha} = d_s f, & \text{on } \Omega \times \{0\}, \\
\end{dcases}
\end{equation}
where the constant $d_s$ is as in \eqref{d_s}. In complete analogy to \S~\ref{sub:CaffarelliSilvestre} it
is possible to show that
\[
  d_s \mathcal{L}^s u = \frac{\partial \ve}{\partial \nu^{\alpha}}: \Hs \longmapsto \Hs'.
\]
Notice that the differential operator in \eqref{alpha_harm_L} is
\[
  \DIV \left( y^{\alpha} \mathbf{A} \nabla \ve \right) + y^{\alpha} c\ve,
\]
where, for all $x \in \C $, $\mathbf{A}(x) = \diag\{A(x'),1\} \in \GL(n+1,\R)$.

It suffices now to notice that both $y^\alpha c$ and $y^\alpha \mathbf{A}$ are in $A_2(\R^{n+1}_+)$,
to conclude that, given $f \in \Hs'$, there is a unique $\ve \in \HL(\C,y^{\alpha})$
that solves \eqref{alpha_harm_L}, \cite{FKS:82}. In addition, $u = \ve(\cdot,0) \in \Hs$ solves
\eqref{gl=f_bdddom} and we have the stability estimate
\begin{equation}
\label{estimate_s_L}
  \| u \|_{\Hs} \lesssim \| \nabla \ve \|_{L^2(\C,y^{\alpha})} \lesssim \| f\|_{\Hs'},
\end{equation}
where the hidden constants depend on $A$, $c$, $C_{2,y^\alpha}$ and $\Omega$.

The representation \eqref{exactforms} of $\ve$ in terms of the Bessel functions
is still valid. We can thus repeat the arguments in the proof of Theorem~\ref{Thm:v-v^T} to conclude that
\[
 \|\nabla \ve\|_{L^2(\Omega \times (\Y,\infty),y^{\alpha})} \lesssim e^{-\sqrt{\lambda_1} \Y/2}
  \| f\|_{\Hs'},
\]
and introduce $v\in \HL(\C_\Y,y^{\alpha})$ --- solution of a truncated version of
\eqref{alpha_harm_L} --- and show that
\begin{equation}
\label{le:v-v^TCL}
  \| \nabla(\ve - v) \|_{L^2(\C, y^{\alpha})} \lesssim e^{-\sqrt{\lambda_1} \Y/4} \| f\|_{\Hs'}.
\end{equation}

Next, we define the finite element approximation of the solution to \eqref{alpha_harm_L} as
the unique function $V_{\T_{\Y}} \in \V(\T_{\Y})$ that solves
\begin{equation}
\label{harmonic_extension_weak_L}
  \int_{\C_\Y} {y^{\alpha}\mathbf{A}(x)} \nabla V_{\T_{\Y}}\cdot \nabla W 
 + y^{\alpha} c(x') V_{\T_{\Y}} W \diff x' \diff y
 = d_s \langle f, \textrm{tr}_{\Omega} W \rangle,
  \quad \forall W \in \V(\T_{\Y}).
\end{equation}
We construct, as in \S~\ref{sub:graded}, a shape regular triangulation $\T_\Omega$ of $\Omega$,
which we extend to $\T_\Y \in \Tr$ with the partition given in \eqref{graded_mesh}, with
$\gamma > 3/(1-\alpha)$. Following the proof of Theorem~\ref{TH:error_estimates} we can also show
the following error estimate.

\begin{theorem}[Error estimate for general operators]
Let $V_{\T} \in \V(\T_\Y)$ be the solution of \eqref{harmonic_extension_weak_L} and
$U_{\T_{\Omega}} \in \U(\T_{\Omega})$ be defined as in \eqref{eq:defofU}.
If $\ve$, solution of \eqref{alpha_harm_L}, is such that
$\mathcal{L} \ve,\ \partial_y \nabla \ve \in H^2(y^\alpha, \C)$,
then we have
\[
  \| u - U_{\T_\Omega} \|_{\Hs} \lesssim
  \| \ve - V_{\T_{\Y}} \|_{\HLn(\C,y^\alpha)} \lesssim 
  |\log(\# \T_{\Y}) |^s (\# \T_{\Y})^{-1/(n+1)} \|f \|_{\Ws}.
\]
\end{theorem}

%============ References ==================
\bibliographystyle{siam}
\bibliography{biblio}

\end{document}